\documentclass[12pt]{amsart}

\usepackage[colorlinks=true,urlcolor=blue,
bookmarks=true,bookmarksopen=true,citecolor=blue
]{hyperref}

\usepackage{pdfsync}
\usepackage[all]{xy}
\usepackage{graphicx}

\usepackage{epsfig}
\usepackage{amscd}
\usepackage[mathscr]{eucal}
\usepackage{amssymb}
\usepackage{amsxtra}
\usepackage{amsmath}
\usepackage{latexsym}
\usepackage[all]{xy}
\usepackage{enumerate}
\usepackage{mathrsfs}

\usepackage{color}
\theoremstyle{plain}

\newtheorem{thm}{Theorem}[subsection]

\newtheorem{lem}[thm]{Lemma}
\newtheorem{prop}[thm]{Proposition}

\theoremstyle{definition}
\newtheorem{dfn}[thm]{Definition}

\theoremstyle{remark}
\newtheorem{rem}[thm]{Remark}

\newtheorem{rems}[thm]{Remarks}

\theoremstyle{plain}

\newcommand{\RE}[1]{{{#1}}}
\newcommand{\RB}[1]{{#1}}

\newcommand{\Id}{{{\mathchoice {\rm 1\mskip-4mu l} {\rm 1\mskip-4mu l}
      {\rm 1\mskip-4.5mu l} {\rm 1\mskip-5mu l}}}}

\newcommand{\cobto}{\leadsto}

\newcommand{\R}{\mathbb{R}}
\newcommand{\Z}{\mathbb{Z}}

\newcommand{\C}{\mathbb{C}}

\newcommand{\La}{\Lambda}

\newcommand{\Crit}{\textnormal{Crit\/}}
\newcommand{\cob}{\mathcal{C}ob}
\newcommand{\fuk}{\mathcal{F}uk}

\newcommand{\mor}{{\textnormal{Mor\/}}}

\newcommand{\pbaddress}{biran@math.ethz.ch}
\newcommand{\ocaddress}{cornea@dms.umontreal.ca}

\begin{document}

\title[Lagrangian cobordism I.]{Lagrangian cobordism I.}
\date{\today}

\thanks{The second author was supported by an NSERC Discovery grant
  and a FQRNT Group Research grant}


\author{Paul Biran and Octav Cornea}

\address{Paul Biran, Department of Mathematics, ETH-Z\"{u}rich,
  R\"{a}mistrasse 101, 8092 Z\"{u}rich, Switzerland}
\email{\pbaddress} \address{Octav Cornea, Department of Mathematics
  and Statistics University of Montreal C.P. 6128 Succ.  Centre-Ville
  Montreal, QC H3C 3J7, Canada} \email{\ocaddress}

\bibliographystyle{alphanum}

%


\maketitle

%
%


\tableofcontents

\section{Introduction} \label{s:intro}
Embedded Lagrangian cobordism is a natural notion, initially introduced by
Arnold~\cite{Ar:cob-1, Ar:cob-2} at the beginnings of symplectic
topology. This notion was studied by Eliashberg~\cite{El:cob} and
Audin~\cite{Aud:calc-cob} who showed that, in full generality, this is
a very flexible notion that can be translated to purely algebraic
topological constraints. By contrast, the work of
Chekanov~\cite{Chek:cob} points out a certain form of rigidity valid
in the case of monotone cobordisms.

In this paper we will see that Floer theoretic tools lead to a further
understanding of cobordism. It turns out that, remarkably, Lagrangian
cobordism, in its monotone version, preserves Floer homology and all
similar invariants.

Moreover, Lagrangian cobordism can be structured as a category - there
are actually a number of ways to do this, in particular, one
introduced here as well as a different one introduced independently by
Nadler and Tanaka~\cite{Na-Ta}.

The behavior of the Floer-theoretic invariants with respect to
Lagrangian cobordism, as reflected in our results, translates into
properties of the morphisms in the cobordism category
in~\cite{Bi-Co:cob2}. This strongly suggests that this cobordism
category is related in a functorial way to an appropriate Fukaya
category, roughly in the way topological spaces are related to groups
via the (singular) homology functor.  This is indeed the case and in
the last section of the paper we review this categorical perspective.
The full proof of this functoriality is based on the techniques
introduced in this paper but is postponed to the forthcoming
paper~\cite{Bi-Co:cob2}.

\subsubsection*{Acknowledgments} The first author would like to thank
Dietmar Salamon and Ivan Smith for helpful discussions on Floer theory
and Fukaya categories. The second author thanks Mohammed Abouzaid,
Denis Auroux, Fran\c{c}ois Charette, Yasha Eliashberg, Paul Gauthier
and Cl\'ement Hyvrier for useful discussions as well as the MSRI for
its hospitality during the Fall of 2009, when the work presented here
was initiated. We also thank the referee for useful comments and in
particular for asking critical questions regarding the example
in~\S\ref{subsec:cob-lagr-non-iso} which led to the \RB{correction of an
earlier mistake in the Maslov} index calculation \RB{contained in} that
example.


\section{Main results}\label{sec:main-res}

Here we first fix the setting of the paper, in particular the
definition of Lagrangian cobordisms that we use.  We then list the
main results followed by a few comments.

\subsection{Setting.}
In this paper $(M^{2n},\omega)$ is a fixed connected symplectic
manifold. We assume that $M$ is compact but the constructions
described have immediate adaptations to the case when $M$ is only tame
(see~\cite{ALP}). Lagrangian submanifolds $L^n \subset M^{2n}$ will be
generally assumed to be closed unless otherwise indicated.

\subsubsection{Monotonicity} \label{sb:monotonicity} All families of
Lagrangian submanifolds in our constructions have to satisfy a
monotonicity condition in a uniform way as described below.  This is
crucial for the transversality issues involving bubbling of disks to
be approachable by the methods in \cite{Bi-Co:qrel-long} and
\cite{Bi-Co:rigidity}.

Given a Lagrangian submanifold $L \subset M$ there are two canonical
morphisms
$$\omega : \pi_{2}(M,L)\to \R \ , \ \mu:\pi_{2}(M,L)\to \Z$$
the first given by integration of $\omega$ and the second being the
Maslov index. The Lagrangian $L$ is \emph{monotone} if there exists a
positive constant $\rho>0$ so that for all $\alpha\in \pi_{2}(M,L)$ we
have $\omega(\alpha)=\rho\mu(\alpha)$. Unless otherwise specified we
will always assume in this paper that the minimal Maslov number
$$N_{L}: =\min\{\mu(\alpha) : \alpha\in \pi_{2}(M,L) \ ,\
\omega(\alpha)>0\}$$ satisfies $N_{L}\geq 2$.

In what follows we will use $\mathbb{Z}_2$ as the ground ring.
However, most of the discussion generalizes under additional
assumptions on the Lagrangians to arbitrary rings. We therefore denote
the ground ring by $K$, keeping in mind that in this paper $K =
\mathbb{Z}_2$.

To each connected closed, monotone Lagrangian $L$ there is an
associated basic Gromov- Witten type invariant $d_{L}\in K$ which is
the number (in $K$) of $J$-holomorphic disks of Maslov index $2$ going
through a generic point $P\in L$ for $J$ a generic almost complex
structure that is compatible with $\omega$. (Under different forms
this invariant has appeared in~\cite{Oh:HF1,Oh:HF1-add, Chek:cob,
  FO3:book-vol1}. It has also been used for instance
in~\cite{Bi-Co:rigidity}.)

A family of Lagrangian submanifolds $L_{i}$, $i\in I$, is called
\emph{uniformly monotone} if each $L_{i}$ is monotone and the
following condition is satisfied: there exists $d\in K$ so that for
all $i\in I$ we have $d_{L_{i}}=d$ and, if $d\not=0$, then there
exists a positive real constant $\rho$ so that the monotonicity
constant of $L_{i}$ equals $\rho$ for all $i\in I$.

In the absence of other indications, all the Lagrangians $L$ used in
the paper will be assumed monotone \RE{with $N_{L}\geq 2$} and,
similarly, the Lagrangian families will be assumed uniformly monotone.

To fix notation, for $d\in K$ and $\rho\in [0,\infty)$, we consider
the family $\mathcal{L}_{d}(M)$ formed by the closed, connected
Lagrangian submanifolds $L\subset M$ that are monotone with
monotonicity constant $\rho$ and with $d_{L}=d$.

\subsubsection{Cobordism: main definition.}
The plane $\mathbb{R}^2$ as well as domains in $\mathbb{R}^2$ will be
endowed with the symplectic structure $\omega_{\mathbb{R}^2} = dx
\wedge dy$, $(x,y) \in \mathbb{R}^2$.  We endow the product
$\mathbb{R}^2 \times M$ with the symplectic form
$\omega_{\mathbb{R}^2} \oplus \omega$. We denote by $\pi: \mathbb{R}^2
\times M \to \mathbb{R}^2$ the projection. For a subset $V \subset
\mathbb{R}^2 \times M$ and $S \subset \mathbb{R}^2$ we write $V|_{S} =
V \cap \pi^{-1}(S)$.

\begin{dfn}\label{def:Lcobordism} 
   Let $(L_{i})_{1\leq i\leq k_{-}}$ and $(L'_{j})_{1\leq j\leq
     k_{+}}$ be two families of closed, Lagrangian submanifolds of
   $M$. We say that that these two (ordered) families are Lagrangian
   cobordant, $(L_{i}) \simeq (L'_{j})$, if there exists a smooth
   compact cobordism $(V;\coprod_{i} L_{i}, \coprod_{j}L'_{j})$ and a
   Lagrangian embedding $V \subset ([0,1] \times \mathbb{R}) \times M$
   so that for some $\epsilon >0$ we have:
   \begin{equation} \label{eq:cob_ends}
      \begin{aligned}
         V|_{[0,\epsilon)\times \mathbb{R}} = & \coprod_{i}
         ([0, \epsilon) \times \{i\})  \times L_i \\
         V|_{(1-\epsilon, 1] \times \mathbb{R}} = & \coprod_{j} (
         (1-\epsilon,1]\times \{j\}) \times L'_j~.~
      \end{aligned}
   \end{equation}
   The manifold $V$ is called a Lagrangian cobordism from the
   Lagrangian family $(L'_{j})$ to the family $(L_{i})$. We will
   denote such a cobordism by $V:(L'_{j}) \cobto (L_{i})$ or $(V;
   (L_{i}), (L'_{j}))$.
\end{dfn}
\begin{figure}[htbp]
   \begin{center}
      \epsfig{file=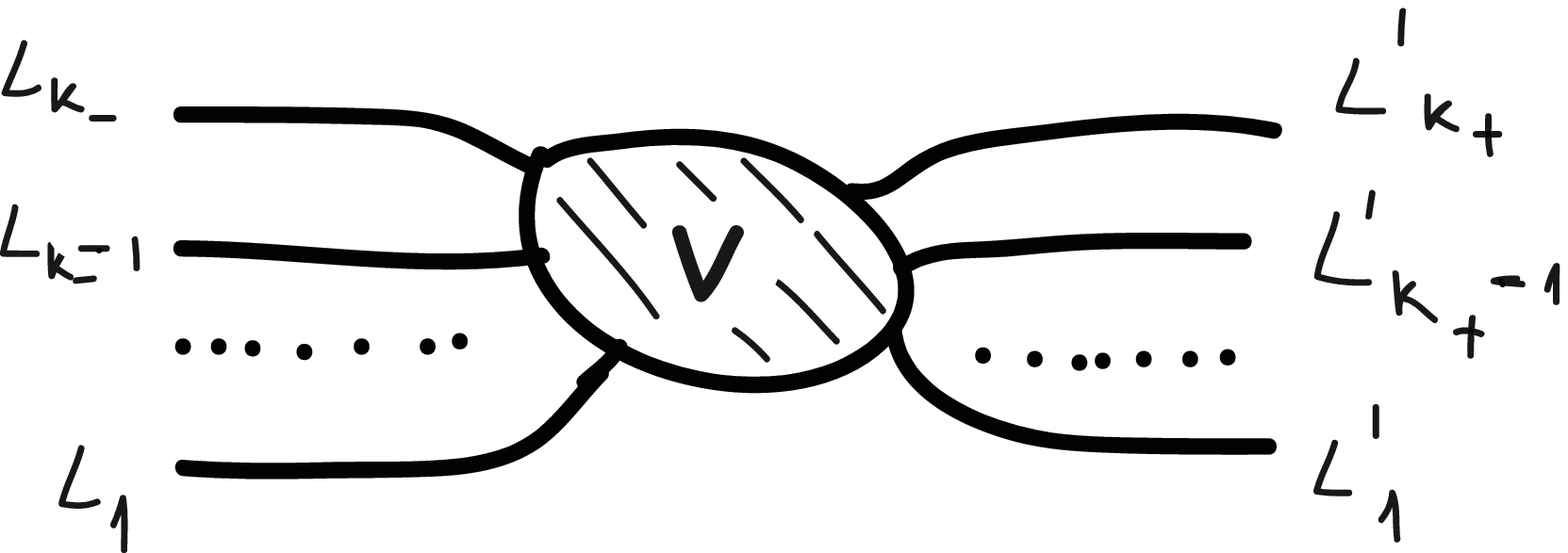, width=0.6\linewidth}
   \end{center}
   \caption{A cobordism $V:(L'_{j})\cobto (L_{i})$ projected
     on $\mathbb{R}^2$.}
\end{figure}

The Lagrangians in the family $(L_i)$ (or $(L'_j)$) are not assumed to
be mutually disjoint inside $M$. In this respect our setting is
somewhat different than in~\cite{Chek:cob}.  An {\em elementary}
cobordism is a cobordism (which might be connected or not) so that the
number of negative ends $k_{-}$ as well as the number of positive ends
$k_{+}$ in Definition~\ref{def:Lcobordism} both have value at most
one.  A cobordism is called {\em monotone} if
$$V\subset ([0,1]\times \mathbb{R})\times M$$ is a 
monotone Lagrangian submanifold.  As in the smooth case, there are
many other possible variants of cobordism depending on additional
structures (for instance, oriented, spin etc).

The definition above, as well as the notation, suggests the existence
of a category where morphisms are represented by cobordisms. This is
discussed in \S\ref{sec:category}.
 
\subsection{Statement of the results.}

Assuming monotonicity, the variant of Floer homology used in most of
the paper is defined over the universal Novikov ring $\mathcal{A}$
with base ring $\Z_{2}$. This homology is not graded.  We will also
make use of quantum homology which, unless otherwise indicated, is
graded and defined over the graded ring $\La$ of Laurent polynomials
in one variable.  We refer to \S \ref{sb:rev-HF} for a quick review of
both constructions and all the relevant notation.

\subsubsection{Floer homology exact sequences.}

\begin{thm}\label{cor:exact-tri-explicit}
   If $V: L \cobto (L_{1}, \ldots, L_{k})$ is a monotone cobordism and
   $N\subset M$ is another Lagrangian so that $L,L_{1},\ldots, L_{k},
   N$ are uniformly monotone, then there exists a sequence of chain
   complexes $K_{i}$ and a sequence of chain maps
   $$m_{i}:CF(N,L_{i};J) \longrightarrow K_{i}$$ so that $K_{i+1}$ 
   is the cone over the map $m_{i}$ (in the category of chain
   complexes), $K_{0}=0$ and there is a quasi-isomorphism $h: CF(N,L;J)
   \longrightarrow K_{k+1}$.  
   Each of these maps only depends on
   $V$ up to chain homotopy and product with some element of the form $T^{a}$ in
   $\mathcal{A}$. Here $J$ is a generic almost complex structure on $M$
    that is compatible with $\omega$.
\end{thm}   

Note that Theorem~\ref{cor:exact-tri-explicit} together with Example
d.  in~\S\ref{subsec:examples} (expanded
in~\S\ref{subsec:cob-lagr-non-iso}) imply the existence of exact
sequences associated to surgery. An exact sequence of Floer homologies
corresponding to Lagrangian surgery (of two Lagrangian submanifolds)
has been previously obtained in~\cite{FO3:book-chap-10} by other
methods. See also~\cite{Se:long-exact}.

By inspecting Definition \ref{def:Lcobordism} it is easy  to see that a cobordism
$$(V; (L_{1},\ldots, L_{k_{-}}), (L'_{1}, \ldots, L'_{k_{+}}))$$
with $k_{+}>1$ can be transformed by ``bending'' the positive ends to
the left into a cobordism $(V'; (L_{1},\ldots, L_{k_{-}},
L'_{k_{+}},\ldots, L'_{2}), L'_{1})$ so that one can apply Theorem
\ref{cor:exact-tri-explicit} to $V'$. Thus, even if the theorem
associates cone-decompositions only to cobordisms with a single
positive end there are in fact analogous results applying to arbitrary
cobordisms.
 
\subsubsection{Quantum homology restrictions} We recall that a
monotone Lagrangian $L$ is narrow if $QH(L)=0$ and it is wide if
$QH(L)\cong H(L;K)\otimes \La$ see~\cite{Bi-Co:rigidity}.

Notice that Theorem~\ref{cor:exact-tri-explicit} implies, in
particular, that if $(V; L, L')$ is a monotone elementary cobordism
and $N\subset M$ is any other Lagrangian submanifold so that $L,L',N$
are uniformly monotone, then $HF(N,L)$ is isomorphic to $HF(N,L')$.
Here is another homological rigidity result concerning elementary
cobordisms that holds this time over the graded ring $\La$.

\begin{thm}\label{thm:quantum_h} 
   If $(V;L,L')$ is a monotone cobordism with $L$ and $L'$ uniformly
   monotone, then $V$ is a quantum $h$-cobordism in the sense that
   $QH(V,L)=0=QH(V,L')$ and moreover $QH(L)$ and $QH(L')$ are
   isomorphic (via an isomorphism that depends on $[V]$) as rings. If
   additionally $L$ and $L'$ are wide, then the singular homology
   inclusions $H_{1}(L;\Z_{2})\to H_{1}(V;\Z_{2})$ and
   $H_{1}(L';\Z_{2})\to H_{1}(V;\Z_{2})$ have the same image. When
   $\dim(L)=2$, both these inclusions are injective and thus
   $H_{1}(L;\Z_{2})\cong H_{1}(L';\Z_{2})$.
\end{thm}
As seen before, elementary cobordisms preserve Floer theoretic
invariants (up to isomorphism) and cobordisms with multiple ends also
satisfy some restrictions as indicated in Theorem
\ref{cor:exact-tri-explicit}.  Below is an example of a more explicit
obstruction to the existence of certain non-elementary cobordisms.
\begin{thm} \label{t:2-end-split} Assume that $(V; (L_{1},L_{2}), L)$
   is a monotone cobordism, that $L_{1},L_{2}, L$ are uniformly
   monotone and none of them is narrow. If $QH(L)$ is a division ring
   (i.e. each non-zero element in $QH(L)$ admits an inverse with
   respect to the quantum multiplication), then the inclusion
   $QH(L)\to QH(V)$ is injective. Moreover, for each $k$ we have the
   inequality:
   \begin{equation}\label{eq:rk}
      rk(QH_{k}(L))\leq |rk(QH_{k}(L_{1}))- rk(QH_{k}(L_{2}))|~.~
   \end{equation}
\end{thm}

\begin{rem} \label{rem:Chek-d}In these two results the uniform monotonicity condition
   can be relaxed to just the requirement that the respective
   monotonicity constants be the same because in view of a result of
   Chekanov \cite{Chek:cob} if a cobordism $V: (L'_{1}, \ldots,
   L'_{k_{+}})\cobto (L_{1},\ldots, L_{k_{-}})$ is connected, then the
   numbers $d_{L_{i}}$, $d_{L'_{j}}$ are all the same.
\end{rem}

The inequality~\eqref{eq:rk} shows that often a Lagrangian with
$QH(L)$ a division ring can not be split in two non-narrow parts by a
Lagrangian cobordism. For instance, it is easy to see that the
inequality~\eqref{eq:rk} never holds in dimension $n=2$, showing that
such a splitting is not possible if $QH(L)$ is a division ring.  Note
that the assumption that $QH(L)$ is a division ring is not rare, at
least when working over the base field $\mathbb{Q}$ (which requires
additional general assumptions on the Lagrangians $L_1, L_2, L$). In
that case if $L$ is a $2$-torus then $QH(L)$ is division ring as soon
as its discriminant (see~\cite{Bi-Co:lagtop}) is not a perfect square.

The main idea for the proof of all three results above is that once a
sufficiently robust notion of Floer homology for Lagrangian cobordisms
is defined, one can deduce relations among the Floer homologies of the
ends of a cobordism out of (non-compactly supported) Hamiltonian
isotopies lifted from isotopies in the plane.

All the constructions involved in these results should extend over
$\Z$ in the presence of additional constraints on the Lagrangians.
There are also generalizations in the graded category along the lines
of \cite{Se:graded} but these extensions will be postponed to future
publications and will not be further discussed here.

\subsubsection{Examples}
The next result is based on analyzing Lagrangian surgery from the
point of view of cobordism.  \RE{ We will see that the trace of
  surgery is a Lagrangian cobordism and will discuss a few resulting
  examples.  In particular, we will prove that:}

\begin{thm}\label{thm:examples}
   There are examples of connected Lagrangians, monotone-cobordant
   \RE{with $N_{L}=1$}, that are not isotopic (even smoothly) and
   \RE{are not monotone-cobordant with $N_{L}=2$}.
\end{thm}

\subsection{Comments on the definition of cobordism and some
  constructions}\label{subsec:examples}

In practice, particularly when studying one cobordism at a time, it is
often more convenient to view cobordisms as embedded in $\mathbb{R}^2
\times M$. Given a cobordism $V\subset ([0,1] \times \mathbb{R})
\times M$ as in Definition~\ref{def:Lcobordism} we can extend
trivially its negative ends towards $-\infty$ and its positive ends to
$+\infty$ thus getting a Lagrangian $\overline{V} \subset \mathbb{R}^2
\times M$. We will in general not distinguish between $V$ and
$\overline{V}$ but if this distinction is needed we will call
\begin{equation}\label{eq:R-extension}
   \overline{V} = \Bigl(\coprod_{i} (-\infty, 0] \times \{i\} 
   \times L_i \Bigr) \cup V \cup 
   \Bigl(\coprod_{j} [1,\infty) \times \{j\} 
   \times L'_j \Bigr)
\end{equation}
The $\mathbb{R}$-extension of $V$.

Here are a few examples of constructions of cobordisms.
\begin{itemize}
  \item[a.] If $L\subset M$ is a Lagrangian submanifold and $\gamma\in
   \C$ is any curve so that outside a compact set $\gamma$ agrees with
   $\mathbb{R} \times \{y\}$, then $\gamma \times L \in \widetilde{M}$
   is an elementary cobordism. If $L$ is monotone, then so is the
   cobordism $\gamma \times L$, with the same minimal Maslov number
   and monotonicity constant. More generally, a possibly non-connected
   curve $\gamma$ that coincides with $\coprod\R\times \{j\}$ outside
   a compact set gives rise to a cobordism $L\times \gamma$. In
   particular, this shows that the Lagrangian family $(L,L)$ is
   null-bordant.
  \item[b.] If the connected Lagrangians $L,L'\subset M$ are
   Hamiltonian isotopic it is easy to construct an elementary
   cobordism joining them using the Lagrangian suspension
   construction~\cite{ALP} (notice however that the projection of this
   cobordism on $\mathbb{R}^2$ will in general not be a curve).
  \item[c.] Let $(V; (L_{i}), (L'_{j}))$ be an {\em immersed}
   Lagrangian cobordism between two families of {\em embedded}
   Lagrangians. This is a cobordism as in
   Definition~\ref{def:Lcobordism} with the exception that $V\to
   ([0,1]\times \mathbb{R}) \times M$ is not a Lagrangian embedding
   but only a Lagrangian immersion. Such a cobordism can be
   transformed into an embedded one by first changing the self
   intersection points of $V$ into generic double points and then
   resolving these double points by Lagrangian surgery (see for
   instance~\cite{Po:surgery}).  It is important to note that by
   resolving these singularities various properties that the initial
   $V$ might have satisfied are in general lost. Monotonicity, for
   instance, is in general not preserved, nor is orientability.
   However, if we do not keep track of these additional structures we
   see that immersed Lagrangian cobordism implies embedded cobordism
   (as noticed by Chekanov~\cite{Chek:cob}).
  \item[d.] Finally, a less immediate verification shows that the
   trace of surgery is also a Lagrangian cobordism. In other words,
   given two transverse Lagrangians $L_{1}, L_{2}$ by applying surgery
   at each of their intersection points one can obtain (a possibly
   disconnected) Lagrangian $L$ that is cobordant to the family
   $(L_{1}, L_{2})$, the cobordism being given by the composition of
   the traces of the surgeries. We will elaborate more on this
   construction in~\S\ref{subsec:cob-lagr-non-iso}.
   \end{itemize}

\begin{rem} \label{rem:group-gen-cob}
   \begin{itemize}
     \item[i.]It is not difficult to see that cobordism is an
      equivalence relation among Lagrangian families: reflexivity is
      of course obvious as well as transitivity. For symmetry a little
      argument is required.  Assume $V$ is a cobordism between
      $(L_{1}, L_{2},\ldots L_{h})$ and $(L_1', \ldots, L_k')$. The
      transformation $a:\C\times M \to \C\times M$ given by
      $a(z,m)=(-z,m)$ is symplectic and, after adjusting the ends of
      the cobordism $a(V)$, it provides a cobordism from
      $(L'_{k},\ldots, L'_{1})$ to $(L_{h},\ldots, L_{1})$. This
      cobordism can be easily adjusted at the ends to an immersed
      cobordism between $(L_1', \ldots, L_k')$ and $(L_{1},
      L_{2},\ldots, L_{h})$. By point c. above this can be transformed
      into an embedded Lagrangian cobordism. (Note that this
      construction fails for monotone cobordisms, hence being monotone
      cobordant does not seem to be an equivalence relation \RB{for families}.)
     \item[ii.] Given two Lagrangian families
      $\mathcal{L}=(L_{1},\ldots, L_{h})$ and
      $\mathcal{L}'=(L'_{1},\ldots, L'_{k})$ define their sum
      $\mathcal{L}+\mathcal{L'}=(L_{1},\ldots, L_{h}, L'_{1},\ldots,
      L'_{k})$. In view of the properties described above it is easy
      to see that this operation defines a group structure on the set
      of equivalence classes of Lagrangian families of $M$. By
      applying appropriate surgeries it is easy to see that this group
      is commutative. (In contrast, there is no apriori reason why
      $\mathcal{L} + \mathcal{L}'$ should be monotone cobordant to
      $\mathcal{L}' + \mathcal{L}$.)
     \item[iii.] It is easy to see that elementary cobordism is also
      an equivalence relation among the Lagrangians of $M$ (surgery is
      not needed for this argument).
     \item[iv.] Special elementary cobordism of any of the three
      following types - monotone, oriented, or spin - is an
      equivalence relation. Again reflexivity is obvious and symmetry
      follows as in Remark~\ref{rem:group-gen-cob} i. without any need
      to perform surgeries. Transitivity is obvious too in the
      orientable and spin cases. In the monotone case, it follows from
      the Van Kampen theorem for relative $\pi_{2}(-,-)$'s viewed as
      cross-modules (see~\cite{Brown-Higgins:rel-homotopy}) that
      gluing two monotone cobordisms with the same monotonicity
      constant along a connected monotone end produces a monotone
      cobordism. However, as already mentioned earlier, non-elementary
      monotone cobordism is not necessarily an equivalence relation.
      \RB{\item[v.] As noted in Remark \ref{rem:Chek-d},} it follows from an observation of Chekanov in
        \cite{Chek:cob} that if $V:(L_{1},\ldots, L_{k})\to
        (L'_{1},\ldots L'_{s})$ is a connected, monotone cobordism
        (with $N_{L}\geq 2$), then $d_{L_{i}}=d_{L'_{k}}=d_{V}$ for
        all $i,k$.  As noted in \cite{Chek:cob}, this implies for
        instance that the Clifford and Chekanov tori in $\C^{2}$ are
        not monotone cobordant (with $N_{L}\geq 2$).
   \end{itemize}
\end{rem}


\section{A quick review of Lagrangian Floer theory}
\label{sb:rev-HF} 
This section recalls briefly the basic definitions and notational
conventions for Floer homology and Lagrangian quantum homology in the
standard case of closed Lagrangian submanifolds. As such it can be
safely skipped by experts.  We refer the reader to~\cite{Oh:HF1,
  Oh:HF1-add, Oh:spectral} for the foundations of Floer homology for
monotone Lagrangians, and to~\cite{FO3:book-vol1, FO3:book-vol2} for
the general case. \RE{For details on the variant of Lagrangian quantum
  homology used here} see~\cite{Bi-Co:qrel-long, Bi-Co:rigidity,
  Bi-Co:Yasha-fest, Bi-Co:lagtop}.

\subsection{Lagrangian Floer homology} \label{sbsb:lag-hf} Let $L_0,
L_1 \subset M$ be two monotone Lagrangian submanifolds with $d_{L_0} =
d_{L_1} = d$. In case $d \neq 0$ we assume in addition that $L_0$ and
$L_1$ have the same monotonicity constant (or in other words that the
pair $(L_0, L_1)$ is uniformly monotone). \RE{In case the ground ring
  is not $2$-torsion we also assume that $L_{0}$, $L_{1}$ are spin
  with fixed spins structures.}

Denote by $\mathcal{A}$ the universal Novikov ring, i.e. $$\mathcal{A}
= \Bigl\{ \sum_{k=0}^{\infty} a_k T^{\lambda_k} \mid a_k \in K, \;
\lim_{k \to \infty} \lambda_k = \infty \Bigr\},$$ endowed with the
obvious multiplication. We do not grade $\mathcal{A}$.

Denote by $\mathcal{P}(L_0, L_1) = \{ \gamma \in C^0([0,1],M) \mid
\gamma(0) \in L_0, \gamma(1) \in L_1\}$ the space of paths in $M$
connecting $L_0$ to $L_1$. For  $\eta \in \pi_0(\mathcal{P}(L_0,
L_1))$ we denote the path connected component of $\eta$ by
$\mathcal{P}_{\eta}(L_0,L_1)$. 

Fix $\eta \in \pi_0(\mathcal{P}(L_0, L_1))$ and let $H: M \times [0,1]
\to \mathbb{R}$ be a Hamiltonian function with Hamiltonian flow
$\psi_t^H$. We assume that $\psi_1^H(L_0)$ is transverse to $L_1$. (We
generally view $H$ as a mean of possible perturbation of $L_0$, and
when not needed we will often use $H=0$.)  We denote by
$\mathcal{O}_{\eta}(H)$ the set of paths $\gamma \in
\mathcal{P}_{\eta}(L_0, L_1)$ which are orbits of the flow $\psi^H_t$.
Finally, we choose also a generic $1$-parametric family of almost
complex structures $\mathbf{J} = \{J_t\}_{t \in [0,1]}$ compatible
with $\omega$.

Using this data one can define in a standard way the Floer complex
$CF(L_0, L_1; \eta; H, \mathbf{J})$ with coefficients in
$\mathcal{A}$. Recall that the underlying module of this complex is
generated over $\mathcal{A}$ by the elements of
$\mathcal{O}_{\eta}(H)$. The Floer differential $\partial: CF(L_0,
L_1; \eta; H, \mathbf{J}) \longrightarrow CF(L_0, L_1; \eta; H,
\mathbf{J})$ is defined as follows. For a generator $\gamma_{-} \in
\mathcal{O}_{\eta}(H)$ define
$$\partial(\gamma_{-}) = 
\sum_{\scriptscriptstyle \gamma_+ \in \mathcal{O}_{\eta}(H)}
\sum_{\scriptscriptstyle u \in \mathcal{M}_0(\gamma_{-}, \gamma_{+};H,
  \mathbf{J})} \varepsilon(u)T^{\omega(u)} \gamma_{+}.$$ Here
$\mathcal{M}_0(\gamma_{-}, \gamma_{+}; H, \mathbf{J})$ stands for the
$0$-dimensional components of the space of Floer strips
$u:\mathbb{R}\times [0,1] \longrightarrow M$ connecting $\gamma_{-}$
to $\gamma_{+}$, modulo the $\mathbb{R}$-action coming from
translation in the $\R$ coordinate. The strips $u$ are assumed to have
finite energy and we denote by $\omega(u) = \int_{\mathbb{R} \times
  [0,1]} u^* \omega$ the symplectic area of $u$.  Finally, each such
strip $u$ comes with a sign $\varepsilon(u) = \pm 1 \in K$. As
mentioned before, in this paper we will mostly work with $K =
\mathbb{Z}_2$ hence the \RE{signs} $\varepsilon(u)$ are irrelevant.
Under the preceding assumptions on $L_0, L_1$ we have $\partial^2 = 0$
hence one can define the homology $$HF(L_0, L_1; \eta; H, \mathbf{J})
= \ker (\partial) / \textnormal{image\,} (\partial) \,.$$

\begin{rem}\label{rem:non-grading}
   In the general context of the paper, with $CF$ defined over
   $\mathcal{A}$, the chain complex $CF$ is not graded and hence $HF$
   has no grading too. In special situations one can endow $CF$ with
   some grading though not always over $\mathbb{Z}$ (e.g. when $L_0$
   and $L_1$ are both oriented, then there is a
   $\mathbb{Z}_2$-grading).  See~\cite{Se:graded} for a systematic
   approach to these grading issues. 
\end{rem}

Standard arguments show that the homology $HF(L_0, L_1; \eta; H,
\mathbf{J})$ is independent of the additional structures $H$ and
$\mathbf{J}$ up to canonical isomorphisms. We will therefore omit $H$
and $\mathbf{J}$ from the notation. 

We will often consider all components $\eta \in
\pi_0(\mathcal{P}(L_0,L_1))$ together i.e. take the direct sum complex
\begin{equation} \label{eq:cf-all-eta} CF(L_0,L_1;H,\mathbf{J}) =
   \bigoplus_{\eta} CF(L_0,L_1;\eta; H,\mathbf{J})
\end{equation}
with total homology which we denote $HF(L_0,L_1)$. There is an obvious
inclusion map $i_{\eta}: HF(L_0,L_1;\eta) \longrightarrow
HF(L_0,L_1)$.

\begin{rems} \label{r:auton-J} 
   \begin{enumerate}
     \item[i.] When $L_0$ and $L_1$ are mutually transverse we can take
      $H=0$ in $CF(L_0, L_1; H, \mathbf{J})$ in which case the
      generators of the complex are the intersection points $L_0 \cap
      L_1$ and Floer trajectories $\mathcal{M}_0(\gamma_{-},
      \gamma_{+};0,\mathbf{J})$ are genuine holomorphic strips
      connecting intersection points $\gamma_{-}, \gamma_{+} \in L_0
      \cap L_1$.  When $H=0$ we will omit it from the notation and
      just write $CF(L_0, L_1; \mathbf{J})$. We will sometimes omit
      $\mathbf{J}$ too when its choice is obvious.
     \item[ii.] The use of families of almost complex structures
      $\mathbf{J} = \{J_t\}_{t \in [0,1]}$ is needed for
      transversality reasons typically occurring in the construction
      of Floer homology. However, it is still possible to work with
      almost complex structure $J$ that do not depend on $t$, provided
      the Hamiltonian $H$ is chosen to be generic
      (see~\cite{Fl-Ho-Sa:transversality}).
   \end{enumerate}
\end{rems}

\subsection{Moving boundary conditions}
\label{subsubsec:movingbdry}

As before assume that $L_{0}$ and $L_{1}$ are two transverse
Lagrangians. Fix the component $\eta$ and the almost complex structure
$\mathbf{J}$. We also fix once and for all a path $\gamma_{0}$ in the
component $\eta$. Now let $\varphi = \{\varphi_t\}_{t \in [0,1]}$ be a
Hamiltonian isotopy starting at $\varphi_0 = \Id$. The isotopy
$\varphi$ induces a map $$\varphi_* : \pi_0(\mathcal{P}(L_0, L_1))
\longrightarrow \pi_0(\mathcal{P}(L_0, \varphi_1(L_1)))$$ as follows.
If $\eta \in \pi_0(\mathcal{P}(L_0, L_1))$ is represented by
$\gamma:[0,1] \to M$ then $\varphi_* \eta$ is defined to be the
connected component of the path $t \mapsto \varphi_t(\gamma(t))$ in
$\mathcal{P}(L_0, \varphi_1(L_1))$.

The isotopy $\varphi$ induces a canonical isomorphism
\begin{equation} \label{eq:iso-c-phi}
   c_{\varphi}: HF(L_0, L_1; \eta)
   \longrightarrow HF(L_0, \varphi_1(L_1); \varphi_* \eta)
\end{equation}
which comes from a chain level map defined using moving
boundary conditions (see e.g.~\cite{Oh:HF1}). The isomorphism
$c_{\varphi}$ depends only on the homotopy class (with fixed end
points) of the isotopy $\varphi$.

The definition of the isomorphism $c_{\varphi}$ involves some
subtleties due to our use of the universal Novikov ring $\mathcal{A}$
as base ring: given that the symplectic area of the strips with moving
boundaries can vary inside a one parametric moduli space it follows
that the naive definition of the morphism $c_{\varphi}$ - so that each
strip is counted with a weight given by its symplectic area - does not
provide a chain map.  We explain here in more detail the construction
of the map $c_{\varphi}$.

Let $\varphi = \{ \varphi^{H}_{t}\}$ be a Hamiltonian \RE{diffeomorphism}
generated by $H$. Denote $L'_1 = \varphi^H_{1}(L_{1})$ and assume that
$L'_1$ is also transverse to $L$ and that the Floer complexes
$C_{1}=CF(L_{0}, L_{1};\eta; 0; \mathbf{J})$ and $C_{2}=CF(L_{0},
L'_1;\varphi_{\ast}\eta; 0;\mathbf{J})$ are well-defined.

Put $\psi_{t}=(\varphi_{t}^{H})^{-1}$. We define the functional
$\Theta_H: \mathcal{P}_{\varphi_{\ast}\eta}(L_{0},L'_1)\to \R$,
$$\Theta_H(\gamma)=\int_{0}^{1}H(\psi_t(\gamma(t))dt
-\int_{0}^{1}H(\gamma_{0}(t))dt~.~$$

Let $\beta:\R\to [0,1]$ be a smooth function so that $\beta(s)=0$ for
$s\leq 0$, $\beta(s)=1$ for $s\geq 1$ and $\beta$ is strictly
increasing on $(0,1)$. Let $x$ be a generator of $C_{1}$. We write
\begin{equation}\label{eq:moving-brdy-sum}
   \tilde{c}_{\varphi}(x)=\sum_{y} \Bigl(\sum_{u} 
   T^{\omega(v_{u})-\Theta_H(y)}\Bigr)y
\end{equation}
with $y$ going over the generators of $C_{2}$ and $u$ going over all
the elements of a zero dimensional moduli space of solutions to
Floer's homogeneous equation $\overline{\partial}_{\mathbf{J}}u=0$
that start at $x$ and end at $y$ and satisfy the boundary conditions
$$u(s,0)\subset L_{0} \ , \ u(s,1)\subset \varphi_{\beta(s)}(L_{1})~.~$$
Here the element $v_{u}:\R\times [0,1]\to M$ is defined by the formula
$v_{u}(s,t)=\psi_{t\beta(s)}u(s,t)$ so that $v_{u}(s,t)$ is a strip
with boundary conditions on $L_{0}$ and $L_{1}$.  It is easy to check
that with this definition $\tilde{c}_{\varphi}$ is a chain map.  Note
that the quantity $|\omega(v_{u}) -\omega(u)|$ is bounded by the
variation of $H$ so that $\tilde{c}_{\varphi}$ is well defined over
$\mathcal{A}$. Further, the map $c_{\varphi}$ induced in homology by
$\tilde{c}_{\varphi}$, depends only on the homotopy class with fixed
end points of $\varphi$. Similar constructions can be used to adapt
the rest of the usual Floer theoretic machinery to this moving
boundary situation. They show in particular that $c_{\varphi}$ induces
an isomorphism in homology.
 
\begin{rem} This argument also applies without modification to cases
   when $M$ is not compact (but e.g. tame), if we have some control
   which insures that all solutions $u$ of finite energy as above have
   their image inside a fixed compact set $K \subset M$.
\end{rem}

\subsection{The pearl complex and Lagrangian quantum homology}
\label{sbsb:lag-qh} Next we briefly describe the version of Lagrangian
quantum homology that will be used later in the paper. This is a
version of the self Floer homology of a Lagrangian submanifold. The
identification between the two homologies can be done via a
Piunikin-Salamon-Schwarz type quasi-isomorphism. While Lagrangian
Floer homology has been developed in great generality
in~\cite{FO3:book-vol1, FO3:book-vol2}, this version - \RB{specific to the monotone setting} - has been
suggested by Oh~\cite{Oh:relative}, following an
idea of Fukaya, under the name of {\em relative quantum (co)homology}.
The theory was later implemented and further developed in our previous
work~\cite{Bi-Co:qrel-long, Bi-Co:rigidity, Bi-Co:Yasha-fest,
  Bi-Co:lagtop} to which we \RB{refer, in particular, for technical details}.
This formalism is particularly
efficient in applications and this is why we use it here. We use the
name {\em Lagrangian quantum homology} to avoid confusion with the
relative quantum invariants in the sense of
Ionel-Parker~\cite{IoPa:relGW}.

Let $L \subset M$ be a monotone Lagrangian with minimal Maslov number
$N_L\geq 2$. Denote by $\Lambda = K[t^{-1}, t]$ the ring of Laurent
polynomials in $t$, graded so that $|t| = -N_L$. (In case $L$ is
weakly exact, i.e.  $\omega(A)=0$ for every $A \in \pi_2(M,L)$ we put
$\Lambda = K$.) The chain complex used to define the Lagrangian
quantum homology $QH(L)$ is denoted by $\mathcal{C}(\mathscr{D})$ and
called the pearl complex.  It is associated to a triple of auxiliary
structures $\mathscr{D} = (f, (\cdot, \cdot), J)$ where $f: L
\longrightarrow \mathbb{R}$ is a Morse function on $L$, $(\cdot,
\cdot)$ is a Riemannian metric on $L$ and $J$ is an
$\omega$-compatible almost complex structure on $M$.  With these
structures fixed we have
$$\mathcal{C}(\mathscr{D}) = K \langle \textnormal{Crit}(f) \rangle
\otimes \Lambda.$$ This complex is $\mathbb{Z}$-graded with grading
combined from both factors. The grading on the left factor is defined
by Morse indices of the critical points.  The differential $d$ on this
complex is defined by counting so called pearly trajectories. 
The homology $H_*(\mathcal{C}(\mathscr{D}),d)$
is independent of $\mathscr{D}$ (up to canonical isomorphisms) and is
denoted by $QH_*(L)$. Note that this homology is $\mathbb{Z}$-graded
\RE{and $N_{L}$ periodic}.

In what follows we will actually need also to enrich the coefficients
of $QH(L)$ to the Novikov ring $\mathcal{A}$. This is done as follows.
Denote by 
\begin{equation}\label{eq:min-area}
A_L = \min \{\omega(A) \mid A \in \pi_2(M,L), \;
\omega(A)>0 \}
\end{equation}
the minimal positive area of a disk with boundary on
$L$. We use the convention that $\min \emptyset = \infty$.  The
Novikov ring $\mathcal{A}$ becomes an algebra over $\Lambda$ via the
ring morphism induced by $\Lambda \ni t \mapsto T^{A_L} \in
\mathcal{A}$. (If $L$ is weakly exact we have $\Lambda = K$ and we
view $\mathcal{A}$ as an algebra over $\Lambda$ in the usual way.)
Consider now $$\mathcal{C}(\mathscr{D};\mathcal{A}) =
\mathcal{C}(\mathscr{D}) \otimes_{\Lambda} \mathcal{A}, \quad
d_{\mathcal{A}} = d \otimes_{\Lambda} \textnormal{id}.$$ The homology
of this complex will be denoted by $QH(L;\mathcal{A})$. Denote by $j_{\mathcal{A}}:\mathcal{C}(\mathscr{D})\to \mathcal{C}(\mathscr{D};\mathcal{A})$ the inclusion.  In
contrast to $\mathcal{C}(\mathscr{D})$ and $QH(L)$, their analogues
over $\mathcal{A}$, $\mathcal{C}(\mathscr{D};\mathcal{A})$ and
$QH(L;\mathcal{A})$ are not graded.

To avoid confusion between $\Lambda$ and $\mathcal{A}$ we will
sometimes write $QH(L;\Lambda)$ for $QH(L)$.

The following simple algebraic remark will be useful later in the
paper.
  
\begin{lem}\label{lem:alg-inclusion}
   Suppose that $\mathcal{C}$ is a free $\La$-chain complex and let
   $\mathcal{C}'=\mathcal{C}\otimes_{\La}\mathcal{A}$. The map in
   homology $H(\mathcal{C})\to H(\mathcal{C}')$ induced by
   $j_{\mathcal{A}}:\mathcal{C}\to \mathcal{C}'$ is injective. In
   particular, the change of coefficients $QH(- ;\La)\to QH( -
   ;\mathcal{A})$ is injective.
\end{lem}
\begin{proof} Let $\mathcal{C}= F\otimes \La$ where $F$ is a graded,
   finite dimensional $\Z_{2}$-vector space. Given $T^{a}f\in
   \mathcal{C}'$ with $a\in \R$ and $f\in F$ put $v(T^{a}f)=a/\rho
   -|f|$, thus $v(j_{\mathcal{A}}(t^{k}f))= -|t^{k}f|$.  Assume first
   that $c'\in \mathcal{C}'$ is the image of a cycle $c\in\mathcal{C}$
   of pure degree equal to $k$. Assume also $c'=d_{\mathcal{A}}e'$,
   $e'\in \mathcal{C}'$.  We now decompose
   $e'=\sum_{\alpha}e'_{\alpha}$ so that $v(e'_{\alpha})=\alpha$.
   Notice that for an element of pure degree $f\in F$ we have
   $v(d_{\mathcal{A}} T^{a}f)=v(T^{a}f)+1$. Therefore,
   $d_{\mathcal{A}}e'=c'$ means that $d_{\mathcal{A}}e'_{-k-1}=c'$ and
   $d_{\mathcal{A}}e'_{\beta}=0$ for all $\beta\not= -k-1$.  As
   $v(e'_{-k-1})=-k-1$ we can write $e'_{-k-1}=\sum T^{h_{i}}e_{i}$
   where $\{e_{i}\}\subset F$ is a basis formed by elements of pure
   degree. Write $c'=\sum T^{ k_{i}}e_{i}$.  Thus, each $k_{i}$ is an
   integral multiple of $A_{L}$. Moreover, in each differential
   $d_{\mathcal{A}}e_{i}$ the powers of $T$ that appear are also
   integral powers of $A_{L}$. In view of this we put $e''=\sum_{i\in
     S}= T^{h_{i}}e_{i}$ where $S$ is the set of indexes $i$ so that
   $h_{i}$ is an integral multiple of $A_{L}$. We write
   $e'_{-k-1}=e''+e'''$ and we see that $d_{\mathcal{A}}e''=c'$ and
   $e''\in \textnormal{image\,}(j_{\mathcal{A}})$.  When
   $c'=j_{\mathcal{A}}(c)$ with $c$ not necessarily of pure degree we
   decompose $c'$ as $c'=\sum_{j} c'_{j}$ with $c'_{j}$ so that
   $v(c'_{j})=j$ and we apply the argument above to each non-vanishing
   $c'_{j}$.
\end{proof}

\subsubsection{The PSS isomorphism} \label{sbsb:pss} Let $L \subset M$
be a monotone Lagrangian. Denote by $\eta_0 \in
\pi_0(\mathcal{P}(L,L))$ the connected component of a constant path on
$L$. In contrast to the case of two general Lagrangians, the Floer
homology of the pair $(L,L)$ is all concentrated in the component
$\eta_0$, i.e.  $i_{\eta_0} : HF(L,L;\eta_0) \longrightarrow HF(L,L)$
is an isomorphism.

The PSS (Piunikin-Salamon-Schwarz) isomorphism is a comparison between
the Lagrangian quantum homology and the self Floer homology of $L$.
More precisely, there is a canonical isomorphism
$$PSS: QH(L;\mathcal{A}) \longrightarrow HF(L,L)$$
coming from a chain morphism $\widetilde{PSS}_{\eta_0}
:\mathcal{C}(\mathscr{D};\mathcal{A}) \longrightarrow
CF(L,L;\eta_0;H,\mathbf{J})$. The construction of
$\widetilde{PSS}_{\eta_0}$ is very similar to the one described
in~\cite{Alb:PSS, Bi-Co:rigidity, Bi-Co:Yasha-fest} over the ring
$\Lambda$.  The only needed modification when working with
$\mathcal{A}$ is to incorporate the total areas of the connecting
trajectories that appear in the morphism $\widetilde{PSS}_{\eta_0}$.

The map $PSS_{\eta_0}: QH(L;\mathcal{A}) \longrightarrow
HF(L,L;\eta_0)$ induced in homology by $\widetilde{PSS}_{\eta_0}$ is
an isomorphism. The isomorphism $PSS$ is now defined as $i_{\eta_0}
\circ PSS_{\eta_0}$.

There also exists a version of the PSS morphism which is defined using
moving boundary conditions. Specifically, assume that $\varphi$ is a
Hamiltonian isotopy and let $L'=\varphi_{1}(L)$.  Then we have an
isomorphism
$$\widehat{PSS}:QH(L;\mathcal{A}) \longrightarrow HF(L,L')~.~$$ 
Its definition is straightforward in view of the definition of
$\widetilde{PSS}$ and~\S\ref{subsubsec:movingbdry}.

\subsection{Products and other structures} \label{sbsb:prod} The
Lagrangian and ambient quantum homologies as well as the Floer
homologies are all related via several compatible algebraic structures
endowed with ring and module operations. In the more general context
of Floer homology these issues are developed in \cite{FO3:book-vol1,
  FO3:book-vol2}. We refer the reader to~\cite{Bi-Co:qrel-long,
  Bi-Co:rigidity, Bi-Co:Yasha-fest} for the explicit constructions in
the special setting that is used in this paper.
  
The quantum homologies $QH(L;\Lambda)$ and $QH(L;\mathcal{A})$ are
endowed with an associative  product $*$ with unity  (they are in general 
not commutative). We denote the unity by $[L] \in QH_n(L;\Lambda)$.

For a uniformly monotone pair of Lagrangians $(L_1, L_2)$ the Floer
homology $HF(L_1, L_2)$ is a left module over $QH(L_1;\mathcal{A})$
and a right module over $QH(L_2;\mathcal{A})$. We denote these module
operations by $\alpha_1*x$ and $x*\alpha_2$, for $x \in HF(L_1, L_2)$,
$\alpha_1 \in QH(L_1;\mathcal{A})$, $\alpha_2 \in
QH(L_2;\mathcal{A})$. The two module structures are mutually
compatible in the sense that associativity holds:
$(\alpha_1*x)*\alpha_2 = \alpha_1*(x*\alpha_2)$.

There is a duality isomorphism relating $HF(L_1, L_2)$ and
$\textnormal{hom}_{\mathcal{A}}(HF(L_2, L_1), \mathcal{A})$. In case
$L=L_1=L_2$ is exact this duality reduces to Poincar\'{e} duality.
Similarly, $QH(L)$ also admits a duality induced by the correspondence
between the pearl complex of a function $f$ and the pearl complex of
the function $-f$.

Finally, the Floer homology $HF(L_1, L_2)$ is also a module over the
ambient quantum homology $QH(M)$.


\section{Floer homology and the proof of
  Theorem~\ref{cor:exact-tri-explicit}}
\label{sb:func-F}

In the sequel we will make use of Floer homology for pairs of
Lagrangian submanifolds with cylindrical ends - a natural extension of
cobordisms that we introduce just below.  Given this definition there
are essentially three ingredients that are important in the proofs of
all our results: a compactness argument, a definition of Floer
complexes for Lagrangians with cylindrical ends, and finally a method
to use plane curve combinatorics to deduce algebraic properties of the
differential in such Floer complexes.  Variants of these constructions
appear in slightly different settings in the literature (see for
instance the works of Seidel~\cite{Se:book-fukaya-categ},
Abouzaid~\cite{Abouzaid:homog-coord},
Auroux~\cite{Aur:fuk-cat-sym-prod}, as well as earlier work of
Oh~\cite{Oh:hf-non-compact}). Besides this, standard techniques
together with the methods in
\cite{Bi-Co:qrel-long},\cite{Bi-Co:rigidity} are sufficient to deal
with transversality issues.

\subsection{Lagrangian submanifolds with cylindrical ends}
\label{sbsb:lag-cyl} 
To simplify notation we will write from now on $\widetilde{M} =
\mathbb{R}^2 \times M$ endowed with the split form
$\omega_{\mathbb{R}^2} \oplus \omega$. We will also identify in the
standard way $\mathbb{R}^2 \cong \mathbb{C}$ endowed with the standard
complex structures $i$.

By a {\em Lagrangian submanifold with cylindrical ends} we mean a
Lagrangian submanifold $\overline{V} \subset \widetilde{M}$ without
boundary that has the following properties:
\begin{enumerate}
  \item For every $a<b$ the subset $\overline{V}|_{[a,b] \times
     \mathbb{R}}$ is compact.
  \item There exists $R_+$ such that $$\overline{V}|_{[R_+, \infty)
     \times \mathbb{R}} = \coprod_{i=1}^{k_+} [R_+, \infty) \times
   \{a^+_i\} \times L^{+}_i$$ for some $a^+_1 < \cdots < a^+_{k_+}$
   and some Lagrangian submanifolds $L^+_1, \ldots, L^+_{k_+} \subset
   M$.
  \item There exists $R_{-} \leq R_+$ such that
   $$\overline{V}|_{(-\infty, R_{-}] \times
     \mathbb{R}} = \coprod_{i=1}^{k_-} (-\infty, R_-] \times \{a^-_i\}
   \times L^{-}_i$$ for some $a^{-}_1 < \cdots < a^-_{k_-}$ and some
   Lagrangian submanifolds $L^{-}_1, \ldots, L^{-}_{k_-} \subset M$.
\end{enumerate}
We allow $k_+$ or $k_-$ to be $0$ in which case $\overline{V}|_{[R_+,
  \infty)\times \mathbb{R}}$ or $\overline{V}|_{(-\infty, R_{-}]
  \times \mathbb{R}}$ are void.

For every $R \geq R_+$ write $E^+_{R}(\overline{V}) =
\overline{V}|_{[R, \infty) \times \mathbb{R}}$ and call it a positive
cylindrical end of $\overline{V}$. Similarly, we have for $R \leq
R_{-}$ a negative cylindrical end $E^{-}_{R}(\overline{V})$.

Obviously if $W$ is a cobordism between $(L'_1, \ldots, L'_r)$ and
$(L_1, \ldots, L_s)$ then its $\mathbb{R}$-extension $\overline{W}$ -
see (\ref{eq:R-extension}) - is a Lagrangian submanifold of
$\widetilde{M}$ with cylindrical ends.  Vice versa, if $\overline{W}$
is a Lagrangian submanifold with cylindrical ends then by an obvious
modification of the ends (and a possible symplectomorphism on the
$\mathbb{R}^2$ component) it is easy to obtain a Lagrangian cobordism
between the families of Lagrangians corresponding to the positive and
negative ends of $\overline{W}$.

In order to simply terminology, we will say that a Lagrangian with
cylindrical ends $\overline{V}$ is cylindrical outside of a compact
subset $K \subset \mathbb{R}^2$ if $\overline{V}|_{\mathbb{R}^2
  \setminus K}$ consists of horizontal ends, i.e. it is of the form
$E^{-}_{R_{-}}(\overline{V}) \cup E^+_{R_+}(\overline{V})$.

We will also need the following notion.
\begin{dfn} \label{d:cyl-dist} Two Lagrangians with cylindrical ends
   $\overline{V}, \overline{W} \subset \widetilde{M}$ are said to be
   cylindrically distinct at infinity if there exists $R>0$ such that
   $\pi(E^+_R(\overline{V})) \cap \pi(E^+_R(\overline{W})) =
   \emptyset$ and $\pi(E^-_{-R}(\overline{V})) \cap
   \pi(E^-_{-R}(\overline{W})) = \emptyset$.
\end{dfn}

Finally, let us describe a class of Hamiltonian isotopies that will be
useful in the following.
\begin{dfn}[Horizontal isotopies] \label{d:isotopies} Let
   $\{\overline{V_t}\}_{t \in [0,1]}$ be an isotopy of Lagrangian
   submanifolds of $\widetilde{M}$ with cylindrical ends. We call this
   isotopy horizontal if there exists a (not necessarily compactly
   supported) Hamiltonian isotopy $\{\psi_t\}_{t \in [0,1]}$ of
   $\widetilde{M}$ with $\psi_0 = \Id$ and with the following
   properties:
   \begin{enumerate}
     \item[i.] $\overline{V_t} = \psi_t(\overline{V_0})$ for all $t
      \in [0,1]$.
     \item[ii.] There exist real numbers $R_- < R_+$ such that for all
      $t\in [0,1]$, $x\in E^{\pm}_{R_{\pm}}(\overline{V_0})$ we have
      $\psi_t(x)\in E^{\pm}_{R_{\pm}}(\overline{V_0})$.
     \item[iii.] There is a constant $K>0$ so that for all $x\in
      E^{\pm}_{R_{\pm}}(\overline{V_0})$, $|d\pi_{x}(X_t(x))|< K$.  Here
      $X_t$ is the (time dependent) vector field of the flow
      $\{\psi_t\}_{t \in [0,1]}$.
   \end{enumerate}
   In other words, the Hamiltonian flow $\psi_t$ moves tangentially
   along the cylindrical ends of $\overline{V_0}$ and at bounded
   speed.  Of course, the ends of all the Lagrangians $\overline{V_t}$
   coincide at infinity.  We say that two Lagrangians $\overline{V},
   \overline{V'} \subset \widetilde{M}$ with cylindrical ends are {\em
     horizontally} isotopic if there exists an isotopy as above
   $\{\overline{V_t}\}_{t \in [0,1]}$ with
   $\overline{V_0}=\overline{V}$ and $\overline{V_1} = \overline{V'}$.
   Finally, we will sometime say that an ambient Hamiltonian isotopy
   $\{\psi_t\}_{t \in [0,1]}$ as above is horizontal with respect to
   $\overline{V_0}$.
\end{dfn}

\subsection{Compactness} \label{subsec:compact} Given that cobordisms
are viewed as Lagrangians with cylindrical ends and thus are
non-compact, the compactness of pseudo-holomorphic curves with
boundaries on such Lagrangians is the first main technical issue that
one has to deal with.  We address this issue following a variation on
an argument that has originally appeared in Chekanov's work
\cite{Chek:cob}.

For this discussion we fix two Lagrangians with cylindrical ends
$\overline{W}$ and $\overline{W'}$  -
see Figure~\ref{fig:CobordismsFloerStrips}.
\begin{figure}[htbp]
   \epsfig{file=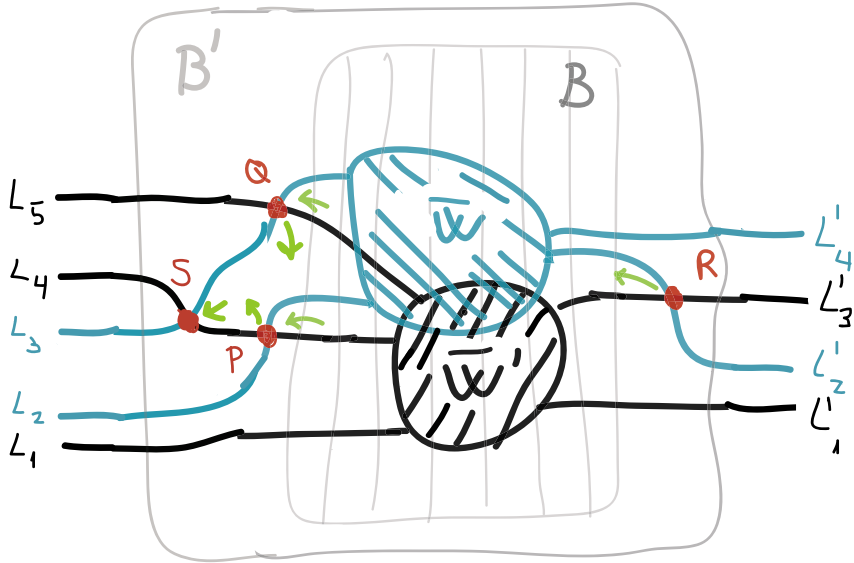, width=0.70\linewidth}
   \caption{Two Lagrangians with cylindrical ends $\overline{W}$ and
     $\overline{W'}$ projected on the plane with the box $B$ outside
     of which $\pi$ is $(\widetilde{J}, i)$-holomorphic and with
     $\widetilde{J}$-holomorphic {\color{green} strips} starting and
     entering intersection points.  Outside the box $B'$ the ends are
     horizontal and do no longer intersect.}
   \label{fig:CobordismsFloerStrips}
\end{figure}
In contrast to~\eqref{eq:R-extension} we do not assume that they are
cylindrical horizontal\\ outside of  $[0,1]\times\mathbb{R} \times
M$, but rather that they are cylindrical outside a
compact subset $B' \subset \mathbb{R}^2$ in the sense of
\S\ref{sbsb:lag-cyl}. We also fix a compact region in the plane
$B\subset B' \cong \mathbb{R}^2$ and we will only consider almost
complex structures $\widetilde{J}$ so that $\pi$ is $(\widetilde{J},
i)$-holomorphic outside $B \times M$.  Moreover, outside $B$ each of
the cobordisms coincide for the negative ends with products
$\gamma^-_i \times L^-_{i}$ between certain planar curves
$\gamma^-_{i}$ and Lagrangians $L_{i}\subset M$ and similarly for the
positive ends, they are products $\gamma^+_j \times L^+_{j}$ with
Lagrangians $L^+_{j} \subset M$ and $\gamma^+_{j}$ curves in $\R^{2}$.

We also assume that the negative planar curves of $\overline{W}$
and those of $\overline{W'}$ intersect transversely, and similarly for
the positive planar curves of the two cobordisms.  Two curves
that correspond to positive (respectively, negative) ends of
$\overline{W}$ do not intersect outside $B$ and similarly for
$\overline{W'}$. Further, we  also assume that the Lagrangians in
$M$ corresponding to the positive ends of $\overline{W}$ and those
corresponding to the positive ends of $\overline{W'}$ are two-by-two
transverse in $M$ and similarly for the negative ends. 

The basic argument here appeared already in~\cite{Chek:cob} and is as
follows.  Assume that $u:\Sigma \to \mathbb{C} \times M$ is a
$\widetilde{J}$-holomorphic curve where $\Sigma$ is either the disk
$D^{2}$, the strip $\R\times [0,1]$ or the sphere $S^{2}$. In case
$\Sigma$ is the disk we assume that $u$ maps the boundary $\partial
\Sigma$ either to $\overline{W}$ or to $\overline{W'}$, and if
$\Sigma$ is the strip, we assume $u(\R\times\{0\})\subset
\overline{W}$, $u(\R\times\{1\})\subset \overline{W'}$.

\begin{lem}\label{lem:open-mapping} 
   Assume that the symplectic energy of $u$ is finite. Then either
   $\pi \circ u$ is constant or $\pi\circ u(\Sigma)\subset B'$.
\end{lem}
\begin{proof}
   The first remark is that $\pi\circ u(\Sigma)$ is bounded. Indeed,
   this is clear for $\Sigma= D^{2}, S^{2}$.  If $\Sigma=\R\times
   [0,1]$ then due to the finite energy condition we get that
   $u(\Sigma)$ converges at $\pm \infty$ to some point in
   $\overline{W}\cap \overline{W'}$. But as $\pi(\overline{W}\cap
   \overline{W'})\subset B'$ we get that $\pi\circ u(\Sigma)$ is
   bounded in this case too.

   Now assume that $\pi\circ u (\Sigma)\not\subset B'$.  Notice that
   $\C\backslash (B' \cup \pi(\overline{W}) \cup \pi(\overline{W'}))$
   is a union of unbounded domains in $\C$.  As
   $\textnormal{image\,}(\pi\circ u)$ is bounded it follows that
   $\pi\circ u$ is constant. Indeed, otherwise an application of the
   open mapping theorem to the holomorphic map $\pi\circ u$ implies
   that the image of $\pi\circ u|_{\textnormal{Int\,} \Sigma}$
   contains an unbounded region.
\end{proof}

\begin{rem} a. \label{rem:extension-opem-mapp} It is a simple exercise
   to show that the conclusion of the Lemma \ref{lem:open-mapping}
   remains valid even if $u$ is not $\widetilde{J}$-holomorphic but
   rather it satisfies a perturbed Cauchy-Riemann equation of the form
   $\overline{\partial}_{\widetilde{J}} u - \widetilde{J}
   X_{H}(z,u)=0$ where $H_{z}:\widetilde{M}\to \R$, $z\in \Sigma$ is a
   smooth family of Hamiltonians with a compact support contained in
   $B'\times M$.
   
   b. It is easy to see that, in the argument above, the actual set
   $B$ does not intervene: only $B'$ plays a role. In other words,
   with $B'$ fixed as above, the almost complex structure
   $\widetilde{J}$ only needs to satisfy the requirement relative to
   $B=B'$.  However, we will reuse
   Figure~\ref{fig:CobordismsFloerStrips} later in the paper
   in~\S\ref{sb:certain-strips} and at that time the particular choice
   of $B$ will be useful.
\end{rem}
 
\subsection{Definition of Floer homology for Lagrangians with cylindrical ends}
\label{sbsb:hf-lag-cyl} Here we explain the necessary modifications
needed for the constructions and structures from~\S\ref{sb:rev-HF} to
adapt to Lagrangian cobordisms (rather than just closed Lagrangian
submanifolds).

Let $\overline{W}$ and $\overline{W}'$ be two uniformly monotone
Lagrangians with cylindrical ends.  We will {\em not} assume for now
that they are cylindrically distinct at $\infty$ - see
Definition~\ref{d:cyl-dist}.

We intend to define the Floer complex
$CF(\overline{W},\overline{W}';\eta; (H,f); \widetilde{\mathbf{J}})$
with coefficients in $\mathcal{A}$ 
 (see \S\ref{sb:rev-HF}) and we now
describe the data involved in this definition.

\noindent {\bf A}. The almost complex structure
$\widetilde{\mathbf{J}}=\{ \widetilde{J}_{t}\}_{t\in [0,1]}$.  For a
compact subset $B \subset \mathbb{R}^2$ denote by
$\widetilde{\mathbf{\mathcal{J}}}_B$ the (families of) almost complex
structures $\{ \widetilde{J_t} \}_{t \in [0,1]}$ on $(\widetilde{M},
\widetilde{\omega}) = (\R^{2}\times M, \omega_{0}\oplus \omega)$ with the
following properties:
\begin{enumerate}
  \item For every $t$, $\widetilde{J_t}$ is an
   $\widetilde{\omega}$-tamed almost complex structure on
   $\widetilde{M}$.
  \item For every $t$, the projection $\pi$ is $(\widetilde{J_t},
   i)$-holomorphic on $(\mathbb{R}^2 \setminus B) \times M$.
\end{enumerate}
If $B = \emptyset$ we simply write $\widetilde{\mathbf{\mathcal{J}}}$.

\noindent {\bf B}. The component $\eta\in
\pi_{0}(\mathcal{P}(\overline{W},\overline{W}')$ is fixed as in
\S\ref{sbsb:lag-hf}.

\noindent {\bf C}. The perturbation $(H,f)$.  To describe these
perturbations we first fix the notation for the ends of $\overline{W}$
(see \S\ref{sbsb:lag-cyl}).  Thus for $R_{+}$ and $R_{-}$ sufficiently
big we assume
$$\overline{W}|_{[R_+, \infty) \times \mathbb{R}} = 
\coprod_{i=1}^{k_+} [R_+, \infty) \times \{a^+_i\} \times L^{+}_i$$
for some $a^+_1 < \cdots < a^+_{k_+}$ and
$$\overline{W}|_{(-\infty, R_{-}] \times
  \mathbb{R}} = \coprod_{i=1}^{k_-} (-\infty, R_-] \times \{a^-_i\}
\times L^{-}_i$$ for some $a^{-}_1 < \cdots < a^-_{k_-}$.

The couple $(H,f)$ consists of two Hamiltonians $H: [0,1]\times
\widetilde{M}\to \R$ and $f:\R^{2}\to \R$ with the following
properties:
\begin{enumerate}
  \item The support of $H$ is compact.
  \item The function $f:\R^{2}\to \R$ satisfies:
   \begin{itemize}
     \item[1.] The support of $f$ is contained in the union of the
      sets $$U_{i}^{+}= [R_{+}+1,\infty)\times
      [a^{+}_{i}-\epsilon^{+}_{i}, a^{+}_{i}+\epsilon^{+}_{i}]$$ and
      $$U_{i}^{-}= (-\infty, R_{-}-1]\times
      [a^{-}_{i}-\epsilon^{-}_{i}, a^{-}_{i}+\epsilon^{-}_{i}]~,~$$ where
      the positive constants $\epsilon^{\pm}_{i}$ are small enough
      (and the numbers $R_{+}$ and $R_{-}$ are big enough) so that
      all the sets $U^{\pm}_{i}$ above are pairwise disjoint.
     \item[2.] The restriction of $f$ to each set
      $V_{i}^{+}=[R_{+}+2,\infty)\times [a^{+}_{i}-\epsilon^{+}_{i}/2,
      a^{+}_{i}+\epsilon^{+}_{i}/2]$ and $V_{i}^{-}=(-\infty, R_{+}-2]
      \times [a^{-}_{i}-\epsilon^{-}_{i}/2,
      a^{-}_{i}+\epsilon^{-}_{i}/2]$ is of the form
      $$f(x,y)=\alpha_{i}^{\pm}x+\beta_{i}^{\pm}$$ with
      $\alpha_{i}^{\pm}\in \R$ sufficiently small so that the 
      Hamiltonian isotopy of $\R^{2}$, $\phi^{f}_{t}$, associated to $f$
      keeps the sets $[R_{+}+2,\infty)\times \{a^{+}_{i}\}$ and
      $(-\infty, R_{-}-2]\times \{a^{-}_{i}\}$ inside the respective
      $V_{i}^{\pm}$ for $0\leq t\leq 1$.
     \item[3.] We assume $R_{+}$ and $R_{-}$ sufficiently big so that
      $\overline{W}'$ is cylindrical on $(-\infty, R_{-}]\times \R$ as
      well as on $[R_{+},\infty)\times \R$ and we assume that
      $\epsilon^{\pm}_{i}$ is sufficiently small so that
      \begin{equation}\label{eq:neigh-disj}
         (U_{i}^{+}\backslash ([R_{+}+1,\infty)
         \times \{a^{+}_{i}\}))\cap \overline{W}'=\emptyset
      \end{equation}
      for all indexes $i$ and similarly 
      \begin{equation}\label{eq:neigh-disj2}
         (U_{i}^{-}\backslash ((-\infty, R_{-}-1]\times 
         \{ a^{-}_{i}\} ))\cap \overline{W}'=\emptyset~.~
      \end{equation}
   \end{itemize}
\end{enumerate}
Let $e=f\circ \pi$ be the composition with
$\pi:\widetilde{M}=\R^{2}\times M\to \R^{2}$ the projection. We denote
by $\mathcal{H}(\overline{W}, \overline{W}')$ the space of pairs
$(H,f)$ as above \RE{so that additionally $\phi^{e}_{1}(\overline{W})$
  and $\overline{W}'$ are cylindrically distinct at infinity.} The
role of the function $f$ is to perturb the Lagrangian $\overline{W}$
so as to render it cylindrically distinct at infinity from
$\overline{W}'$ by using the Hamiltonian flow associated to $e$.  This
is precisely the meaning of the requirements in equations
(\ref{eq:neigh-disj}), (\ref{eq:neigh-disj2}): the Hamiltonian flow
$\phi^{e}_{t}$ associated to $e$ has the property that
$\phi^{e}_{t}(\overline{W})$ and $\overline{W}'$ are cylindrically
distinct at infinity for all $t\in (0,1]$ whether or not
$\overline{W}$ and $\overline{W}'$ are cylindrically distinct at
infinity to start with.  Clearly, if $\overline{W}$ and
$\overline{W}'$ are not cylindrically distinct at infinity, then the
constants $\alpha^{\pm}_{i}$ associated to those ends of
$\overline{W}$ that coincide with some ends of $\overline{W}'$ satisfy
$\alpha^{\pm}_{i}\not=0$.  This implies that in this case the space
$\mathcal{H}(\overline{W},\overline{W}')$ has more than a single
connected component. It is easy to see that each such component is
convex, hence contractible.  Moreover, these components only depend on
$f$ and not on $H$ so that we will denote the path component of
$\mathcal{H}(\overline{W},\overline{W}')$ associated to a pair $(H,f)$
by $[f]$.
 
Finally, we define the complex $CF(\overline{W},\overline{W}';\eta;
(H,f);\widetilde{\mathbf{J}})$ where $\eta$ is as at point \RE{{\bf  B}}
above, $(H,f)\in \mathcal{H}(\overline{W},\overline{W}')$ generic and
$\widetilde{\mathbf{J}} \in \widetilde{\mathbf{\mathcal{J}}}_B$ for
some compact set $B$ is also generic.
 
We put:
\begin{equation}\label{eq:Floer-cplx-cyl}
   CF(\overline{W},\overline{W}';\eta; (H,f);\widetilde{\mathbf{J}}):= 
   CF(\phi^{f\circ \pi}_{1}(\overline{W}), 
   \overline{W}';\eta'; H;\widetilde{\mathbf{J}}),
\end{equation} 
where $\eta'$ is the path component that corresponds to $\eta$ under
the isotopy $\phi^{f\circ \pi}_{t}$. 
 
Of course, we still have to justify the right term in equation
(\ref{eq:Floer-cplx-cyl}). In view of the fact that $H$ is compactly
supported and due to our choice of $\widetilde{\mathbf{J}}$ it is
immediate to see that the (standard) construction of the Floer complex
- recalled in~\S\ref{sbsb:lag-hf} - carries over to this setting. This
is true because compactness for the finite energy solutions of Floer's
equation
\begin{equation}\label{eq:Floer-equation}
   \overline{\partial}_{\widetilde{\mathbf{J}}}
   u +\nabla H(t,u)=0
\end{equation}
for $u:[0,1]\times \R\to \widetilde{M}$ subject to the boundary
conditions $u(\{0\}\times \R)\subset \phi^{e}_{1}(\overline{W})$ and
$u(\{1\}\times\R)\subset \overline{W}'$ follows from an immediate
adaptation of Lemma \ref{lem:open-mapping} as indicated in Remark
\ref{rem:extension-opem-mapp}.  Thus the Floer complex
$CF(\phi^{f\circ\pi}_{1}(\overline{W}), \overline{W}';\eta';
H;\widetilde{\mathbf{J}})$ is well-defined. 
 
As in \S\ref{sbsb:lag-hf} we omit the component $\eta$ in case we take
into account all Hamiltonian chords, belonging to all the connected
components of $\mathcal{P}(\overline{W},\overline{W}')$.
 
\begin{prop} \label{prop:invariance-Fl} The homology of the complex
   $CF(\overline{W}, \overline{W}';(H,f);\widetilde{\mathbf{J}})$ is
   independent of $H$, $\widetilde{\mathbf{J}}$ and only depends on
   the path connected component $[f]\in \pi_0(
   \mathcal{H}(\overline{W}, \overline{W}'))$ up to canonical
   isomorphism.  We denote this homology by
   $HF(\overline{W},\overline{W}'; [f])$.
   
   If $\phi = \{\phi_t\}_{t \in
     [0,1]}$ is a horizontal isotopy with respect to $\overline{W}$,
   then there is an isomorphism $HF(\overline{W},\overline{W}'; [f])\to
   HF(\phi_{1}(\overline{W}),\overline{W}';\phi_{1}[f])$ that only depends on the
   homotopy class of the path of Hamiltonian diffeomorphisms
   $\phi_{t}$ (with fixed end-points). A similar statement is valid if
   we act with a horizontal isotopy on $\overline{W}'$ and keep
   $\overline{W}$ fixed.
\end{prop}
 
In case $\overline{W}$ and $\overline{W}'$ are distinct at infinity,
then $\mathcal{H}(\overline{W},\overline{W}')$ is path connected and
we may take $f=0$. In this case we denote the homology simply by
$HF(\overline{W},\overline{W}')$.  Moreover, if $\overline{W}$,
$\overline{W}'$ are distinct at infinity and transverse, then for
generic $\widetilde{\mathbf{J}} \in \widetilde{\mathcal{J}}_{B}$ (with
$B$ sufficiently big) the complex $CF(\overline{W}, \overline{W}';
(0,0); \widetilde{\mathbf{J}})$ is well-defined and we denote it by
$CF(\overline{W},\overline{W}'; \widetilde{\mathbf{J}})$.
 
\begin{proof}[Proof of Proposition \ref{prop:invariance-Fl}] First,
   the standard invariance arguments for Floer homology easily adapt
   to this setting, again by using the compactness argument in Lemma
   \ref{lem:open-mapping}, to show independence with respect to
   choices of $H$ and $\widetilde{\mathbf{J}}$. The only less
   immediate invariance statements concern the independence of $f$ -
   inside the same connected component of $\mathcal{H}(\overline{W},
   \overline{W}')$ - and with respect to horizontal homotopies.
 
   The invariance in both these cases follows from the standard
   construction of Floer Lagrangian comparison maps in the case of
   moving Lagrangian boundary conditions - as described in
   \S\ref{subsubsec:movingbdry} combined with yet another application
   of the compactness Lemma \ref{lem:open-mapping}.
 
   We exemplify the argument to prove independence with respect to
   $f$.  Thus assume that $f$ and $f'$ are so that $(H,f),
   (H,f')\subset \mathcal{H}(\overline{W},\overline{W}')$ and
   $[f]=[f']$. We also pick a compact set $B\subset \R^{2}$ as well as
   a generic $\widetilde{\mathbf{J}}\in\widetilde{\mathcal{J}}_{B}$.
   Let $\nu:\R\to [0,1]$ be an increasing $C^{\infty}$ function so
   that $\nu(\tau)=0$ for $\tau\leq 0$ and $\nu(\tau)=1$ for $\tau\geq
   1$.  Define $f_{\tau}=\nu(\tau) f +(1-\nu(\tau))f'$,
   $f_{\tau}:\R^{2}\to \R$, $\tau\in \R$.  Let $e_{\tau}=f_{\tau}\circ
   \pi$.  Denote by
   $\overline{W}_{\tau}=\phi^{e_{\tau}}_{1}(\overline{W})$.
   Therefore, $\overline{W}_{\tau}=\phi^{f'\circ
     \pi}_{1}(\overline{W})$ for $\tau\leq 0$ and
   $\overline{W}_{\tau}=\phi^{f\circ \pi}_{1}(\overline{W})$ for
   $\tau\geq 1$.  We now define a morphism:
   $$\psi: CF(\phi^{f'\circ \pi}_{1}(\overline{W}), 
   \overline{W}';H;\widetilde{\mathbf{J}})\to CF(\phi^{f\circ
     \pi}_{1}(\overline{W}), \overline{W}';
   H;\widetilde{\mathbf{J}})$$ by a sum like in Equation
   (\ref{eq:moving-brdy-sum}) running over the elements of zero
   dimensional moduli spaces consisting of finite energy solutions to
   Floer's equation (\ref{eq:Floer-equation}) subject to the boundary
   conditions
   \begin{equation}\label{eq:bdry-cond}
      u(0,s)\in \overline{W}_{s} \ , \ \ u(1,s)\in \overline{W}'
      \ \ \forall s\in\R ~.~
   \end{equation}
   The only difficulty in checking that this morphism is well-defined
   and satisfies the expected properties in standard Floer theory
   (i.e.  it induces a canonical isomorphism in homology as
   in~\S\ref{subsubsec:movingbdry}) is to insure that the moduli
   spaces of finite energy Floer trajectories with moving boundary
   conditions as above satisfy the usual compactness properties.  But
   this follows immediately by noticing that, because $[f]=[f']$ we
   have that $\overline{W}_{\tau}$ and $\overline{W}'$ are
   cylindrically distinct at infinity for all $\tau\in \R$. This
   implies that Lemma \ref{lem:open-mapping} can still be applied and
   it shows that the image of a finite energy solution of equation
   (\ref{eq:Floer-equation}) subject to (\ref{eq:bdry-cond}) is either
   constant or it has its image contained in a compact set $K\subset
   \widetilde{M}$ that contains the support of $H$ and whose
   projection on $\R^{2}$ contains $B$ as well as the rectangle
   $[R_{-}-3, R_{+}+3]\times [a, b]$ where $a <
   a^{\pm}_{i}-\epsilon^{\pm}_{i}$ and
   $b>a^{\pm}_{i}-\epsilon^{\pm}_{i}$ for all $i$.
     
   The argument showing invariance with respect to horizontal
   isotopies is similar.
\end{proof}

\subsection{Non-existence of certain holomorphic strips and the proof
  of Theorem~\ref{cor:exact-tri-explicit}} \label{sb:certain-strips}

We now construct a particular family of cobordisms. Let $a \geq 0$,
$q, r, s \in \mathbb{R}$.  Consider a smooth function
$\sigma_{a;q,r,s}: \R\to \R$, with the following properties:
\begin{itemize}
  \item[i.] $\sigma_{a;q,r,s}(t)= q$ for $t\leq -a$,
   $\sigma_{a;q,r,s}(t)=s$ for $t\geq 3$.
  \item[ii.]$\sigma_{a;q,r,s}(t)=r$ for $t\in [-a+1,2]$.
  \item[iii.] $\sigma_{a;q,r,s}$ is strictly monotone on $(-a,-a+1)$
   and strictly monotone on $(2,3)$.
\end{itemize}
We denote by $\gamma_{a;q,r,s} \subset \mathbb{R}^2$ the graph of
$\sigma_{a;q,r,s}$. See Figure~\ref{fig:graph}.

\begin{figure}[htbp]
   \epsfig{file=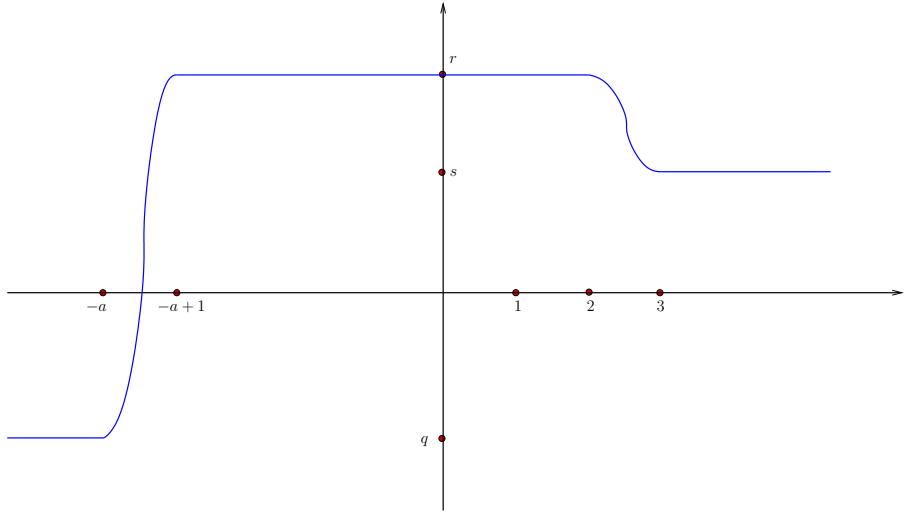, width=0.75\linewidth}
   \caption{The graph of $\sigma_{a;q,r,s}$. \label{fig:graph}}
\end{figure}

Let $V: L \cobto (L_1, \ldots, L_k)$ be a cobordism with one positive
end and let $N\subset M$ be a Lagrangian in $M$. We assume that $N$ is
transverse to $L$ as well as to $L_{i}$ for all $1\leq i \leq k$. By a
possible isotopy of $\overline{V}$ we may assume that $$\overline{V}
\subset \mathbb{R} \times [0,k] \times M, \quad \overline{V}|_{[1,
  \infty)\times \mathbb{R}} = [1,\infty)\times \{1\} \times L$$ and
that for some large enough $a>2$ we have
$$\overline{V}|_{(-\infty, -a+2] \times \mathbb{R}} = \coprod_{i=1}^k
(-\infty, -a+2] \times \{i\} \times L_i.$$ We now consider two
Lagrangians in $\mathbb{R}^2 \times M$:
$$N^{\wedge} = \gamma_{a; -1, k+1, 2} \times N, \quad 
N^{\vee} = \gamma_{a; -1, -2, 2} \times N.$$ The intersections
between these Lagrangians and $\overline{V}$ is:
\begin{align*}
   & N^{\wedge} \cap \overline{V} = \bigcup_{i=1}^k \{(q_i, i)\} \times 
   (N \cap L_i), \;\;\; \textnormal{where } q_i \in (-a, -a+1), \, 
   \sigma_{a; -1, k+1, 2} = i, \\
   & N^{\vee} \cap \overline{V} = \{(p,1)\} \times (N \cap L), \;\;\; 
   \textnormal{where } p \in (2,3), \, \sigma_{a; -1, -2, 2}(p)=1.
\end{align*}

It is easy to see that Theorem \ref{cor:exact-tri-explicit} is a
consequence of the following lemma.

\begin{lem}\label{cor:Floer-structure2}
   There exist (time dependent) almost complex structures
   $\widetilde{\mathbf{J}} = \{\widetilde{J_t}\}_{t \in [0,1]}$ on
   $\mathbb{R}^2 \times M$ with the following properties:
   \begin{enumerate}
     \item For every $t$, $\widetilde{J_t}$ is compatible with
      $\omega_{\mathbb{R}^2} \oplus \omega$.
     \item For every t, $\pi$ is $(\widetilde{J_t}, i)$-holomorphic on
      $(\mathbb{R}^2 \times M) \setminus ([-a+1,2]\times [-K,K] \times
      M)$ (for a large positive constant $K$). Here $i$ is the
      standard complex structure on $\mathbb{R}^2 \cong \mathbb{C}$.
     \item The Floer complexes $CF(N,L_{i};\mathbf{J}^i)$, $i=1,
      \ldots, k$, $CF(N,L;\mathbf{J}^0)$,
      $CF(N^{\wedge},\overline{V};\widetilde{\mathbf{J}})$,
      $CF(N^{\vee}, \overline{V}; \widetilde{\mathbf{J}})$ are all
      well defined, where
      $\mathbf{J}^i=\widetilde{\mathbf{J}}|_{\{(q_{i}, i)\} \times
        M}$, $\mathbf{J}^0=\mathbf{J}|_{\{(p,1)\} \times M}$.
   \end{enumerate}
   Moreover, $CF(N^{\vee},\overline{V})= CF(N,L)$ and there is a chain
   homotopy-equivalence
   $$\bar{\phi}_{V}^{N}:CF(N^{\vee},\overline{V}) \longrightarrow
   CF(N^{\wedge}, \overline{V})$$ implied by the fact that $N^{\vee}$
   and $N^{\wedge}$ are horizontally isotopic. The complex
   $CF(N^{\wedge},\overline{V})$ has the form:
   \begin{equation}\label{eq:traing-form}CF(N^{\wedge},
      \overline{V})=\Bigl(CF(N,L_{1})[-s_1]
      \oplus CF(N,L_{2})[-s_2] \oplus\cdots \oplus CF(N,L_{k})[-s_k],
      D\Bigr)
   \end{equation} with the differential given by an upper triangular matrix
   $D = (D_{ij})$ whose diagonal entries $D_{ii}$ are up to sign the
   differentials of the complex $CF(N, L_{i})$, and the indexes
   $s_{i}\in \Z$ are independent of $N$. (See figure~\ref{fig:cones}.)
\end{lem}
\begin{rem} Even if we work here in a non-graded context we felt
   useful to include degrees in the formulas above so that the
   statement remains true in a graded context assuming additional
   assumptions on the Lagrangians involved.
\end{rem}
\begin{figure}[htbp]
   \epsfig{file=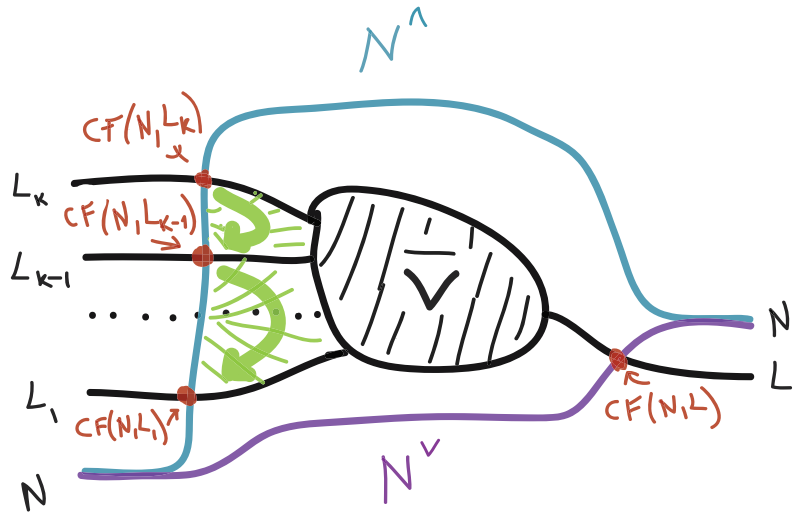, width=0.65\linewidth}
   \caption{The cobordisms $V$, $N^{\wedge}$ and $N^{\vee}$ together
     with, in {\color{green} green}, some of the
     $\widetilde{\mathbf{J}}$-holomorphic strips relevant for the
     iterated cone structure (everything projected to
     $\mathbb{R}^2$).} \label{fig:cones}
\end{figure}

\begin{proof}[Proof of Lemma~\ref{cor:Floer-structure2}]
   Finding an almost complex structure so that all the Floer complexes
   involved are well-defined and with the required properties with
   respect to the projection is standard.  In view of
   Proposition~\ref{prop:invariance-Fl} we only need to show that the
   form of the differential $D$ is as claimed and the existence of the
   chain homotopy equivalence $\bar{\phi}_{V}^{N}$.

   We start with the differential $D$. For this we first return to the
   setting of~\S\ref{subsec:compact}, which in the notation of our
   Lemma reads as follows: $\overline{V}$ is a Lagrangians in
   $\widetilde{M}$ with cylindrical ends and $\widetilde{\mathbf{J}}
   \in \widetilde{\mathcal{J}}_{B}$ for some compact set $B \subset
   \R^{2}$. We also recall the particular choices for the ends of
   $\overline{V}$ from~\S\ref{subsec:compact}.  Namely, outside $B$
   this cobordism coincides - for the negative ends - with products
   $\gamma^{-}_{i} \times L^-_{i}$ between certain planar curves
   $\gamma^-_{i}$ and Lagrangians $L_{i}\subset M$ and similarly for
   the positive end, it is a product $\gamma^+ \times L$ with a
   Lagrangian $L \subset M$ for some curve $\gamma^+ \subset \C$, as
   depicted in Figure~\ref{fig:cones}. We also assume the
   transversality conditions mentioned in~\S\ref{subsec:compact}.

   We will also need the following notation. Denote by
   $\gamma^{\wedge} \subset \mathbb{C}$ the curve corresponding to the
   cylindrical cobordism $N^{\wedge}$ in Figure~\ref{fig:cones}. We
   view $\gamma^{\wedge}$ as a (non-compact) $1$-dimensional
   submanifold of $\mathbb{C}$ and we orient it by going along
   $\gamma^{\wedge}$, starting from the right-hand side of
   Figure~\ref{fig:cones} and ending at its left-hand side. Fix also a
   non-vanishing vector field $\vec{\xi}(z) \in T_z(\gamma^{\wedge})$,
   $z \in \gamma^{\wedge}$, representing this orientation. The curve
   $\gamma^{\wedge}$ separates $\mathbb{C}$ into two connected
   components. We denote by $\mathcal{U}$ the component lying
   ``above'' $\gamma^{\wedge}$ (i.e.  $\mathcal{U}$ is on the
   ``right'' of $\gamma^{\wedge}$ with respect to this orientation).

   Consider now the Floer complex $CF(N^{\wedge}, \overline{V};
   \widetilde{\mathbf{J}})$. Note that the generators of this complex
   are of the form $(x, p)$ with $p \in N \cap L_i$ and $x \in
   \gamma^{\wedge} \cap \gamma^-_i$, $1 \leq i \leq k$. Let
   $u:\mathbb{R} \times [0,1] \longrightarrow \widetilde{M}$ be a
   Floer trajectory, contributing to the differential of this complex,
   connecting $(x,p)$ to $(y,q)$ with $x \in \gamma^{\wedge} \cap
   \gamma^-_i$ and $y \in \gamma^{\wedge} \cap \gamma^-_j$. We have to
   show that $j \leq i$ and moreover if $j=i$ then $u$ is of the form
   $u(s,t) = (x, u'(s,t))$ with $u'(s,t)$ a Floer trajectory of $CF(N,
   L_i; \mathbf{J}^i)$.

   In order to prove this, put $v = \pi \circ u: \mathbb{R} \times
   [0,1] \longrightarrow \mathbb{C}$. Note that $v$ is holomorphic
   over $\mathbb{C} \setminus B$ (i.e. $v|_{(\mathbb{R} \times [0,1])
     \setminus v^{-1}(B)}$ is holomorphic), where we use here the
   standard complex structures on $\mathbb{R} \times [0,1]$ and on
   $\mathbb{C}$, both denoted by $i$. \RB{In the proof we will use the
     following elementary consequence of the open mapping theorem:
     \begin{rem}\label{rem:open-domain} Let $v:\mathbb{R} \times [0,1]
        \to \C$ be a continuous map and $U \subset \mathbb{C}$ an open
        connected subset. Suppose that:
        \begin{itemize}
          \item[i.] $\textnormal{image\,}(v) \cap U \neq \emptyset$
           and moreover $v$ is holomorphic over $U$.
          \item[ii.] The limits $p=\lim_{s\to +\infty} v(s,t)$ and
           $q=\lim_{s\to -\infty} v(s,t)$ both exist, are outside of
           $U$ and $v(\mathbb{R} \times [0,1])\cup \{p,q\}\subset \C$
           is compact.
          \item[iii.] $v(s,0)\not\in U$, $v(s,1)\not\in U$ for all $s\in
           \R$.
        \end{itemize}
        Then the image of $v$ contains $U$.  In particular, it is not
        possible for $U$ to be unbounded.

        Indeed, condition ii implies that the set $U'=v(\mathbb{R}
        \times [0,1]) \cap U$ is closed in $U$.  Condition iii
        together with the open mapping theorem implies that $U'$ is
        also open in $U$.  As $U$ is connected we deduce that $U'=U$.
     \end{rem}}
   
   We now return to the proof of the Lemma.  Note that $v(s,0) \in
   \gamma^{\wedge}$ for every $s \in \mathbb{R}$. Thus our statement
   would follow if we prove that $\partial_s v(s,0)$ points in the
   same direction as $\vec{\xi}(v(s,0))$ for every $s$. More
   precisely, we have to show that if we write $\partial_s v(s,0) =
   C(s) \vec{\xi}(v(s,0))$ for some $C(s) \in \mathbb{R}$ then
   $C(s)\geq 0$ for every $s$, and moreover if $\partial_s v(s,0) = 0$
   for every $s$, then $v$ is the constant map with value $x$.

   To prove this, suppose by contradiction that $C(s_0) < 0$ for some
   $s_0$. As $\gamma^{\wedge}$ is disjoint from $B$ we have
   $\partial_s v(s,0) + i \partial_t v(s,0)=0$, hence $\partial_t
   v(s,0) = C(s) i \vec{\xi}(v(s,0))$.  As $C(s_0)<0$ it follows that
   $\partial_t v(s_0,0)$ points towards $\mathcal{U}$. This implies
   that the image of $v$ intersects $\mathcal{U}$ and as $v$ is
   holomorphic over $\mathcal{U}$ we also have that the image of $v$
   must intersect $\mathcal{U} \setminus \bigcup_{l=1}^k \gamma^-_l$.
   Notice that due to the boundary conditions imposed to $u$ we know
   that $v$ sends $\{0,1\}\times \R$ away from $\mathcal{U} \setminus
   \bigcup_{l=1}^k \gamma^-_l$.  As all connected components of
   $\mathcal{U}\setminus\bigcup_{l=1}^k \gamma^-_l$ are unbounded,
   this contradicts Remark \ref{rem:open-domain} and thus completes
   the proof that $C(s)\geq 0$ for every $s$, hence also proves that
   $j \leq i$.


   If $j=i$ the above proof shows that $v(s,0) \equiv x$ for every
   $s$.  As $v$ is holomorphic near $\mathbb{R} \times \{0\}$ it
   follows that $v$ is constant. Thus $u(s,t) = (x, u'(s,t))$ and it
   is easy to see that $u'$ is a Floer trajectory for $CF(N, L_i;
   \mathbf{J}^i)$. This completes the proof that the differential $D$
   is upper triangular and that the diagonal elements have the form
   claimed.

   The existence of the chain homotopy equivalence
   $\bar{\phi}^{N}_{V}$ results from the invariance of the Floer
   homology for cylindrical Lagrangians with respect to horizontal
   isotopies.
\end{proof}

\begin{rem} \label{r:traj} Slight variations on the \RB{argument} in
   the proof of Lemma~\ref{cor:Floer-structure2} can be used to
   restrict the type of Floer trajectories in various similar
   situations, such as the ones depicted in
   Figures~\ref{fig:CobordismsFloerStrips},~\ref{fig:CobTwoEnds5}.
\end{rem}

Using Lemma~\ref{cor:Floer-structure2} it is a simple exercise in
homological algebra to use the components of the differential $D$ to
identify the complexes $K_{i}$ as well as the maps $m_{i}$ and $h$.
To finish the proof of Theorem \ref{cor:exact-tri-explicit} we also
need to notice that these maps are each unique up to chain homotopy
and multiplication with some $T^{a}\in\mathcal{A}$.  It is enough for
this to understand the reasoning for the chain map
$\bar{\phi}_{V}^{N}$ as the same argument applies to the $m_{i}$'s.
We shorten $\bar{\phi}= \bar{\phi}_{V}^{N}$.  From Proposition
\ref{prop:invariance-Fl} we deduce that as long as $N^{\wedge}$ and
$N^{\vee}$ are kept fixed, then the resulting $\bar{\phi}$ is unique
up to chain homotopy. However, $N^{\wedge}$ and $N^{\vee}$ are not
unique, they depend on the choice of the functions
$\sigma_{a;q,r,s}(t)$. For a different choice of such functions we
have the Lagrangians $N^{\wedge}_{1}$ and $N^{\vee}_{1}$ - that can be
assumed horizontally isotopic to $N^{\wedge}$ and $N^{\vee}$
respectively - and a resulting chain isomorphism $\bar{\phi}_{1}$.
Again by the invariance claim in Proposition \ref{prop:invariance-Fl},
we deduce that $\bar{\phi}_{1}\circ i^{\vee}$ is chain homotopic to
$i^{\wedge} \circ\bar{\phi}$ where $i^{\vee}: CF(N^{\vee},V)\to
CF(N^{\vee}_{1}, V)$ and $i^{\wedge}: CF(N^{\wedge},V)\to
CF(N^{\wedge}_{1}, V)$ are moving boundary conditions comparison maps.
Now the key point here is that the map $i^{\vee}$ is not the identity
via the identification $CF(N^{\vee},V)= CF(N,L) = CF(N^{\vee}_{1},V)$.
Rather, it is multiplication with some $ T^{a}\in\mathcal{A}$ where
$a$ takes into account the energy of the Hamiltonian moving $N^{\vee}$
to $N^{\vee}_{1}$ - see (\ref{eq:moving-brdy-sum}). The same thing
happens for the restrictions of $i^{\wedge}$ to $CF(N,L_{i})$. This
shows that up to this ambiguity given by multiplication with some
$T^{a}\in\mathcal{A}$ the relevant maps are chain homotopic and
concludes the proof of Theorem \ref{cor:exact-tri-explicit}.

\section{Quantum homology and the proofs of
  Theorems~\ref{thm:quantum_h} and~\ref{t:2-end-split}}
\label{sec:quantum}
The arguments in this section use the machinery developed in the last
section together with some specific properties of quantum homology
again adapted to the case of Lagrangians with cylindrical ends.  An
important additional ingredient in these proofs is the homological
injectivity induced by the inclusion $\La\to \mathcal{A}$ as proved in
Lemma \ref{lem:alg-inclusion}.

\subsection{Quantum homology for Lagrangians with cylindrical ends}
\label{sbsb:PSS-cob}

We first discuss the definition of quantum homology in this context and
then will see how the PSS-type comparison morphisms between quantum
homology and Floer homology (recalled in \S\ref{sbsb:pss}) adapt to
this setting. 

Let $\overline{W} \subset \widetilde{M}$ be a monotone Lagrangian with
cylindrical ends and let $S$ be a union of some of its ends. In other
words, assume the ends of $\overline{W}$ are
$$E^{-}_{R_-}(\overline{W}) = 
\coprod_{j=1}^{k_-} (-\infty, R_-] \times \{a^-_j\} \times L^{-}_j,
\qquad E^+_{R_+}(\overline{W}) = \coprod_{i=1}^{k_+} [R_+, \infty)
\times \{a^+_i\} \times L^{+}_i$$ then $$S=\bigcup_{j\in J_{-}}
\{a_{j}^{-}\}\times L^{-}_{j}\cup \bigcup_{i\in J_{+}}
\{a_{i}^{+}\}\times L^{+}_{i}$$ where $J_{-}\subset \{1,\ldots,
k_{-}\}$ and $J_{+}\subset\{ 1,\ldots, k_{+}\}$.

The quantum homology $QH(\overline{W}, S)$ is defined as follows.  Fix
$\epsilon>0$ and put $W= \overline{W}|_{[R_- - \epsilon, R_+ +
  \epsilon] \times \mathbb{R}}$, so that $W$ is a compact manifold
with boundary
$$\partial W = \Bigl(\coprod_{j=1}^{k_-} \{(R_{-}-\epsilon, a^{-}_j)\} 
\times L_j^-\Bigr) \; \coprod \; \Bigl(\coprod_{i=1}^{k_+}
\{(R_{+}+\epsilon, a^{+}_i)\} \times L_i^+\Bigr)$$
Let $S'$ be the part of the boundary of $W$ that corresponds to $S$: 
$$S' = \Bigl(\coprod_{j\in J_{-}} \{(R_{-}-\epsilon, a^{-}_j)\} 
\times L_j^-\Bigr) \; \cup \; \Bigl(\coprod_{i\in
  J_{+}}\{(R_{+}+\epsilon, a^{+}_i)\} \times L_i^+\Bigr)$$ Choose a
Morse function $\widetilde{f}: W \longrightarrow \mathbb{R}$ together
with a Riemannian metric $(\cdot, \cdot)$ and an almost complex
structure $\widetilde{J}$ on $\widetilde{M}$.  We require the function
$\widetilde{f}$ to be so that its negative gradient $-\nabla
\widetilde{f}$ is transverse to $\partial W$ and moreover it points
outside of $W$ along $S'$ and inside $W$ along $\partial W \setminus
S'$.  We also require $\widetilde{J}$ to be so that the projection
$\pi$ is holomorphic outside a compact set $K\subset [R_- -
\epsilon/2, R_+ + \epsilon/2] \times \mathbb{R}\times M$.  Denote by
$\mathscr{D}_{S}=(\widetilde{f}, (\cdot, \cdot), \widetilde{J})$ our
data.

\begin{prop} \label{prop:quantum-relative} If the data
   $\mathscr{D}_{S}$ is generic, then the pearl complex $
   \mathcal{C}(\mathscr{D}_{S})$ is well-defined by the same
   construction as the one recalled in ~\S\ref{sbsb:lag-qh}.  The
   resulting quantum homology does not depend, up to canonical
   isomorphism, on the choice of data $\mathscr{D}_{S}$ nor on the
   choice of $\epsilon$ and $R_{+}$, $R_{-}$ above. We denote the
   resulting homology by $QH(\overline{W},S)$.
\end{prop}

Similarly to the conventions in \S\ref{sbsb:lag-qh} we will denote by
$QH(\overline{W},S;\mathcal{A})$ the homology of the complex
$\mathcal{C}(\mathscr{D}_{S})\otimes_{\La} \mathcal{A}$.

\begin{proof} Recall that the relevant
   pearly trajectories are composed of flow lines of $-\nabla
   \widetilde{f}$ and $\widetilde{J}$-holomorphic disks. By Lemma
   \ref{lem:open-mapping} and our assumption on $\widetilde{J}$,
   there are no pseudo-holomorphic disks with boundary on
   $\overline{W}$ with non-constant projection to $\mathbb{R}^2$ that
   reach the complement of $K$.  In view of the fact that $-\nabla
   \widetilde{f}$ is transverse to $\partial W$ we deduce that all
   pearly trajectories that originate and end at critical points of
   $\widetilde{f}$ can not reach the boundary of $W$.  This
   immediately implies that the complex $\mathcal{C}(\mathscr{D}_{S})$
   is well defined and indeed a chain complex. The same argument also
   applies to show the rest of the statement.
\end{proof}

 The following lemma will be useful later in the paper.

\begin{lem}\label{lem:exact}
   Assume $\overline{W}$ is as in Proposition
   \ref{prop:quantum-relative}.  Pick a union of some of the ends of
   $\overline{W}$ and denote it by $A$.  Take also another union 
   $B$ of some of the ends of $\overline{W}$ so that $A\cap B=\emptyset$.
   There is a long exact sequence:
   \begin{equation}\label{eq:long-qu-rel}
      \to QH_{\ast}(A)\to QH_{\ast}(\overline{W}, B)
      \to QH_{\ast}(\overline{W},A\cup B)
      \to QH_{\ast-1}(A)\to ~.~
   \end{equation}
\end{lem}

A similar exact sequence also exists with coefficients in
$\mathcal{A}$.

\begin{proof} We put $S=A\cup B$ and we intend to construct a
   particular function $\widetilde{f}$ as the one appearing in the
   definition of $QH(\overline{W},S)$ but with a number of additional
   properties.  We use below the same notation as the one fixed before
   the statement of Proposition \ref{prop:quantum-relative}.  In
   particular, $J_+ \subset \{1, \ldots, k_+\}$, $J_- \subset \{1,
   \ldots, k_-\}$ are so that $$S' = \Bigl(\coprod_{j \in J_-}
   \{(R_{-}-\epsilon, a^{-}_j)\} \times L_j^-\Bigr) \; \cup \;
   \Bigl(\coprod_{i \in J_+} \{(R_{+}+\epsilon, a^{+}_i)\} \times
   L_i^+\Bigr)$$ is the part of the boundary of $W$ corresponding to
   $S$.  We also denote by $J'_{+}=\{1,\ldots, k_{+}\}\setminus J_{+}$
   and $J'_{-}=\{1,\ldots, k_{-}\}\setminus J_{-}$.
 
   Let $\widetilde{f}: W \longrightarrow \mathbb{R}$ be a Morse
   function with the following properties.
   \begin{equation} \label{eq:tilde-f}
      \begin{aligned}
         & \widetilde{f}(x,a_i^+, p) = f_i^+(p) + \sigma^+_i(x), \quad
         \sigma^{+}_{i}: [R_+ +\epsilon/4, R_+ + \epsilon]\to \R,
         \; p \in M, \; j = 1, \ldots, k_+, \\
         & \widetilde{f}(x,a_j^-, p) = f_j^-(p) + \sigma^-_j(x), \quad
         \sigma^{-}_{j}: [-R_- - \epsilon, -R_- - \epsilon/4]\to \R,
         \; p \in M,\; j=1, \ldots, k_-,
      \end{aligned}
   \end{equation}
   where $f_i^+: L_i^+ \longrightarrow \mathbb{R}$, $f_j^{-}:L_j^{-}
   \longrightarrow \mathbb{R} $ are Morse functions.  The functions
   $\sigma_i^+$, $\sigma_j^+$ are also Morse, each with a
   \emph{single} critical point and are required to satisfy the
   following conditions:
   \begin{enumerate}
     \item $\sigma^{+}_i(x)$ is a non-constant linear function for $x
      \in [R_{+} + 3\epsilon/4, R_{+} + \epsilon]$. Moreover, in this
      interval $\sigma^{+}_i$ is decreasing if $i\in J_{+}$ and
      increasing if $i \in J'_{+}$. Further, $\sigma^{+}_{i}$ has a
      single critical point at $R_{+}+\epsilon/2$ and this is of index
      $1$ if $i\in J_{+}$ and of index $0$ if $i\in J'_{+}$.
     \item $\sigma^{-}_j(x)$ is a non-constant linear function for $x
      \in [-R_{-}-\epsilon, -R_{-} - 3 \epsilon /4]$. Moreover, in
      this interval $\sigma^{-}_j$ is increasing if $j\in J_{-}$ and
      increasing if $j \in J'_{-}$; $\sigma^{-}_{j}$ has a single
      critical point at $R_{-}-\epsilon/2$ and this is of index $1$ if
      $j\in J_{-}$ and of index $0$ if $j\in J'_{-}$.
   \end{enumerate}
   A function $\widetilde{f}$ with these properties will be called
   \emph{adapted to the exit region $S$}.

   We now pick a Riemannian metric $(\cdot, \cdot)$ on $W$ which
   splits as $g^{\pm} \oplus dx^2$ on $W \cap \pi^{-1}([R_+ +
   \epsilon/4, R_+ + \epsilon] \times \mathbb{R})$ and $W\cap
   \pi^{-1}([-R_{-}-\epsilon, -R_{-}-\epsilon/4] \times \mathbb{R})$
   for some Riemannian metrics $g^{\pm}$ on the manifolds $\coprod_i
   L^+_i$ and $\coprod_j L^{-}_j$.  We call such a metric
   \emph{adapted} to the ends of $\widetilde{W}$.  Finally we also
   pick (a time independent) almost complex structure $\widetilde{J}$
   on $\widetilde{M}$ such that $\pi$ is $(\widetilde{J},
   i)$-holomorphic outside a compact set contained in $\widetilde{M}
   \setminus \pi^{-1}([R_{-}-\epsilon/4, R_{+}+\epsilon/4] \times
   \mathbb{R})$.

   Let now $I_{-}, I_{+}$ be index sets so that
   $$A'= \Bigl(\coprod_{j \in I_-} \{(R_{-}-\epsilon, a^{-}_j)\} 
   \times L_j^-\Bigr) \; \cup \; \Bigl(\coprod_{i \in I_+}
   \{(R_{+}+\epsilon, a^{+}_i)\} \times L_i^+\Bigr)$$ corresponds to
   $A$ and let $U(A')$ be a tubular neighborhood of $A'$ in $W$ given
   by
   $$U(A')=\Bigl(\coprod_{j \in I_-} [R_{-}-\epsilon, R_{-}-5\epsilon/8]
   \times \{ a^{-}_j)\} \times L_j^-\Bigr) \; \cup \; \Bigl(\coprod_{i
     \in I_+} [R_{+}+5\epsilon/8, R_{+}+\epsilon]\times\{ a^{+}_i)\}
   \times L_i^+\Bigr)~.~$$

   We now let $V=W\setminus U(A')$ and also denote $$A''=
   \Bigl(\coprod_{j \in I_-} \{(R_{-}-\epsilon/2, a^{-}_j)\} \times
   L_j^-\Bigr) \; \cup \; \Bigl(\coprod_{i \in I_+}
   \{(R_{+}+\epsilon/2, a^{+}_i)\} \times L_i^+\Bigr)~.~$$

   We assume the various choices made are generic so that the pearl
   complexes $\mathcal{C}(W, \widetilde{f},\widetilde{J})$,
   $\mathcal{C}(A'', \widetilde{f}|_{A''}, \widetilde{J})$ and
   $\mathcal{C}(V, \widetilde{f}|_{V}, \widetilde{J})$ are well
   defined.  These three complexes are related by an obvious short
   exact sequence:

   $$0 \to \mathcal{C}(V, \widetilde{f}|_{V}, \widetilde{J})\to
   \mathcal{C}(W, \widetilde{f},\widetilde{J})\to\mathcal{C}(A'',
   \widetilde{f}|_{A''}, \widetilde{J})\to 0~.~$$

   The claim now follows by noticing that $\mathcal{C}(A'',
   \widetilde{f}|_{A''}, \widetilde{J})$ is isomorphic to a pearl
   complex associated to $A$ with a shift in degree by one,
   $H(\mathcal{C}(V, \widetilde{f}|_{V},
   \widetilde{J}))=QH(V'',B)=QH(\overline{W},B)$ and, by definition,
   $H(\mathcal{C}(W,
   \widetilde{f},\widetilde{J}))=QH(\overline{W},A\cup B)$.

\end{proof}

\begin{rem}
   We will mainly apply the construction above to Lagrangians
   $\overline{V}$ that are the $\R$-extensions of Lagrangian
   cobordisms $V$. In this case we denote $QH(\overline{V}, S)$ by
   $QH(V,S)$ and similarly when working over $\mathcal{A}$.
\end{rem}

\subsection{The PSS isomorphism for Lagrangians with cylindrical
  ends} \label{sbsb:pss-cyl}
  
Let $\overline{W}\subset \widetilde{M}$ be a Lagrangian with
cylindrical ends and assume that $S$ is a union of some of its ends as
in \S\ref{sbsb:PSS-cob}. The choice of $S$ determines a path
component $c_S \in \pi_{0}(\mathcal{H}(\overline{W},\overline{W}))$ in
the following way.  Consider a perturbation function $f$, as at point
C in~\S\ref{sbsb:hf-lag-cyl}, so that:
\begin{enumerate}
  \item for each positive end $i$ of $\overline{W}$, the constant
   $\alpha^{+}_{i}$ is negative if the end is in $S$ and is positive
   if the end $i$ is not in $S$.
  \item for each negative end $j$, the constant $\alpha^{-}_{j}$ is
   positive if the end is in $S$ and is negative if the end $j$ is not
   in $S$
\end{enumerate}
and put $c_S:=[f]$.

The purpose of this subsection is to discuss the proof of the
following result.
   
\begin{prop}\label{prop:PSS-cyl}
   There exists a PSS-type isomorphism over $\mathcal{A}$
   $$\overline{PSS}_{S}:HF(\overline{W},\overline{W};c_S)
   \longrightarrow QH(\overline{W}, S;\mathcal{A})~.~$$
\end{prop}
  
\begin{proof} With the notations in the proof of
   Lemma~\ref{lem:exact}, let $\widetilde{f}: W\to \R$ be adapted to
   the exit region $S$.  Extend the function $\widetilde{f}$ to the
   whole of $\overline{W}$ by using the formulas in~\eqref{eq:tilde-f}
   and extending the functions $\sigma_i^+(x)$ linearly beyond
   $R_{+}+\epsilon$ and also extending linearly the functions
   $\sigma_j^{-}(x)$ linearly below $R_{-}-\epsilon$.
 
   Fix a Darboux-Weinstein neighborhood $\mathcal{U}$ of
   $\overline{W}$ in $\widetilde{M}$ which is symplectomorphic to a
   neighborhood of the zero-section in $T^*\overline{W}$. Due to the
   cylindrical ends of $\overline{W}$ we can choose $\mathcal{U}$ so
   that $\pi(\mathcal{U}) \cap ((-\infty,-R_-] \times \mathbb{R})$
   contains the strips $\cup_i (-\infty,R_-] \times (a^-_i -\delta,
   a^-_i + \delta)$ for some $\delta>0$ and similarly for
   $\pi(\mathcal{U}) \cap ([R_+, \infty) \times \mathbb{R})$.

   After multiplying $\widetilde{f}$ by a small positive constant we
   may assume that $\widetilde{f}$ has a small differential $d
   \widetilde{f}$ so that the graph of $d \widetilde{f}$ fits inside
   of $\mathcal{U}$.  Recall that $\widetilde{f}$ has a linear
   horizontal component along the ends. Extend the function
   $\widetilde{f}$ first to a function on $\mathcal{U}$ using the
   identification of $\mathcal{U}$ with a neighborhood of
   $\overline{W} \subset T^* \overline{W}$ (making it constant along
   each cotangent fibre), and then to the rest of $\widetilde{M}$ so
   that the resulting function $H_{\widetilde{f}}$ vanishes outside a
   slightly larger neighborhood $\mathcal{U}'$ of $\mathcal{U}$.

   Pick a generic autonomous almost complex structure $\widetilde{J} \in
   \widetilde{\mathcal{J}}_{B}$ with $B$ a compact set sufficiently
   large so that $\overline{W}$ is cylindrical outside $B$. We will
   also assume $R_{+}$ and $|R_{-}|$ sufficiently big so that
   $[R_{+},\infty)\times \R$ as well as $(-\infty, R_{-}]\times \R$
   are both outside $B$.

   The linearity of the function $\widetilde{f}$ at infinity
   immediately shows that $\overline{W}_{1}
   :=\phi^{H_{\widetilde{f}}}_{1}(\overline{W})$ and $\overline{W}$
   are cylindrically distinct at infinity, that for a generic choice
   of $\widetilde{f}$ the Floer complex
   $CF(\overline{W}_{1},\overline{W};\widetilde{J})$ is well defined
   and that, by Proposition~\ref{prop:invariance-Fl}, its homology is
   canonically identified with $HF(\overline{W},\overline{W};c_{S})$
   (see~\S\ref{sbsb:hf-lag-cyl} and in
   particular~\eqref{eq:Floer-cplx-cyl}).
 
   We will also need below another function $\widetilde{f}':W\to \R$
   with the same properties from Lemma \ref{lem:exact} as
   $\widetilde{f}$ except that the value of $\epsilon$ used to
   construct $\widetilde{f}'$ is fixed to be $\epsilon'=\epsilon/2$.
   We also fix a metric $(\cdot,\cdot)$ on $W$ that is adapted to the
   ends of $\overline{W}$ (in the sense indicated in the proof of
   Lemma \ref{lem:exact}) and so that the pearl complex
   $\mathcal{C}(\mathscr{D}_{S})$ is defined for
   $\mathscr{D}_{S}=(\widetilde{f}',(\cdot,\cdot), \widetilde{J})$.
   We will work in this proof only over $\mathcal{A}$ so that the
   homology computed by $\mathcal{C}(\mathscr{D}_{S})$ is
   $QH(\overline{W},S;\mathcal{A})$.

   We now intend to consider the moving boundaries PSS - chain
   morphism - see \S\ref{sbsb:pss}:
   $$\widehat{PSS}:\mathcal{C}(\mathscr{D}_{S})\to 
   CF(\overline{W}_{1},\overline{W}; \widetilde{J})~.~$$ In fact, the
   only issue that is specific to our cylindrical at infinity setting
   is again whether the necessary compactness is satisfied by the
   moduli spaces used to define this map. If this is the case, the
   rest of the construction takes place like in the compact setting.
   In particular, we also obtain that this morphism induces an
   isomorphism in homology.

   Thus, our focus will now be to describe the relevant moduli spaces
   and indicate the reason why compactness hols.

   Let $x\in \Crit(\widetilde{f'})$ and let $a\in \overline{W}_{1}\cap
   \overline{W}$ be an intersection point.  Consider a $C^{\infty}$
   function $\beta: \R\to [0,1]$ so that $\beta(s)=0$ for $s\leq 0$,
   $\beta(s)=1$ for $s\geq 1$ and $\beta$ is strictly increasing on
   $(0,1)$.  Put
   $\overline{W}_{s}=\phi^{H_{\widetilde{f}}}_{\beta(s)}$.

   We consider the moduli space $\mathcal{M}(x,a;\widetilde{J})$
   consisting of pairs $(v,u)$ where $v$ is a string of pearls on
   $\overline{W}$ formed by flow lines of $-\nabla \widetilde{f}'$
   (the first one originating at $x$) alternating with
   $\widetilde{J}$-holomorphic disks in $\widetilde{M}$ with boundary
   on $\overline{W}$ (see~\cite{Bi-Co:rigidity, Bi-Co:Yasha-fest}) so
   that the last flow line in the string $v$ ends at a point $b\in
   \overline{W}$. This point $b$ is the starting point of a solution
   $u:[0,1]\times \R\to \widetilde{M}$, of the Cauchy-Riemann equation
   $\overline{\partial}_{\widetilde{J}}u=0$ subject to the following
   moving boundary condition:
   \begin{equation}\label{eq:mov-bdr-pss}
      u(0,s)\in \overline{W}_{s} \ , \ u(1,s)\in \overline{W}.
   \end{equation}
   By ``starting point'' we mean that $\lim_{s\to-\infty} u(-,s)=b$.
   We also have $\lim_{s\to\infty}u(-,s)=a$.
   
   It is easy to see that the needed compactness properties for the
   definition of $\widetilde{PSS}$ as well as that of its
   (homological) inverse and all the other relevant properties are an
   immediate consequences of the following result.

   \begin{lem}\label{lem:compact-pss}
      With the notation above $$\textnormal{image\,}(\pi\circ
      u)\subset B\cup (([R_{-}-\epsilon/2, R_{+}+\epsilon/2]\times
      \R)\cap \mathcal{U}'))~.~$$
   \end{lem}
   \begin{proof}[Proof of Lemma~\ref{lem:compact-pss}]
      We will prove that $\textnormal{image\,}(\pi\circ u)\subset
      B\cup (((-\infty, R_{+}+\epsilon/2]\times \R)\cap
      \mathcal{U}')$. The fact that $\textnormal{image\,}(\pi\circ
      u)\subset B\cup (([R_{-}-\epsilon/2, \infty)\times \R)\cap
      \mathcal{U}')$ can be proved by an analogous argument.
      
      Put $P=\{R_{+}+\epsilon/2\}\times a^{+}_{i}$ and notice that
      this is a point of intersection of $l_{s}=\pi(\overline{W}_{s})$
      and $l=\pi(\overline{W})$ for all $s$, and moreover the
      intersection is transverse for $s>0$. This is because $P$ is a
      critical point for the function $\sigma_{i}^{+}$. Without loss
      of generality we assume that $P$ is as depicted in
      Figure~\ref{fig:PSS-quadrants} (the other cases are analogous).
      Denote by $Q_1, Q_2$ the connected components of $\mathbb{C}
      \setminus \bigcup_{s \in [0,1]} l_s$ corresponding to the second
      and fourth quadrants respectively near the intersection points
      $P$, where $l$ plays the role of the $x$-axis and $l_1$ the role
      of the $y$-axis. (Thus $Q_1$ is ``above'' $P$ and $Q_2$
      ``below'' $P$.) Denote by $Q_+, Q_- \subset \mathbb{C} \setminus
      (l \cup l_1)$ the connected components corresponding to the
      first and fourth quadrants respectively (so that $Q_+$ is on the
      ``right'' of $P$ and $Q_-$ on its ``left''). Note that $Q_1,
      Q_2, Q_+$ are unbounded.

      \begin{figure}[htbp] 
         \epsfig{file=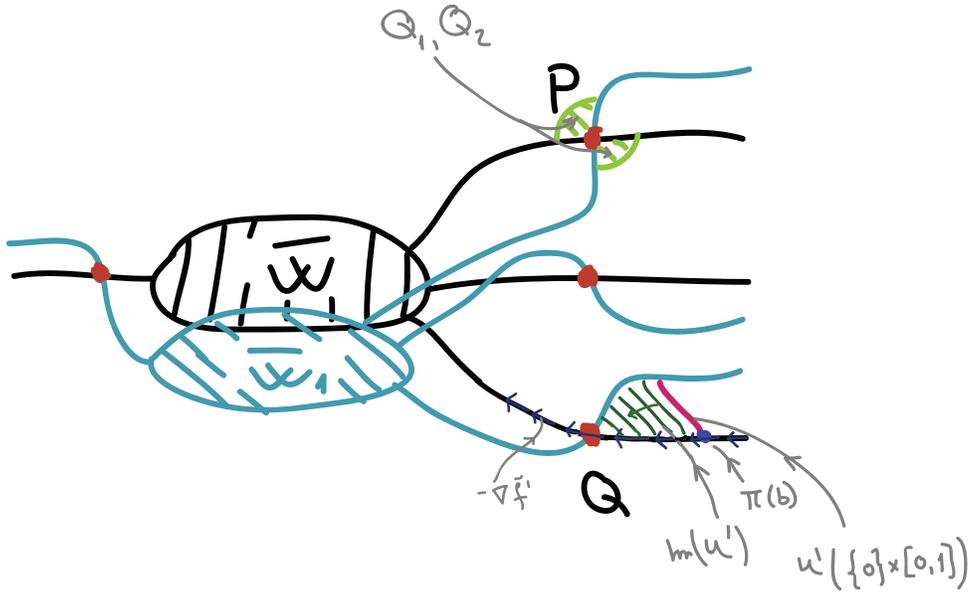, width=.80\linewidth}
         \caption{\label{fig:PSS-quadrants} The cobordisms
           $\overline{W}_{1}$ and $\overline{W}$ . The quadrants
           $Q_{1}$, $Q_{2}$ around $P$.  Also appear the image of
           $u'$, the points $Q=\pi_{a}$ and $\pi(b)$ as well as the
           direction of the flow $-\nabla \widetilde{f}'$ when
           projected on $\R^{2}$.}
      \end{figure}

      Put $u'=\pi\circ u$. First note that $u' (\textnormal{Int\,}
      (\mathbb{R} \times [0,1]) \cap (Q_1 \cup Q_2) = \emptyset$. This
      follows from the open mapping theorem and the fact that $Q_1,
      Q_2$ are unbounded, in a similar way to the arguments
      in~\S\ref{subsec:compact} (see also the end of the proof of
      Lemma~\ref{cor:Floer-structure2}).

      Next note that it is impossible to have an interior point $z_0
      \in \mathbb{R} \times (0,1)$ with $u'(z_0) = P$. Indeed, if such
      a $z_0$ would exist then by the open mapping theorem the image
      of $u$ would intersect $Q_1$ (and $Q_2$) which we have just seen
      is impossible.

      Next we claim that it is impossible to have two points $z_-, z_+
      \in \mathbb{R} \times [0,1]$ with $u'(z_-) \in Q_-$ and $u'(z_+)
      \in Q_+$ (i.e. \RB{the image of} $u'$ cannot intersect both
      $Q_-$ and $Q_+$). Indeed, if such points $z_{\pm}$ would exist,
      then connect them by a path $\gamma \subset \mathbb{R} \times
      [0,1]$ such that $\gamma$ lies in $\textnormal{Int\,}
      (\mathbb{R} \times [0,1])$ except of possibly at its end points
      (in case $z_-$ or $z_+$ are on the boundary). As the \RB{image
        of} $u'$ avoids both $Q_1$ and $Q_2$ it follows that there is
      an interior point $z' \in \gamma$ with $u'(z') = P$. But we have
      seen that this is impossible. This proves the claim.

      It follows from the above that if the image of $u$ does not
      satisfy the claim of the lemma, then the whole image of $u'$ is
      contained either in $((-\infty, R_{-}-\epsilon/2] \times
      \mathbb{R}$ or in $[R_{+}+\epsilon/2,\infty))\times \R$. This
      means that there is some point $Q$ of the form
      $Q=\{R_{+}+\epsilon/2\}\times a^{+}_{i_{0}}$ or $Q=
      \{R_{-}-\epsilon/2\}\times a^{-}_{j_{0}}$ so that $\pi(a)=Q$.
      To simplify the discussion assume that we are in the first case,
      the second one is treated in a perfectly similar fashion. The
      fact that $\pi(a)=Q$ implies that the strip $u'$ ``arrives'' at
      $Q$ and this is easily seen to imply that
      $\textnormal{ind}_{\sigma^{+}_{i_{0}}}(Q)=0$.  Moreover,
      $\pi(b)$ can be written as $\pi(b)=(b', a^{+}_{i_{0}})$ with
      $b'\geq R_{+}+\epsilon/2$.  At this point we use the particular
      form of the function $\widetilde{f}'$: as the function
      $\sigma^{+}_{i_{0}}$ used in the construction of
      $\widetilde{f}'$ is increasing on the interval
      $[R_{+}+3\epsilon/8, +\infty)$ (because $\epsilon'=\epsilon/2$)
      and, as the metric $(\cdot,\cdot)$ is adapted to the ends of
      $\overline{W}$, we deduce that there can not be any flow lines
      of $-\nabla (\widetilde{f}')$ that come from the interior of the
      region $\overline{W}\cap \pi^{-1}([R_{-}-\epsilon/2,
      R_{+}+\epsilon/2])$ and reach the point $b$.  Clearly, by Lemma
      \ref{lem:open-mapping}, there can not be any
      $\widetilde{J}$-holomorphic disk with boundary on $\overline{W}$
      reaching $b$ either.  Taken together, these two facts contradict
      our assumption on the image of $u$ and this concludes the proof
      of the lemma.
   \end{proof}
   The proof of Proposition~\ref{prop:PSS-cyl} follows now by
   standard arguments.
\end{proof}

\subsection{Proof of Theorem \ref{thm:quantum_h}} \label{sb:elem-cob}

Recall that we are considering the monotone Lagrangian  cobordism $(V; L', L)$
and we intend to compare the quantum homologies of the two ends.

\begin{proof}
   Let $V'$ be a (non-compactly supported) small Hamiltonian
   deformation of $V$ so that $V'$ is cylindrically distinct from $V$
   and the negative and the positive ends of $V'$ are below those of
   $V$ in the sense that they have lower imaginary coordinates in the
   plane than the ends of $V$ - see Figure \ref{fig:CobordismsQuant}.
   By Proposition~\ref{prop:PSS-cyl} the Floer homology associated to
   the two Lagrangian cobordisms, $\overline{V}$ and $\overline{V}'$
   satisfies:
   \begin{equation}\label{eq:iso-quantumFloer}
      HF(\overline{V'},\overline{V})\cong HF(\overline{V},\overline{V};c_L)
      \cong QH(V, L;\mathcal{A}),
   \end{equation}
   where $c_L \in \pi_{0}(\mathcal{H}(\overline{V},\overline{V}))$ is defined as
   at the beginning of~\S\ref{sbsb:pss-cyl}.
   \begin{figure}[htbp]
      \epsfig{file=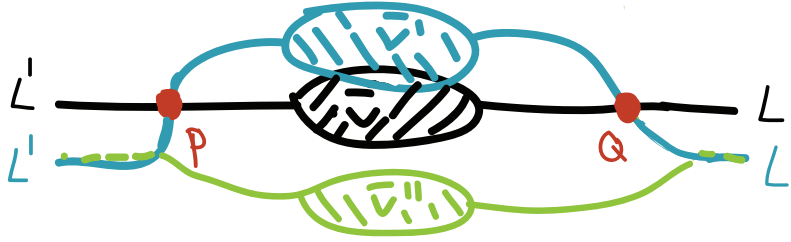, width=0.65\linewidth}
      \caption{The elementary cobordism $\overline{V}$, its (non-isotopic)
        deformation $\overline{V}'$ together with one horizontally isotopic
        deformation of $\overline{V}'$, $\overline{V}''$.  We have
        $QH(V,L;\mathcal{A})\cong HF(\overline{V}',\overline{V})\cong
        HF(\overline{V}'',\overline{V})=0$.}
      \label{fig:CobordismsQuant}
   \end{figure}
  
   It is clear that, as in Figure~\ref{fig:CobordismsQuant}, we may
   find $\overline{V}''$ horizontally isotopic to $\overline{V}'$ and
   disjoint from $\overline{V}$.  Thus,
   $HF(\overline{V}',\overline{V})\cong
   HF(\overline{V}'',\overline{V})=0$.  But now, from Lemma
   \ref{lem:exact}, we also have the long exact sequence:
   $$\to QH(L;\mathcal{A})\to QH(V;\mathcal{A})\to 
   QH(V,L;\mathcal{A})\to$$ as well as a similar exact sequence over
   $\La$.  From the exact sequence over $\mathcal{A}$ we deduce
   $QH(L;\mathcal{A})\to QH(V;\mathcal{A})$ is an isomorphism.  Recall
   from Lemma \ref{lem:alg-inclusion} that the map $QH(-)\to
   QH(-;\mathcal{A})$ is injective. Thus, $QH(V,L)=0$ and therefore
   $QH(L)\to QH(V)$ is also an isomorphism.  For further use, this
   arrow can be viewed, as in the Morse case, as induced by the
   inclusion $u_{1}: L\to V$. Clearly, a similar argument is valid for
   $QH(L')\to QH(V)$ with respect to the inclusion $u_{2}:L'\to V$.
   This proves the first part of the statement of Theorem
   \ref{thm:quantum_h}.  The next step is to show that we can find an
   isomorphism of $QH(L)$ and $QH(L')$ that also preserves the quantum
   product.  For this we consider the maps $p_{1}: QH(V;L\cup L')\to
   QH(L)$ the dual of $(u_{1})_{\ast}$ and $p_{2}:QH(V;L\cup L')\to
   QH(L')$ the dual of $(u_{2})_{\ast}$.  Both are again isomorphisms
   and it is an easy exercise to see that they also are algebra maps
   (with respect to the quantum product).  All these map are actually
   defined over $\La^{+}=\Z_{2}[t]$ but not necessarily isomorphisms
   over $\La^{+}$.
   
   Next we show that the morphisms induced by the inclusions $u_1$,
   $u_{2}$ on $H_{1}(-;\Z_{2})$ have the same image in
   $H_{1}(V;\Z_{2})$ if we also assume that $L$ and $L'$ are wide. For
   this it is enough to show that the composition $c_{1}:
   H_{1}(L;\Z_{2})\to H_{1}(V;\Z_{2})\to H_{1}(V, L';\Z_{2})$ vanishes
   as well as the other composition, obtained by switching $L$ and
   $L'$.  By duality, the vanishing of $c_{1}$ is equivalent to the
   vanishing of the composition $c_{1}':H_{n}(V,L;\Z_{2})\to
   H_{n}(V,L\cup L';\Z_{2})\to H_{n-1}(L;\Z_{2})$.  We now notice the
   existence of two maps $H_{n}(V,L;\Z_{2})\to QH(V,L)$ and
   $H_{n}(V,L\cup L';\Z_{2})\to QH(V,L\cup L')$ defined as follows.
   Assume that $f$ is a Morse function on $V$ adapted to the exit
   region $L\cup L'$. Then we may assume that $f$ has a single maximum
   $w$. As the map $p_{2}$ is an isomorphism and is defined over
   $\La^{+}$, it follows that $[w]\not=0\in QH(V, L\cup L')$.  But
   this means that all the Morse cycles of $f$ in dimension $n$ are
   also pearl cycles. A similar argument applies to
   $H_{n}(V,L;\Z_{2})$. For the same reasons, there is as well a map
   $H_{n-1}(L;\Z_{2})\to QH(L)$ which is well defined because $L$ is
   not narrow. It is immediate to see that the resulting diagram
   commutes:
   \begin{equation}\label{eq:commut-simple}
      \xymatrix@-2pt{
        H_{n}(V,L;\Z_{2}) \ar[d]\ar[r]    & 
        H_{n}(V,L\cup L';\Z_{2}) \ar[r]\ar[d]   &  
        H_{n-1}(L;\Z_{2}) \ar[d]\\
        QH(V,L) \ar[r]    &   QH(V,L\cup L') \ar[r]  &   QH(L) }
   \end{equation} 
   The top row composition here is $c_{1}'$.  But now $QH(V,L)=0$ and,
   as $L$ is wide, the rightmost vertical arrow is an injection. This
   means that $c_{1}'$ vanishes and as a similar argument applies to
   $H_{1}(L';\Z_{2})\to H_{1}(V;\Z_{2})\to H_{1}(V, L;\Z_{2})$ this
   shows that the two inclusions $u_{1}$, $u_{2}$ have the same image
   in homology.  To end the proof we now specialize to $n=2$. Notice
   that the map $c_{1}':H_{2}(V,L;\Z_{2})\to H_{1}(L;\Z_{2})$ is
   easily identified with the connectant morphism in the long exact
   sequence of the pair $(V,L)$. Thus the next map in this exact
   sequence $H_{1}(L;\Z_{2})\to H_{1}(V;\Z_{2})$ is injective.  A
   similar argument applies to the inclusion $L'\to V$.
\end{proof}

\begin{rem} Similar methods easily imply also that the quantum module
   structures on $QH(L)$ and $QH(L')$ (over $QH(M)$) are isomorphic.
   Further, it is also possible to show that, for $n=2$, the
   enumerative invariants over $\Z_{2}$ that were introduced
   in~\cite{Bi-Co:lagtop} coincide for $L$ and $L'$.
\end{rem}

\subsection{Proof of Theorem \ref{t:2-end-split}} \label{sb:2-end-split} 

Here we assume that $(V; (L_{1},L_{2}), L)$ is a monotone cobordism so
that $QH(L)$ is a field and both $L_{1}$ and $L_{2}$ are not narrow.
The family $L, L_{1}, L_{2}$ is assumed uniformly monotone. We intend
to show the rank inequality (\ref{eq:rk}).

\begin{proof}
   The first part of the argument is based on the existence of the
   diagram:
   \begin{equation}\label{eq:two-exact-q}\xymatrix@-2pt{
        QH_{\ast}(V,L)\ar[r]^{j_{1}}\ar[d]_{j_{2}}& QH_{\ast}(V, L_{1}\cup L) 
        \ar[d]_{\eta_{1}}\ar[r]^{s_{1}}&QH_{\ast-1}(L_{1})
        \ar[d]^{k_{1}}\ar[r]^{l_{1}}&QH_{\ast-1}(V,L)\\
        QH_{\ast}(V, L_{2}\cup L )\ar[d]_{s_{2}}\ar[r]^{\eta_{2}} & 
        QH_{\ast-1}(L)\ar[r]^{i_{2}}\ar[d]_{i_{1}} & QH_{\ast-1}(V,L_{2})
        \ar[d]^{r_{2}}&\\
        QH_{\ast-1}(L_{2})\ar[d]_{l_{2}}\ar[r]_{k_{2}}& QH_{\ast-1}(V,L_{1}) 
        \ar[r]_{r_{1}}& QH_{\ast-1}(V,L_{1}\cup L_{2}) & \\
        QH_{\ast-1}(V,L) & & &}
   \end{equation} 
   where the columns and rows are exact. Here $i_{1}$, $i_{2}$,
   $j_{1}$, $j_{2}$, $k_{1}$, $k_{2}$, $l_{1}$, $l_{2}$, $r_{1}$,
   $r_{2}$ are induced by inclusions and $\eta_{1}$, $\eta_{2}$ and
   $s_{1}$, $s_{2}$ are connecting morphisms in the long exact
   sequences associated to these inclusions.  A further important
   remark is that, in appropriate degrees, $\eta_{1}$ is dual to
   $i_{2}$, $\eta_{2}$ is dual to $i_{1}$, $s_{1}$ is dual to $k_{1}$
   and $s_{2}$ is dual to $k_{2}$ - the duality here is similar to
   Poincar\'e duality (for pearl homology it appears in~\S 4.4
   of~\cite{Bi-Co:rigidity}).  The existence of this commutative
   diagram is shown in a way similar to the proof of Lemma
   \ref{lem:exact}. Note that Diagram~(\ref{eq:two-exact-q}) exists,
   together with the dualities indicated above, also with coefficients
   in $\mathcal{A}$.

   The next step is to notice the commutativity of the diagram
   \begin{equation}\label{eq:inclusion-Floer}
      \xymatrix@-2pt{
        QH(L;\mathcal{A}) \ar[r]^{i_{1}}\ar[d]_{PSS}& 
        QH(V,L_{1};\mathcal{A})\ar[d]^{PSS'}\\
        HF(L,L)\ar[r]_{\phi_{V}}& HF(L,L_{2})}
   \end{equation}
   up to multiplication by $T^{a}$ for some $a\in\R$.  We first
   describe the different morphisms showing up in this diagram and
   then we will justify its commutativity.
   
   Both $PSS$
   and $PSS'$ are  isomorphisms as explained below.     
   The morphism $PSS$ is just the
   Piunikin-Salamon-Schwarz-type isomorphism $QH(L;\mathcal{A})\to
   HF(L,L)$ as recalled in \S\ref{sbsb:pss}.  The morphism $PSS'$ is
   given by the composition:
   \begin{equation}\label{eq:comp2}
      QH(V,L_{1};\mathcal{A})
      \stackrel{\overline{PSS}_{L_{1}}}{\longrightarrow} 
      HF(\overline{V},\overline{V}; c_{L_{1}})\stackrel{\eta}{\to} 
      HF(\overline{V}',\overline{V}) \stackrel{\xi}{\to} HF (L,L_{2})~.~
   \end{equation} Here the
   first morphism $\overline{PSS}_{L_{1}}$ is the 
   PSS-type isomorphism discussed in
   Proposition \ref{prop:PSS-cyl}.     
   \begin{figure}[htbp]
      \begin{center}
         \epsfig{file=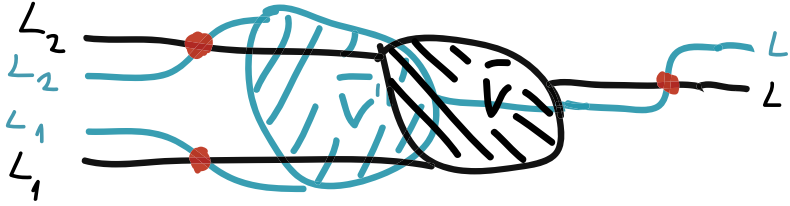, width=0.65\linewidth}
         \caption{\label{fig:CobTwoEnds1} The cobordism
           $\overline{V}'$ obtained by a Hamiltonian deformation
           associated to a small function $f:\overline{V}\to \R$
           adapted to the exit region $L_{1}$. We have
           $QH(V,L_{1};\mathcal{A})\cong
           HF(\overline{V}',\overline{V})$.}
      \end{center}
   \end{figure}
   The second isomorphism, $\eta$, follows from the definition of
   $HF(-,-)$ in \S\ref{sbsb:hf-lag-cyl} and
   Proposition~\ref{prop:invariance-Fl}.  The third isomorphism,
   $\xi$, is itself a composition of two isomorphisms
    $$HF(\overline{V}',\overline{V})\stackrel{\xi'}{\to} 
    HF(\overline{V}'',\overline{V})\stackrel{\xi''}{\to}
    HF(L,L_{2})~.~$$ Here $\xi'$ is provided (again via Proposition
    \ref{prop:invariance-Fl} ) by the fact that $\overline{V}'$ is
    horizontally isotopic to the cobordism $\overline{V}''$ in
    Figure~\ref{fig:CobTwoEnds2}.
     \begin{figure}[htbp]
       \begin{center}
          \epsfig{file=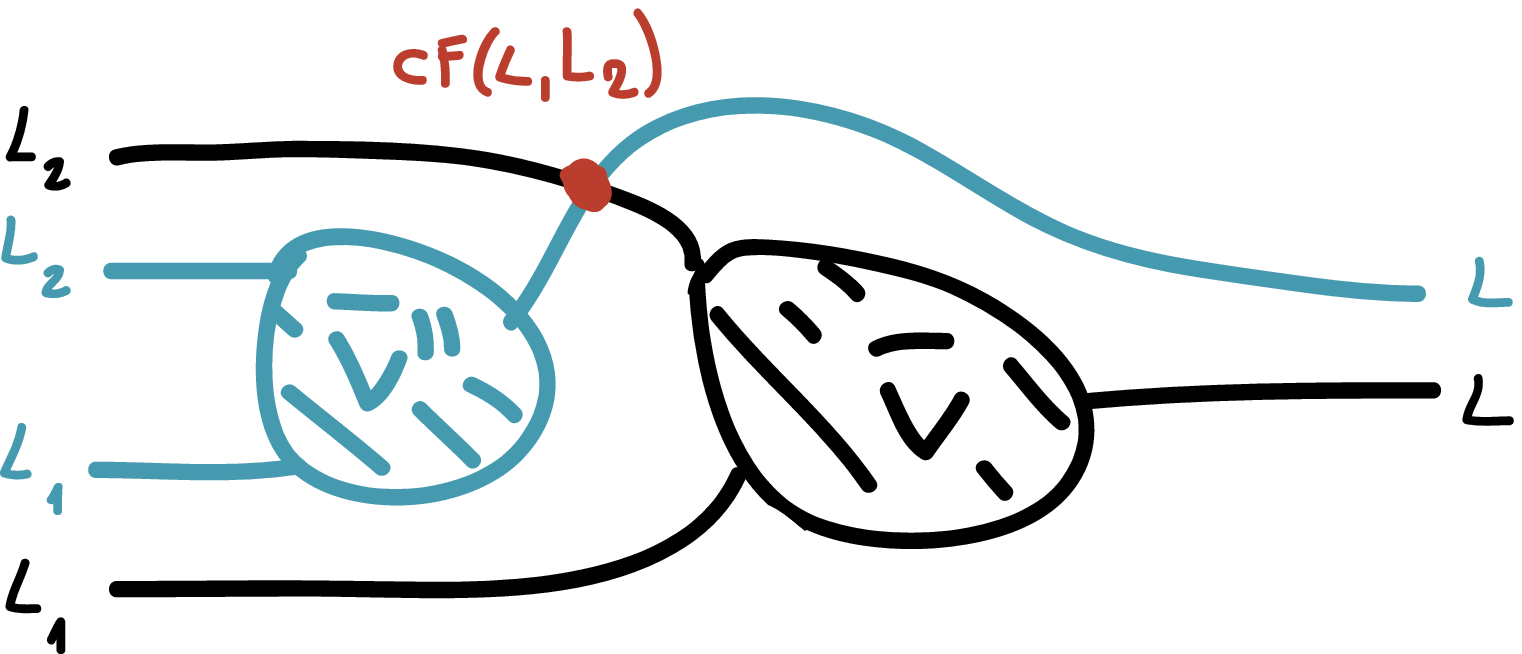, width=0.60\linewidth}
          \caption{\label{fig:CobTwoEnds2} The cobordism
            $\overline{V}''$ is isotopic to $\overline{V}'$. We have
            $HF(\overline{V}',\overline{V}) \cong
            HF(\overline{V}'',\overline{V})=HF(L,L_{2})$.}
       \end{center}
    \end{figure}
     As for $\xi''$, it is an
    identification $$\xi'':HF(\overline{V}'',\overline{V})=HF(L, L_{2})$$ that
    follows from the fact that $\pi(\overline{V}'')$ and $\pi(\overline{V})$
    intersect in a single point in the region where both $\overline{V}''$
    and $\overline{V}$ are just products between curves in the plane and,
    respectively, $L$ and $L_{2}$.  
         
    We now describe the map $\phi_{V}$. The construction of this map
    is very similar to the construction of the maps $h=
    [\bar{\phi}_{V}^{N}]$ and $m_{i}$ in Theorem
    \ref{cor:exact-tri-explicit}.  We first fix $L'\subset M$
    Hamiltonian isotopic to $L$ and transverse to $L,L_{1}, L_{2}$.
    We consider $\tilde{L}'=\lambda_{a, 3/2, k+1, 0}\times L'$ - see
    Figure \ref{fig:graph}.  Then, for appropriate almost complex
    structures, as in Lemma~\ref{cor:Floer-structure2}, the Floer
    complex $CF(\tilde{L}', \overline{V};\mathbf{J})$ is well defined
    and has the form: $$CF(\tilde{L}', \overline{V}; \mathbf{J}) =
    CF(L', L_{2};\mathbf{J})\oplus CF(L', L ; \mathbf{J})$$ for some
    $l_1, l_2 \in \Z$ and differential
    $$
    D =
    \begin{pmatrix}   d_{1} & \tilde{\phi}_{V} \\
       0 & d_{2}
    \end{pmatrix}~,~$$ where $d_1$ and $d_2$ are, up to sign, the
    Floer differentials of $CF(L',L_{2})$ and $CF(L', L )$
    respectively. See Figure~\ref{fig:simple-phi}.
   \begin{figure}[htbp]
      \begin{center}
         \epsfig{file=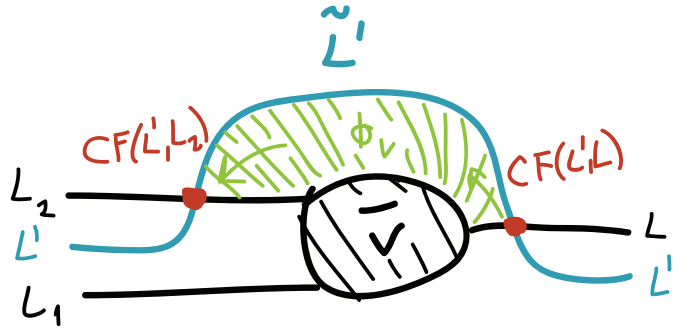, width=0.55\linewidth}
         \caption{\label{fig:simple-phi} The cobordisms $\overline{V}$
           and $\tilde{L}'$. The map $\phi_{V}$ counts the strips in
           {\color{green} green}.}
       \end{center}
    \end{figure}
  
    In the graded case there are some suspensions in the expression
    above - as in Lemma~\ref{cor:Floer-structure2} - but we neglect
    them here. Similarly, there are certain signs in the matrix above
    that we again neglect as we work over $\Z_{2}$. What matters here
    is that the upper left component of $D$ is a chain map
    $\tilde{\phi}_{V}: CF(L',L)\to CF(L',L_{2})$. We put
    $\phi_{V}=[\tilde{\phi}_{V}]$.  We notice that as in the proof of
    Theorem \ref{cor:exact-tri-explicit} this map is uniquely defined
    up to chain homotopy and multiplication by a an element $
    T^{a}\in\mathcal{A}$. In geometric terms, this map counts the
    Floer strips that project to the green strips in
    Figure~\ref{fig:simple-phi}.

    The next step is to justify the commutativity of
    Diagram~\ref{eq:inclusion-Floer}.  For this verification we will
    identify geometrically the maps $\phi_{V}$ and $i_{1}$ and will
    relate them to the construction of $PSS'$. The geometric part of
    this argument consists in composing the two isotopic cobordisms
    $\overline{V}'$ and $\overline{V}''$ from the
    Figures~\ref{fig:CobTwoEnds1} and~\ref{fig:CobTwoEnds2} with a
    cobordism of the form $\gamma\times L$ as in the
    Figure~\ref{fig:CobTwoEnds5}.  
    \begin{figure}[htbp]
      \begin{center}
         \epsfig{file=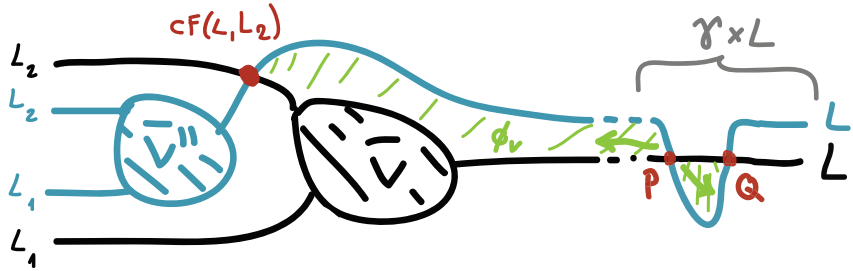, width=0.70\linewidth}
         \caption{\label{fig:CobTwoEnds5} The cobordism
           $\overline{W}''$ obtained by the extension of
           $\overline{V}''$ by $\gamma\times L$ and its intersections
           with $\overline{V}$. We have
           $HF(\overline{V}',\overline{V}) \cong
           CF(\overline{W}'',\overline{V})=(CF(L,L_{2})\oplus
           CF(L,L)[1]\oplus CF(L,L), D)$. In green the two
           non-internal components of $D$: $\phi_{V}$ (to the left)
           and $id_{CF(L,L)}[-1]$ (to the right).}
      \end{center}
   \end{figure}

   To be more precise, assume, without loss of generality, that the
   cylindrical positive end of both $\overline{V}'$ and
   $\overline{V}''$ coincide with $[1,+\infty)\times \{2\}\times L$.
   Assume also that the positive end of $\overline{V}$ coincides with
   $[1,+\infty)\times\{1\}\times L$.  Now take the curve $\gamma$ to
   be the graph of a function $g:[1,+\infty)\to \R$ so that $g$ is
   smooth, $g(t)=2$ for $t\in [1,2]\cup [4,+\infty)$, $g$ attains its
   minimum at the point $3$ with minimal value $g(3)=-1$ and $3$ is
   the single critical point of $g$ in the interval $(2,4)$. The curve
   $\gamma$ intersects (transversely) the curve $y=1$ in two points
   $P=(p,1)$ and $Q=(q,1)$ with $p<q$.  Finally, we put
   $\overline{W}'=(\overline{V}'\cap \pi^{-1}((-\infty, 1]\times
   \R))\cup \gamma \times L$ and similarly $\overline{W}''=
   (\overline{V}''\cap \pi^{-1}(-\infty, 1]\times \R))\cup \gamma
   \times L$.  Certainly, $\overline{W}'$ is horizontally isotopic to
   $\overline{W}''$ (and both are horizontally isotopic with
   $\overline{V'}$ and $\overline{V}''$). We will use the fact that
   the isotopy from $\overline{W}'$ to $\overline{W}''$ may be assumed
   constant on $\pi^{-1}([1,+\infty)\times \R)$.
   
    We use the two cobordisms $\overline{V}$ and $\overline{W}'$ to
    deduce the commutativity of the following diagram:
    \begin{equation}\label{eq:inclusion-Floer2}\xymatrix@-2pt{
        QH(L;\mathcal{A}) \ar[r]^{i_{1}}\ar[d]_{PSS}& 
        QH(V,L_{1};\mathcal{A})\ar[d]^{PSS''}\\
        HF(L,L)\ar[r]_{j}& HF(\overline{W}', \overline{V})}
   \end{equation}
   where $j$ is the map induced in homology by the inclusion of the
   subcomplex of $CF(\overline{W}',\overline{V})$ generated by the
   intersection points of $\overline{W}'$ and $\overline{V}$ that
   project onto $Q$; $PSS''$ is a composition like $\eta\circ
   \overline{PSS}_{L_{1}}$ from Equation (\ref{eq:comp2}) only with
   $\overline{W}'$ instead of $\overline{V}'$.

   We now use the cobordisms $\overline{W}''$ and $\overline{V}$. The
   fact that the horizontal isotopy from $\overline{W}'$ to
   $\overline{W}''$ may be assumed constant
   $\pi^{-1}([1,+\infty)\times \R)$ implies the commutativity of the
   triangle below up to multiplication with a term of the form
   $T^{a}$:
   \begin{equation}\label{eq:inclusion-Floer3}\xymatrix@-1pt{
        HF(L,L)\ar[r]^{j}\ar[rd]_{\phi_{V}}& 
        HF(\overline{W}', \overline{V})\ar[d]^{\xi'}\\
        & HF(\overline{W}'',\overline{V})=HF(L,L_{2}) ~.~
      }
   \end{equation}
   Indeed, with the correct choice of perturbations and almost complex
   structure, the Floer complex $CF(\overline{W}'',\overline{V})$ is
   of the form $(CF(L,L_{2})\oplus CF(L,L)[1]\oplus CF(L,L), D)$ where
   the differential $D$ is just the internal differential on both
   $CF(L,L_{2})$ and $CF(L,L)$ and on $CF(L,L)[1]$ (which is
   represented geometrically by the intersection points of
   $\overline{W}''$ and $\overline{V}$ that project on $P$) it has the
   form $D=d_{L}[1]+\phi_{V}-\textnormal{id}_{CF(L,L)}[-1]$ where
   $d_{L}$ is the differential on $CF(L,L)$.  The choice of isotopy
   shows that $j$ corresponds to the inclusion
   $$CF(L,L)\to CF(L,L_{2})\oplus CF(L,L)[1]\oplus CF(L,L)$$
   and this implies the commutativity of
   Diagram~\eqref{eq:inclusion-Floer3} up to multiplication by
   $T^{a}$.
  
   To summarize what was shown till now, we proved that
   Diagram~(\ref{eq:inclusion-Floer}) commutes and that $PSS$ and
   $PSS'$ are isomorphisms. The next remark is that the morphism
   $\phi_{V}$ is a $QH(L;\mathcal{A})$-module morphism.  This an easy
   verification based on our definition of $\phi_{V}$ that we leave as
   exercise.  As $QH(L)$ is a field this means that either $\phi_{V}$
   is null or it is an injection.  Thus, the same is true for $i_{1}$
   and it is easy to see that a similar argument can be applied to the
   morphism $i_{2}$ from Diagram (\ref{eq:two-exact-q}).  The
   exactness of~\eqref{eq:two-exact-q} together with the duality
   between the $i_{j}$'s and the $\eta_{r}$'s implies that one of the
   $i_{j}$'s has to vanish and the other is injective. We will assume
   that $i_{1}$ is injective and that $i_{2}$ vanishes.  From Lemma
   \ref{lem:alg-inclusion} it is immediate to see that injectivity of
   $i_{1}$ with coefficients in $\mathcal{A}$ implies that the
   corresponding morphism $i_{1}^{\La}:QH(L)\to QH(V, L_{1})$ is also
   injective.  Similarly, the vanishing of $i_{2}$ with coefficients
   in $\mathcal{A}$ also implies the vanishing of $i_{2}$ over $\La$.
   To shorten notation we will not indicate the coefficients in the
   notation for these morphisms $i_{1}$, $i_{2}$, etc as long as there
   is no risk of confusion.
   
   The first claim of the Theorem now follows easily. Indeed $i_{1}$
   (now taken over $\La$) factors: $$ QH(L)\to QH(V)\to QH(V,L_{1})$$
   and thus $QH(L)\to QH(V)$ is injective.  The rank
   inequality~\eqref{eq:rk} follows immediately if we can show that for
   $I_{i}=Im (QH(L_{i})\to QH(V))$ and $I_{0}=Im(QH(L)\to QH(V))$,
   we have $I_{1}\oplus I_{0}\subset I_{2}$ and $QH(L_{1})\to QH(V)$ is
   injective.
 
   To do this we go back to the Diagram~\eqref{eq:two-exact-q} and we
   start by noticing that the vanishing of $i_{2}$ implies that
   $k_{1}$ vanishes. This is seen as follows. First, by an argument
   similar to that applied to $i_{1}$ and $i_{2}$ we see that, over
   $\mathcal{A}$, $k_{1}$ is a $QH(L_{1};\mathcal{A})$-module map.
   Thus it suffices to show that $k_{1}([L_{1}])=0$ ($[L_{1}]$ is the
   fundamental class and is the unit in $QH(L_{1};\mathcal{A})$).
   Secondly, by using explicitly the form of the pearl complexes
   associated to a function $f:V\to \R$ adapted to the exit region
   $L\cup L_{1}$ it is easy to see that $i_{2}([L])=k_{1}([L_{1}])$
   and thus $k_{1}([L_{1}])=0$.  This means that $k_{1}$ vanishes over
   $\mathcal{A}$. But this implies that it also vanishes over $\La$.
   Now $k_{1}$ and $s_{1}$ are dual so the vanishing of $k_{1}$
   implies that of $s_{1}$ which means that $l_{1}:QH(L_{1})\to
   QH(V,L)$ is injective. But this implies that $QH(L_{1})\to QH(V)$
   is injective.
 
   We now show that $I_{0},I_{1}\subset I_{2}$.  This follows from the
   exact sequence:
   $$\to QH(L_{2})\to QH(V)\to QH(V,L_{2})\to $$
   combined with the fact that both maps $k_{1}: QH(L_{1})\to QH(V)\to
   QH(V,L_{2})$ and $i_{2}:QH(L)\to QH(V)\to QH(V,L_{2})$ vanish.

   The last step is to show that $I_{0}\cap I_{1}=\{0\}$. This follows
   form the exact sequence
   $$\to QH(L_{1})\to QH(V)\to QH(V,L_{1})\to$$
   together with the fact that the map $i_{1}:QH(L)\to QH(V)\to
   QH(V,L_{1})$ is injective.
\end{proof}


\section{Examples}\label{sec:examples}

In this section we show Theorem~\ref{thm:examples}. The examples
presented here are based on the Lagrangian surgery construction as
described for instance by Polterovich in~\cite{Po:surgery}.  These
examples contrast with the rigidity results contained in the previous
sections, in particular Theorems~\ref{cor:exact-tri-explicit}
and~\ref{thm:quantum_h}.

\subsection{The trace of surgery as Lagrangian cobordism}
\label{subsec:surgery-trace}The purpose here is to show that the trace
of a Lagrangian surgery gives rise to a Lagrangian cobordism.  As we
shall see, this is a bit less obvious than one might first
expect because Lagrangian cobordism is less flexible than Lagrangian
isotopy.

We start with the local picture and fix the following two Lagrangians:
$L_{1}=\R^{n}\subset \C^{n}$ and $L_{2}= i\R^{n}\subset \C^{n}$.

We define a particular curve $H\subset \C$, $H(t)=a(t)+ ib(t)$,
$t\in\R$, with the following properties (see also Figure
\ref{fig:handle}):
\begin{itemize}
  \item[i.] $H$ is smooth.
  \item[ii.] $(a(t), b(t))=(t,0)$ for $t\in(-\infty,-1]$.
  \item[iii.] $(a(t), b(t))=(0, t)$ for $t\in [1,+\infty)$.
  \item[iv.] $a'(t), b'(t)>0$ for $t\in (-1,1)$.
\end{itemize}

\begin{figure}[htbp]
   \begin{center}
      \epsfig{file=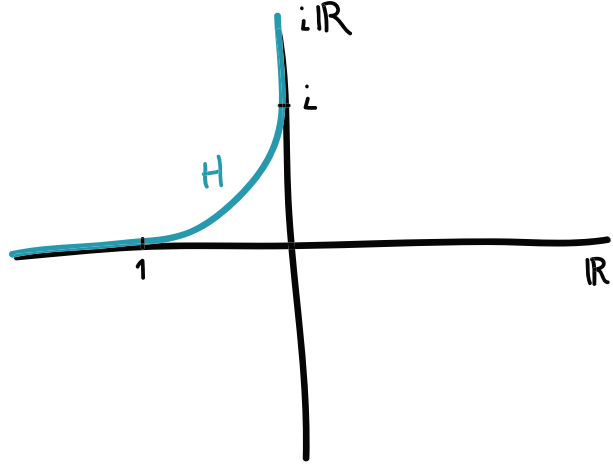, width=0.50\linewidth}
   \end{center}
   \caption{\label{fig:handle} The curve $H\subset \C$.}
\end{figure}
  
Consider $L = H\cdot S^{n-1}\subset \C^{n}$ or more explicitly $$L =
\Bigl\{\bigl((a(t)+ib(t))x_{1}, \ldots, (a(t)+ib(t))x_{n}\bigr) \mid t
\in \mathbb{R}, \, \sum x_{i}^{2}=1\Bigr\}\subset \C^{n}~.~$$
 
\begin{lem}\label{lem:surgery-trace}
   The submanifold $L\subset \C^{n}$ as defined above is Lagrangian
   and there is a Lagrangian cobordism $L \cobto
   (L_{1},L_{2})$.
\end{lem}
 
By a slight abuse of notation (because we omit the handle from the
notation) we will denote $L=L_{1}\# L_{2}$.
  
\begin{proof}
   A straightforward calculation shows that $L \subset \mathbb{C}^n$
   is Lagrangian (see e.g.~\cite{Po:surgery}).
  
   To construct the desired cobordism we now define $$\widehat{H}=
   H\cdot S^{n}\subset \C^{n+1}~.~$$ Or more explicitly
   $$\widehat{H}=\Bigl\{ \bigl( (a(t)+ib(t))x_{1}, \ldots,
   (a(t)+ib(t))x_{n+1} \bigr) \mid t\in \R, \, \sum x_{i}^{2}=1
   \Bigr\}~.~$$ A similar computation as before shows that
   $\widehat{H}$ is also Lagrangian.
 
   We consider the projection \RE{$\pi: \C^{n+1}\to \C$, $\pi(z_{1},\ldots
   z_{n+1})=z_{1}$} and we denote by $\widehat{\pi}$ its restriction
   to $\widehat{H}$:
   $$\widehat{\pi}((a(t)+ib(t))x_{1}, \ldots,
   (a(t)+ib(t))x_{n+1}) = (a(t)+ib(t))x_{1}~.~$$ Define $W =
   \widehat{\pi}^{-1}(S_{+})$ where $S_{+}=\{(x,y)\in \R^{2} \ | \
   y\geq x\}$, see Figure \ref{fig:projection-image}. (As usual, we
   identify $\R^{2}$ with $\C$ under $(x,y)\to x+iy$.) A simple calculation 
   shows that $W$ is a manifold with boundary.
  \begin{figure}[htbp]
     \begin{center}
        \epsfig{file=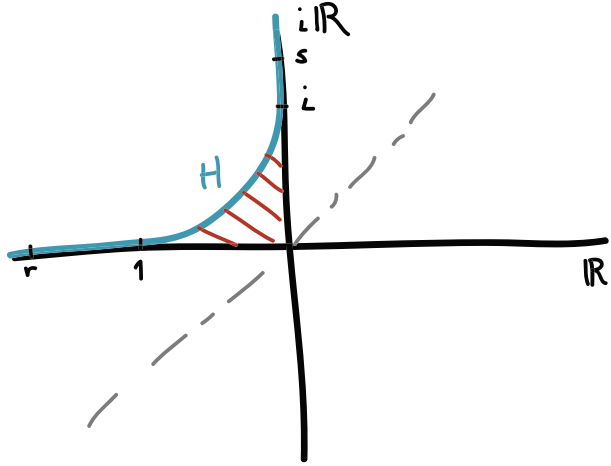, width=0.50\linewidth}
     \end{center}
     \caption{\label{fig:projection-image} The projection of $W$ is
       the red region together with the two semi-axes
       $(-\infty,0]\subset \R$ and $i[0,+\infty)\subset i\R$ and the
       curve $H$.}  
  \end{figure}
  
  Fix some $r < 0$ and notice that if $$\widehat{\pi}((a(t) +
  ib(t))(x_1, \ldots, x_{n+1})) = (r,0)~,~$$ then $b(t)=0$ so that
  $t\leq -1$ and $a(t)=t$. Moreover, $tx_{1}=r$ so that
  $\sum_{i=2}^{n+1} t^{2}x_{i}^2=t^{2}-r^{2}$.  Thus, for $r\leq -1$, we
  have $\widehat{\pi}^{-1}(r,0)=(r,0) \times  L_{1}\subset \C \times \C^{n}$.  
  Similarly, for $s\geq 1$, $\widehat{\pi}^{-1}(0,s)=  (0,s)\times
 L_{2}\subset \C\times \C^{n}$. Also notice that
  $L=\widehat{\pi}^{-1}(0)$.

  We now only look to $W_{0}=W\cap \pi^{-1}([-2,0]\times[0,2])$. It is
  not difficult to see that $W_{0}$ is a manifold with boundary and that 
  $\partial W_{0}=\{(-2,0)\}\times L_{1}\cup \{ (0,2)\}\times L_{2} \cup 
  \{0,0\}\times L$.  We would like to be
  able to say that $W_{0}$ is a cobordism $W_{0}:L\cobto
  (L_{1},L_{2})$. For this however we still need to show that the
  $L$-boundary component of $W_0$ can be continued to be cylindrical.
  We now describe explicitly this adjustment (the argument here is in
  fact quite general).  Let $V_L \subset \C\times \C^{n}$ be the
  Lagrangian given by $V_L = \{(x,y)\in \C : x=-y\}\times L$. 

  Put $L^0 =\{(0,0)\}\times L $. Note that $V_L \cap \pi^{-1}((0,0))
  = \widehat{H} \cap \pi^{-1}((0,0)) = L^0$. Fix a small neighborhood
  $U(L^0) \subset \widehat{H}$ of $L^0 \subset \widehat{H}$ and a
  Darboux-Weinstein neighborhood $\mathcal{N} \subset
  \mathbb{C}^{n+1}$ of $U(L^0)$ and identify symplectically
  $\mathcal{N}$ with a tubular neighborhood of $U(L^0)$ in
  $T^*U(L^0)$. Write $p: \mathcal{N} \to U(L^0)$ for the projection
  corresponding via this identification to the projection in the
  cotangent bundle $T^* U(L^0) \to U(L^0)$.

  Note that at each point of $L^0$, $V_L$ projects 1-1 on the tangent
  space of $\widehat{H}$ (via $p$). Thus reducing $U(L^0)$ if
  necessary we can write $V_L \cap \mathcal{N}$ as the graph of a
  $1$-form $\alpha$ on $U(L^0)$ that vanishes on $L^0$. Since $V_L$ is
  Lagrangian the form $\alpha$ is closed. As $U(L^0)$ can be chosen so
  that it contracts to $L^0$, we have $H^{1}(U(L^0), L^0)=0$ hence
  $\alpha$ is exact.  Let $f:U(L^0)\to \R$ be so that $\alpha=df$.
  Using a partition of unity construct $g:W\cup U(L^0)\to \R$ so that
  it agrees with $f$ on $U(L^0) \setminus W$ and vanishes outside a
  neighborhood of $U(L^0)$. Then the Lagrangian $W'$ obtained by
  isotoping $W_{0}$ by the time-one Hamiltonian diffeomorphism induced
  by $X_{g}$ provides the cobordism desired between $L$ and
  $(L_{1},L_{2})$ - see also Figure \ref{fig:surgery-trace}.
\end{proof}  
 \begin{figure}[htbp]
     \begin{center}
        \epsfig{file=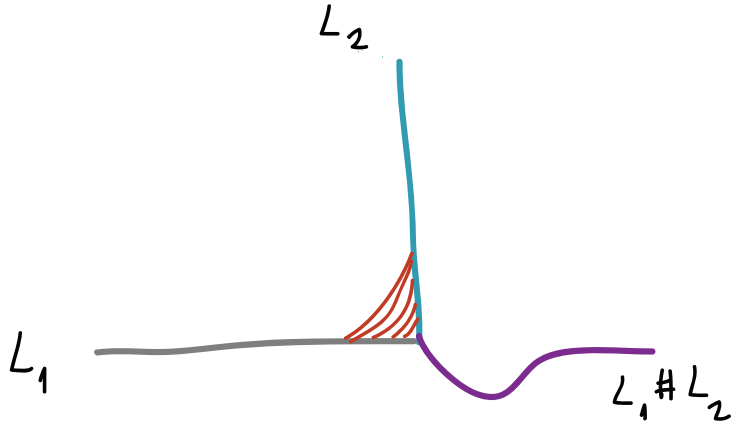, width=0.60\linewidth}
     \end{center}
     \caption{\label{fig:surgery-trace} The trace of the surgery after
       projection on the plane.}
  \end{figure}

\begin{rem}\label{rem:surgery-cob}
   \begin{itemize}
     \item[i.] For further use, we notice that the homotopy type of
      the total space of the cobordism $W':L \cobto L_1 \# L_2$
      constructed above coincides with that of the 
      \RE{topological subspace of $M$ consisting of the union  $L_{1}\cup
      L_{2}$}.
     \item[ii.] The construction described here also provides a
      cobordism $W'':(L_{2},L_{1})\to L_{2}\# L_{1}$ by simply using
      instead of $W$ the region $\widehat{\pi}^{-1}(S_{-})$ with
      $$S_{-}=\{(x,y)\in \R^{2} \ : \ y\leq x\}~.~$$
         \end{itemize}
\end{rem}

Going from the local argument above to a global one is easy. Suppose
that we have two Lagrangians $L'$ and $L''$ that intersect
transversely, possibly in more than a single point. At {\em each}
intersection point we fix symplectic coordinates mapping (locally)
$L'$ to $\R^{n}\subset \C^{n}$ and mapping (again locally) $L''$ to
$i\R^{n}\subset \C^{n}$.  We then apply the construction above at each
of these intersection points.  This produces a new Lagrangian
submanifold $L'\tilde{\#}L''$ as well as a cobordism
$L'\tilde{\#}L''\cobto (L',L'')$ (we use $\tilde{\#}$ in the notation
as $L'\tilde{\#}L''$ is topologically not a connected sum if there are
several intersection points).  The homotopy type of $V$ coincides with
that of the set $L'\cup L'' \subset M$.

\begin{rem} \label{r:multiple-surgery} a. The construction above can
   be used to produce examples of monotone cobordisms of the type $V:L
   \cobto (L_{1},L_{2})$. Indeed assume that $L_{1}$ and $L_{2}$ are
   uniformly monotone and intersect in a single point (for instance
   the longitude and latitude on a torus), then, as a consequence of a
   simple application of the Seifert Van Kampen Theorem, we have that
   $L = L_{1}\# L_{2}$ as well as the cobordism relating it to
   $(L_{1}, L_{2})$ are also monotone with the same constants.  The
   construction can be easily iterated to produce monotone cobordisms
   with arbitrarily many ends.

   b. One can easily generalize the previous construction to a
   configuration of Lagrangian submanifolds $(L_1, \ldots, L_r)$ and
   the total surgery $L$ of the Lagrangians in the configuration. The
   result will be a Lagrangian cobordism $V: L \cobto (L_1, \ldots,
   L_r)$ with one positive end and $r$ negative ends. Of course,
   monotonicity will in general not be preserved in this case.
   However, if the intersection diagram of the configuration is a
   tree, and if $(L_1, \ldots, L_r)$ are uniformly monotone, then \RE{
     the Lagrangian $L$ and the cobordism $V$ will be monotone too}.
   An interesting example is when $(L_1, \ldots, L_r)$ is a
   configuration of Lagrangian spheres corresponding to a simple
   singularity. The relation between singularity theory and Fukaya
   categories has been extensively studied in recent years (see
   e.g.~\cite{Se:book-fukaya-categ}).  Thus the constructions above
   (together with Theorem~\ref{cor:exact-tri-explicit}) suggests that
   the cobordism category is relevant in this study.
\end{rem}

\subsection{Cobordant Lagrangians that are not isotopic}
\label{subsec:cob-lagr-non-iso}

In this subsection we will make use of the constructions described
in~\S\ref{subsec:surgery-trace} to prove Theorem~\ref{thm:examples}
and thus construct an example of monotone cobordant connected
Lagrangians that are not smoothly isotopic. We emphasize that while
the two Lagrangians at the ends of the cobordism are exact, the
minimal Maslov number of the cobordism $V$ between them is $N_{V}=1$.
We will also see that the two Lagrangians in question cannot be
related by a monotone cobordism with minimal Maslov
$\mu_{\textnormal{min}} \geq 2$. A variety of other examples can be
constructed following the same ideas.

We will start our construction in the ambient manifold $M = \C$. We
consider two circles $A=\{z\in \C \ :\ |z+1/2|=1\}$ and $B=\{x\in \C\
: \ |z-1/2|=1\}$. We denote by $D(A)$ and $D(B)$ the two disks bounded
by $A$ and $B$ respectively.  We also consider two smooth curves in
the plane $\C$, $\gamma_{1}:[-1,1]\to \C$ and $\gamma_{2}: [-1,1]\to
\C$ so that - see Figure \ref{fig:plane-surgery1}:
\begin{itemize}
  \item[i.] $\gamma_{1}(t)=t$ for $t\in [-1,-1/2]$
  \item[ii.] $\gamma_{1}(t)=1+(1-t)i$ for $t\in [1/2,1]$
  \item[iii.] $Re(\gamma_{1}(t))$ is strictly increasing for $t\in
   (-1/2,1/2-\epsilon)$. $Im(\gamma_{1}(t))$ is strictly increasing
   for $t\in (-1/2,1/2-\epsilon)$ and strictly decreasing for $t\in
   (1/2-\epsilon, 1/2)$.
  \item[iv.] $\gamma_{2}(t)=-\gamma_{1}(t)$ for all $t\in [-1,1]$.
\end{itemize}

\begin{figure}[htbp]
   \begin{center}
      \epsfig{file=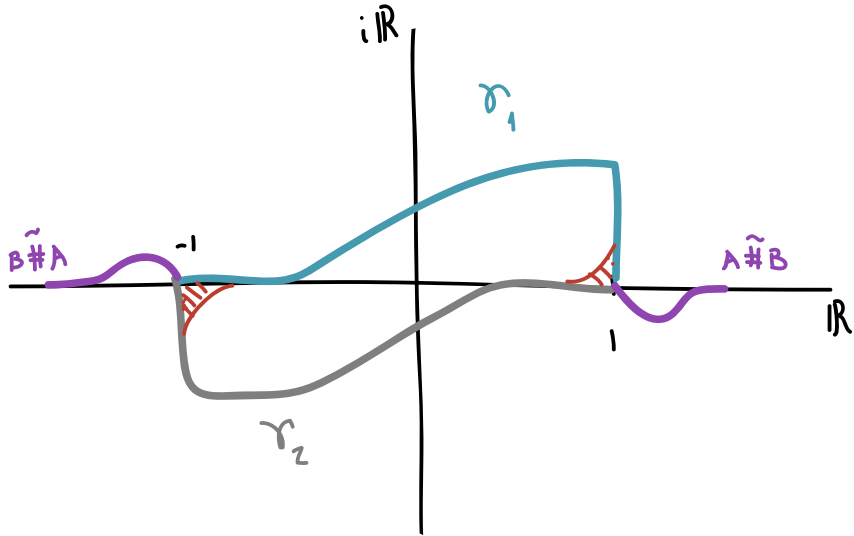, width=0.70\linewidth}
   \end{center}
   \caption{\label{fig:plane-surgery1} The projection of $V$ on $\C$;
     in {\color{red} red} the surgery regions; the curves $\gamma_{1}$
     (in blue) and $\gamma_{2}$ in gray.}
\end{figure}
We now consider the Lagrangians $A'=\gamma_{2}\times A \subset
\mathbb{C} \times M$ and $B'=\gamma_{1}\times B \subset \C\times M$.
By performing surgery - as explained in \S\ref{subsec:surgery-trace} -
at both intersection points $A\cap B$ we can extend the union of the
two Lagrangians $A'\cup B'$ towards the positive end as well as
towards the negative end as in the Figure \ref{fig:plane-surgery1}
thus obtaining a cobordism $V:A\tilde{\#} B\cobto B\tilde{\#} A$.

\begin{figure}[htbp]
   \begin{center}
      \epsfig{file=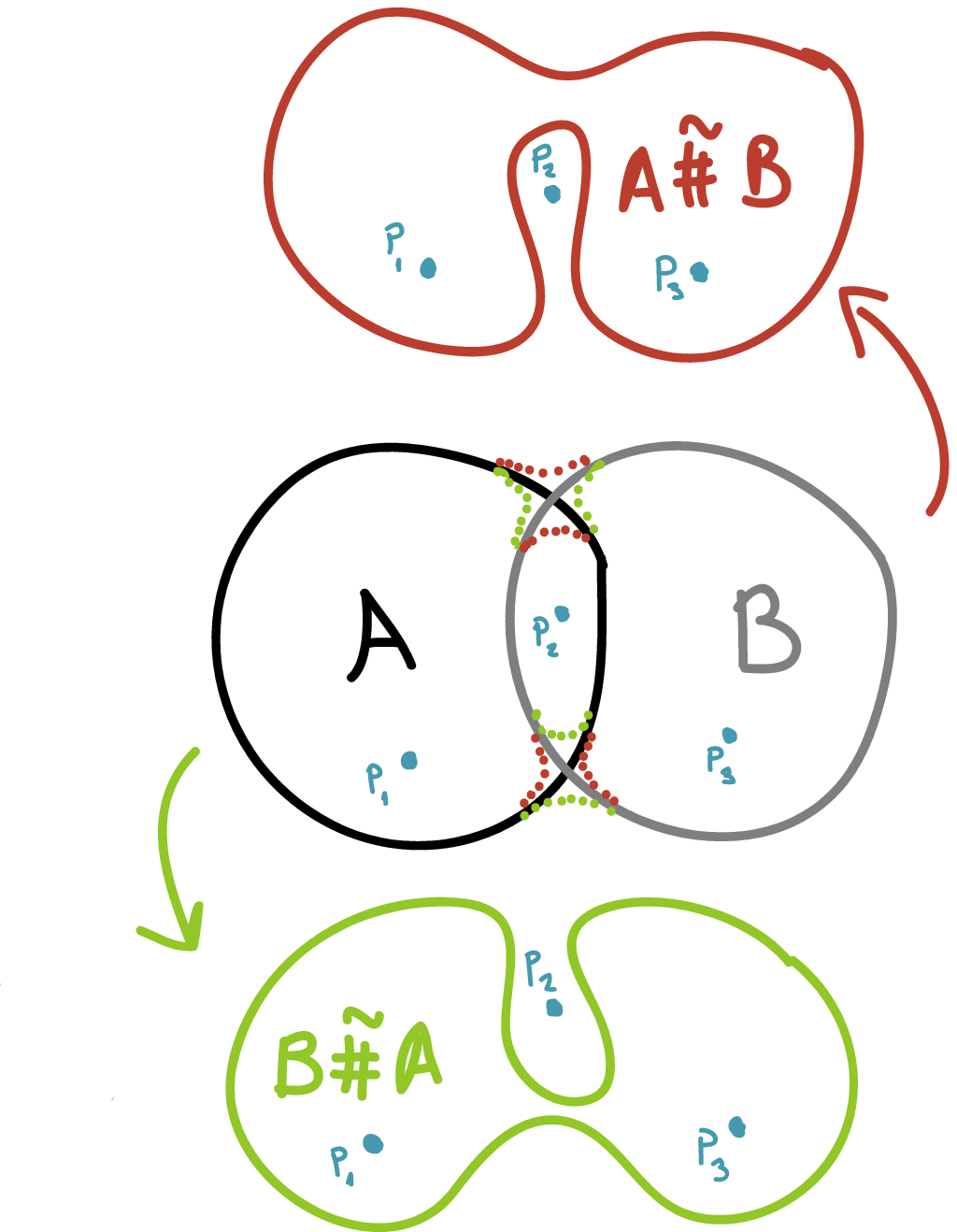, width=0.40 \linewidth}
   \end{center}
   \caption{\label{fig:surgered-circles} The two circles $A$ and $B$
     as well as {\color{red} $A\tilde{\#} B$} and {\color{green}
       $B\tilde{\#} A$}. The three puncture points are indicated in
     blue.}
\end{figure}

Put $L=A\tilde{\#}B$ and $L'=B\tilde{\#}A$.  With our choice of
handles it is easy to see that $L$ and $L'$ look as in
Figure~\ref{fig:surgered-circles}.  Moreover, if the surgeries used in
both intersection points of $A$ and $B$ use the same handle $H$, then
the area inside both circles is precisely $D(A)+D(B)-Area(D(A)\cap
D(B))$ (the two handles can also be picked differently and this can
modify the areas bounded by these two circles, thus producing a -
non-monotone - cobordism relating non-Hamiltonian isotopic, connected
Lagrangians).

It is easy to see that $V$ as constructed before is not orientable.
Moreover, $V$ is also not monotone. However, it is possible to alter
the ambient manifold and make the cobordism monotone in the following
way. Instead of performing all the construction above in $M=\C$ we can
as well do it in $M'=\C\backslash\{P_{1},P_{2},P_{3}\}$ where the
three points $P_{i}$ are such that $P_{1}\in D(A)\backslash D(B)$,
$P_{2}\in D(A)\cap D(B)$ and $P_{3}\in D(B)\backslash D(A)$ as in
Figure~\ref{fig:surgered-circles}.  We will explicitly check
monotonicity below.  We notice for now that in $M'$, $L$ and $L'$ are
not even smoothly isotopic.

To verify that $V$ is monotone in $\C\times M'$ we write $V= V_{+}\cup
V_{-}$, where $V_{+}=V\cap \pi^{-1}([0,+\infty)\times\R)$ and
$V_{-}=V\cap \pi^{-1}((-\infty, 0]\times \R)$. Put
  $\widetilde{M'}_{+} = ([0,+\infty)\times\R)\times M' $,
  $\widetilde{M'}_{-} = ((-\infty, 0]\times\R)\times M'$.  Moreover,
  $V_{+}\cap V_{-}= \{P\}\times A\cup \{Q\}\times B$, where
  $P=\gamma_{2}\cap i\R$ and $Q=\gamma_{1}\cap i\R$. Each of $V_{+}$
  and $V_{-}$ are homotopy equivalent to $A\cup B$.  In particular,
  $H_{2}(\widetilde{M'}_{+},V_{+})=0$ and $H_{2}(\widetilde{M'}_{-},
  V_{-})=0$. This implies that $H_{2}(\C\times M', V)=\Z \oplus \Z$.
  There are two generators for this group each associated to one of
  the intersection points of $A$ and $B$ . Each of them is represented
  by a disk in $\C^{2}$ with boundary on $V$ that is given by a flat
  lift of the planar region bounded by the two curves $\gamma_{1}$ and
  $\gamma_{2}$ and the two planar projections of the handles at the
  ends.  We orient these generators so as to be of positive area and
  we will now verify that they are both of Maslov index $1$.  The
  computation is the same for both generators and we fix just one of
  them, $D_{1}$.  Its boundary is a curve $\gamma: S^{1}\to
  \widetilde{M}=\C\times M$ whose projection onto $\C$ we denote still
  by $\gamma$ (this is, as mentioned before, the union of the two
  curves $\gamma_{1}$ and $\gamma_{2}$ and the two handles at the
  ends).  We look at the tangent space $T_{x}V\subset T\widetilde{M}=
  \C\times \C$ for a point $x=\gamma(t)\in V$.  It is easy to see that
  this tangent space decomposes as:
$$T_{x}V= \R (\dot{\gamma}(t),0) + \R (0,N(t))$$
where $\dot{\gamma}$ is the tangent vector to $\gamma$ and $N(t)\in
\C$.  We need to describe $N(t)$ more explicitly.  For that we pick a
parametrization for $\gamma$ - as described in
Figure~\ref{fig:plane-surgery1} - so that $\gamma: [0,8]\to \C$,
$\gamma(0)=P=\gamma(8)$, $\gamma(4)=Q$, $\gamma(t)=H(t-2)$ for $1\leq
t\leq 3$, $\gamma(t)=-H(t-6)$ for $5\leq t\leq 7$.  With this
parametrization we can write $N(t)=H(t)\in\C$ for $1\leq t\leq 3$;
$N(t)= i H(t)\in \C$ for $5\leq t\leq 7$; $N(t)$ is constant and
tangent to $A$ in the intervals $[0,1]$ and $[7,8]$ (notice however
that $N(8)\not= N(0)$ because $\gamma$ is orientation reversing) and
$N(t)$ is constant and tangent to $B$ in the interval $[3,5]$.  The
horizontal loop $\gamma$ is of Maslov index $2$ and from the formula
above it follows that the contribution of $N(t)$ is $-1$. Thus the
total Maslov index of $D_{1}$ is $1$.

By choosing the handles used for the intersection points appropriately
we may arrange that these two generators have equal areas. Thus $V$ is
monotone of minimal Maslov number $N_V = 1$.

It remains to show that $L$ and $L'$ are not cobordant via a monotone
cobordism with $\mu_{\textnormal{min}} \geq 2$.\footnote{This argument
  was suggested to us by the referee whom we would like to thank.}
Consider a positive ray $R^{+}$ starting at the point $P_{2}$ and
having direction $i\R_+$.  While in the statements in our paper we
always assume that the Lagrangians involved are compact, it is easy to
see that all the techniques involved also apply to $R^{+}$
(alternatively the whole picture can be compactified - see
Remark~\ref{rem:compact-surg}).  Thus the Floer homologies $HF(R^{+},
L)$ and $HF(R^{+}, L')$ are both defined. It is easy to see that they
are not isomorphic as the first does not vanish (it is generated by
the two intersection points between $R^{+}$ and $L$) and the second
vanishes (as $R^{+}\cap L'=0$). However, if $L$ and $L'$ would be
monotone cobordant via a cobordism with $\mu_{\textnormal{min}} \geq
2$, then by Theorem~\ref{cor:exact-tri-explicit} (in the case $k=1$)
these two homologies should be isomorphic.

\begin{rem}\label{rem:compact-surg} 
   It is possible to compactify $M'$ to a closed surface (of high
   genus), while still keeping $V$ monotone. This can be done by
   enlarging the punctures around the points $P_{i}$ and adding
   appropriate handles. The ray $R^+$ will be compactified in this
   picture into a monotone circle.
\end{rem}


\section{Lagrangian cobordism as a category}\label{sec:category}

The aim of this section is to explain how cobordism naturally
organizes the Lagrangian submanifolds of a fixed symplectic manifold
$(M,\omega)$ in a category and to describe a functor relating this
cobordism category to the derived Fukaya category. We will then
interpret Theorems~\ref{cor:exact-tri-explicit},~\ref{thm:quantum_h}
and~\ref{t:2-end-split} from this perspective.  In particular, a non
elementary cobordism $V:L\cobto (L_{1},\ldots L_{k})$ is viewed as a
``splitting'' of $L$ into the ``pieces'' $L_{1},\ldots, L_{k}$.

As mentioned in the introduction an alternative categorical point of
view on Lagrangian cobordism has been independently introduced by
Nadler and Tanaka in \cite{Na-Ta}.

The data is organized in the
following diagram
\begin{equation}\label{eq:main-diag}\xymatrix@-2pt{
     \mathcal{C}ob^{d}_{0}(M)
     \ar[rr]^{\mathcal{F}}\ar[rd]_{\mathcal{\widetilde{\mathcal{F}}}}& & 
     \Sigma D\mathcal{F}uk^{d}(M)\\
     & T^{S}D\fuk^{d}(M)\ar[ru]_{\mathcal{P}} &   }
\end{equation} that will be explained below. 
The proof of the fact that $\widetilde{\mathcal{F}}$ and $\mathcal{F}$
above are functors is postponed to a forthcoming paper.

In the left corner of this Diagram \ref{eq:main-diag} is the category
$Cob^{d}_{0}(M)$ - the cobordism category of $M$ - formally described
in \S\ref{subsec:category-cob}. This is a geometric category with
objects families of Lagrangians $(L_{1},\ldots, L_{k})$, $L_{i}\in
\mathcal{L}^{\ast}_{d}(M)$ (where $\mathcal{L}^{\ast}_{d}(M)$ is a
class of Lagrangians that is also introduced in
\S\ref{subsec:category-cob}).  The morphisms in this category
correspond to (unions) of horizontal isotopy classes of cobordisms
with a single positive end but possibly with many negative ones.

In the right corner in Diagram~\eqref{eq:main-diag},
$\mathcal{F}uk^{d}(M)$, stands for the Fukaya category of $M$ with
objects the Lagrangians in $\mathcal{L}^{\ast}_{d}(M)$.  The Floer
constructions involved in defining the morphisms (and higher
operations) in this $A_{\infty}$-category are with $\Z_{2}$ replacing
$\mathcal{A}$ (as explained in Remark \ref{rem:Floer-restr}; see also
Remark \ref{rem:z2-necessity} on why this change is required).
$D\mathcal{F}uk^{d}(M)$ stands for the resulting derived Fukaya
category of $M$.  The category $D\mathcal{F}uk^{d}(M)$ is triangulated
and $\Sigma D\fuk^{d}(M)$ is the stabilization of $D\fuk^{d}(M)$ in
the sense that the morphisms of $D\fuk^{d}(M)$ are enriched by those
morphisms that shift ``degree'' (see~\S\ref{subsec:cones}).

\begin{rem}\label{rem:split}
   In the construction of $D\fuk^{d}(M)$ we do not complete with
   respect to idempotents (or split factors).
\end{rem}

The category $T^{S}D\fuk^{d}(M)$ is obtained from the category
$D\fuk^{d}(M)$ by a general construction (apparently new) that
associates to any triangulated category $\mathcal{C}$ a new category
$T^{S}\mathcal{C}$ - the category of (stable) triangular (or cone)
resolutions over $\mathcal{C}$. The morphisms sets $\hom(x, -)$ in
this category parametrize the ways in which $x$ can be resolved by
iterated exact triangles (or cone attachments).  We present this
construction below in \S\ref{subsec:cones}. There is a canonical
projection functor $\mathcal{P}:T^{S}\mathcal{C}\to
\Sigma\mathcal{C}$.

In view of the construction of $T^{S}(-)$ the objects in the category
$T^{S}D\fuk^{d}(M)$ are also families $(L_{1},\ldots, L_{k})$,
$L_{i}\in \mathcal{L}^{\ast}_{d}(M)$ and, in fact, the functor
$\widetilde{\mathcal{F}}$ is the identity on objects.  Geometrically,
the existence of $\widetilde{\mathcal{F}}$ is of interest because it
associates to each morphism in $\mathcal{C}ob^{d}_{0}(M)$ and thus, to
each cobordism $V:L\cobto (L_{1},\ldots, L_{k})$, an iterated
decomposition of $L$ by exact triangles in $D\fuk^{d}(M)$ in terms of
the $L_{i}$'s.  In particular, one can deduce a variety of exact
sequences relating the homologies of the ends as well as the higher
structures. By \S\ref{subsec:surgery-trace} this applies, in
particular, to surgery.  This correspondence $$cobordism\
\leftrightarrow \ triangular\ decomposition $$ is reminiscent of the
statement in Theorem \ref{cor:exact-tri-explicit}.  Indeed, as we will
see in \S\ref{subsec:together} where we discuss the relations between
this categorical point of view and our earlier results in the paper,
this theorem is the basic stepping stone for the construction of
$\widetilde{\mathcal{F}}$.

\subsection{The category $\mathcal{C}ob^{d}_{0}(M)$}
\label{subsec:category-cob}
The purpose of this subsection is to set up Lagrangian cobordism as a
category. We first introduce an auxiliary category
$\widetilde{{\mathcal{C}ob}^d}(M)$, $d \in K$. Its objects are
families $(L_{1}, L_{2},\ldots, L_{r})$ with $r \geq 1$, $L_{i}\in
\mathcal{L}_{d}(M)$. (Recall that $\mathcal{L}_d(M)$ stands for
the class of uniformly monotone Lagrangians $L$ with $d_L = d$, and
when $d \neq 0$ with the same monotonicity constant $\rho$ which is
omitted from the notation.)

To describe the morphisms in this category we proceed in two steps.
First, for any two horizontal isotopy classes of cobordisms $[V]$ and
$[U]$ with $V:(L'_{j}) \cobto (L_{i})$ (as in
Definition~\ref{def:Lcobordism}) and $U:(K'_{s})\cobto (K_{r})$ we
define the sum $[V]+[U]$ to be the horizontal isotopy class of a  cobordism 
$W:(L'_{j})+(K'_{s})\cobto (L_{i})+(K_{r})$ so that $W=V \coprod
\widetilde{U}$ with $\widetilde{U}:(K'_{s})\cobto (K_{r})$ a cobordism
horizontally isotopic to $U$ and so that $\widetilde{U}$ is disjoint
from $V$ (to insure embeddedness we can not simply take $V$ and $U$ in
the disjoint union.) Notice that the sum $[V]+[U]$ is not commutative.

The morphisms in $\widetilde{{\mathcal{C}ob}^d}(M)$ are now defined as
follows. A morphism $$[V] \in \mor \bigl( (L'_{j})_{1 \leq j \leq S},
(L_{i})_{1\leq i\leq T} \bigr)$$ is a horizontal isotopy class that
can be written as a sum $[V]=[V_{1}] + \cdots + [V_{S}]$ with each
$V_{j}\in \mathcal{L}_{d}(\C\times M)$ a cobordism from the Lagrangian
family formed by the {\em single} Lagrangian $L'_{j}$ and a subfamily
$(L_{r(j)}, \ldots, L_{r(j)+s(j)})$ of the $(L_{i})$'s, and so that
$r(j)+s(j)+1=r(j+1)$.  In other words, $V$ decomposes as a union of
$V_{i}$'s each with a single positive end but with possibly many
negative ones. We will often denote such a morphism by $V:
(L'_{j})\longrightarrow (L_{i})$.

The composition of morphisms is induced by concatenation followed by a
rescaling to reduce the ``length'' of the cobordism to the interval
$[0,1]$. It is an easy exercise to see that this is well defined
precisely because our morphisms are (horizontal) isotopy classes of
cobordisms and because morphisms are represented by sums of cobordisms
with a single positive end - this is crucial to preserve monotonicity.

We will consider here the void set as a Lagrangian of arbitrary
dimension. We now intend to factor both the objects and the morphisms
in this category by equivalence relations that will transform this
category in a strict monoidal one. For the objects the equivalence
relation is induced by the relations
\begin{equation}
(L, \emptyset) \sim (\emptyset,L) \sim (L).
\label{eq:eq-objects}
\end{equation}  
 
At the level of the morphisms a bit more care is needed. For each $L
\in \mathcal{L}_{d}(M)$ we will define two particular cobordisms
$\Phi_{L}:(\emptyset, L)\cobto (L,\emptyset)$ and $\Psi_{L}:
(L,\emptyset)\cobto (\emptyset,L)$ as follows. Let $\gamma: [0,1] \to
[0,1]$ be an increasing, surjective smooth function, strictly
increasing on $(\epsilon,1-\epsilon)$ and with $\gamma'(t)=0$ for
$t\in [0,\epsilon]\cup [1-\epsilon, 1]$. We now let
$\Phi(L)=graph(\gamma)\times L$ and $\Psi(L)=graph(1-\gamma)\times L$.
The equivalence relation for morphisms is now induced by the following
two identifications:
\begin{enumerate}[(Eq 1)]
  \item For every cobordism $V$ we identify $V+\emptyset \sim
   \emptyset+V \sim V$, where $\emptyset$ is the void cobordism
   between two void Lagrangians. \label{i:eq-1}
  \item If $V:L \longrightarrow (L_{1},..., L_{i}, \emptyset,
   L_{i+2},\ldots, L_{k})$, then we identify $V \sim V' \sim V''$,
   where $V'=\Phi_{L_{i+2}} \circ V$, $V''=\Psi_{L_{i}}\circ V.$
   \label{i:eq-2}
\end{enumerate}

\begin{figure}[htbp]
   \begin{center}
      \epsfig{file=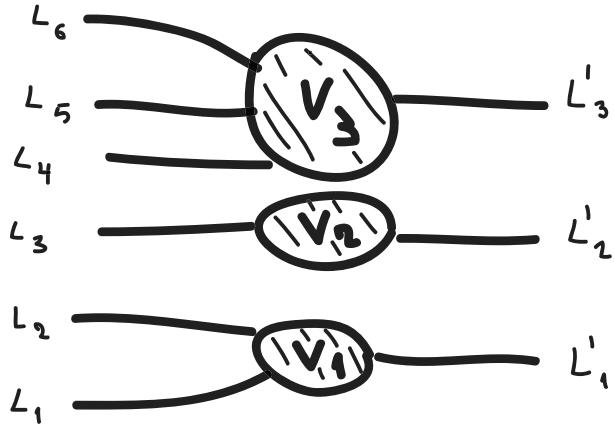, width=0.5\linewidth}
   \end{center}
   \caption{A morphism $V:(L'_{1},L'_{2},L'_{3}) \longrightarrow
     (L_{1}, \ldots, L_{6})$, $V=V_{1}+V_{2}+V_{3}$, projected to
     $\mathbb{R}^2$.}
\end{figure}

We now construct the category $\mathcal{C}ob^{d}(M)$.  First we
consider the full subcategory $\mathcal{S}\subset
\widetilde{\cob}^{d}(M)$ obtained by restricting the objects only to
those families $(L_{1}, \ldots, L_{k})$ with $L_{i}$ non-narrow for
all $1\leq i\leq k$ (this is preferable because our functors
ultimately make use of Floer homology and we need the quantum homology
of each $L_{i}$ to be non trivial).  Then $\mathcal{C}ob^{d}(M)$ is
obtained by the quotient of the objects of $\mathcal{S}$ by the
equivalence relation in~\eqref{eq:eq-objects} and the quotient of the
morphisms of $\mathcal{S}$ by the equivalence relation
in~(Eq~\ref{i:eq-1}), (Eq~\ref{i:eq-2}).

This category is called the {\em Lagrangian cobordism category} of
$M$. As mentioned before, it is a strict monoidal category.  To
recapitulate, its objects are ordered families of Lagrangians $\in
\mathcal{L}_{d}(M)$ and its morphisms:
$$[V]: (L'_{j}) \longrightarrow (L_{i})$$
can be represented by cobordisms $V\in\mathcal{L}_{d}(\R^{2}\times M)$ 
so that all $L_{i}$'s are
non-void and all $L'_{j}$'s are non-void except if there is just a
single $L'_{j}$ which can be void or there is just a single $L_{i}$
which can be void.  Moreover, $V$ can be written as a disjoint union
of cobordisms each with a single positive end.

It turns out that, for the the functorial picture in
Diagram~\eqref{eq:main-diag} to hold, an additional assumption is
required on all the Lagrangians in our constructions, in addition to
the monotonicity conditions discussed in~\S\ref{sb:monotonicity}.
Every Lagrangian $L$ is required to satisfy
\begin{equation}\label{eq:Hlgy-vanishes}
   \textnormal{image\,} \bigl(\pi_{1}(L)\stackrel{i_{\ast}}{\longrightarrow} 
   \pi_{1}(M) \bigr) 
   \ \ \mathrm{is\ torsion},
\end{equation}
where $i_{\ast}$ is induced by the inclusion $L\subset M$.  An
analogous constraint is imposed also to the Lagrangian cobordisms
involved.
\begin{rem}\label{rem:Floer-restr} 
   Assuming the requirement (\ref{eq:Hlgy-vanishes}), an observation
   due to Oh \cite{Oh:HF1} shows that all Floer complexes considered
   earlier in the paper are defined (at the chain level) with
   coefficients in the ``polynomial'' ring $\mathcal{A}^{0}=\Bigl\{
   \sum_{k=0}^{n} a_k T^{\lambda_k} \mid a_k \in K, \; n\in \Z
   \Bigr\}$ (i.e. those elements in $\mathcal{A}$ formed by {\em
     finite} sums). There is an obvious ring map $\mathcal{A}^{0}\to
   \Z_{2}$ obtained by sending $T\to 1$ and this allows to change the
   coefficients in all the structures described by specializing to
   $T=1$.  Clearly, all the results in this paper that have been
   established over $\mathcal{A}$ remain valid when working over
   $\Z_{2}$ using this change of coefficients, assuming of course that
   condition (\ref{eq:Hlgy-vanishes}) is satisfied by all involved
   Lagrangians.
\end{rem}

We denote by $\mathcal{L}^{\ast}_{d}(M)$ the Lagrangians in
$\mathcal{L}_{d}(M)$ that are non-narrow and additionally satisfy
(\ref{eq:Hlgy-vanishes}). There is a subcategory of
$\mathcal{C}ob^{d}(M)$, that will be denoted by
$\mathcal{C}ob^{d}_{0}(M)$, whose objects consist of families of
Lagrangians each one belonging to $\mathcal{L}^{\ast}_{d}(M)$ and
whose morphisms are represented by Lagrangian cobordisms $V$
satisfying the analogous condition to~\eqref{eq:Hlgy-vanishes}, but in
$\mathbb{R}^2 \times M$. This is again a strict monoidal category.

\subsection{Cone decompositions over a triangulated category}
\label{subsec:cones}

In this subsection we will discuss a construction valid in any
triangulated category. The purpose of the construction is to
parametrize the various ways to decompose an object by iterated exact
triangles. 
  
Let $\mathcal{C}$ be a triangulated category. We
recall~\cite{Weibel:book-hom-alg} that this is an additive category
together with a translation automorphism $T:\mathcal{C} \to
\mathcal{C}$ and a class of triangles called {\em exact triangles}
$$X\stackrel{u}{\longrightarrow}Y
\stackrel{v}{\longrightarrow}Z\stackrel{w}{\longrightarrow} TX$$ that
satisfy a number of axioms due to Verdier and to Puppe (see
e.g.~\cite{Weibel:book-hom-alg}).
  
A cone decomposition of length $k$ of an object $A\in \mathcal{C}$ is
a sequence of exact triangles:
$$T^{-1}X_{i}\stackrel{u_{i}}{\longrightarrow}Y_{i}
\stackrel{v_{i}}{\longrightarrow}Y_{i+1}\stackrel{w_{i}}{\longrightarrow}
X_{i}$$ with $1\leq i\leq k$, $Y_{k+1}=A$, $Y_{1}=0$. (Note that $Y_2
= X_1$.) Thus $A$ is obtained in $k$ steps from $Y_{1}=0$.  To such a
cone decomposition we associate the family $l(A)=(X_{1}, X_{2},\dots ,
X_{k})$ and we call it the {\em linearization} of the cone
decomposition. This definition is an abstractization of the familiar
iterated cone construction in case $\mathcal{C}$ is the homotopy
category of chain complexes. In that case $T$ is the shift functor $TX
= X[-1]$ and the cone decomposition simply means that each chain
complex $Y_{i+1}$ is obtained from $Y_i$ as the mapping cone of a
morphism coming from some chain complex, in other words $Y_{i+1} =
\textnormal{cone}(X_{i}[1] \stackrel{u_i}{\longrightarrow} Y_i)$ for
every $i$, and $Y_1=0$, $Y_{k+1}=A$.

We will now define a category $T^{S} \mathcal{C}$. The construction of
this category starts with the {\em stabilization category} of
$\mathcal{C}$, $\Sigma \mathcal{C}$: $\Sigma\mathcal{C}$ has the same
objects as $\mathcal{C}$ and the morphisms in $\Sigma \mathcal{C}$
from $a$ to $b\in \mathcal{O}b(\mathcal{C})$ are morphisms in
$\mathcal{C}$ of the form $a\to T^{s} b$ for some integer $s$.  Next,
the free monoidal isomorphism category $F ^{\ast}\Sigma\mathcal{C}$
over $\Sigma\mathcal{C}$ has as objects finite families
$(x_{1},\ldots, x_{k})$ where the $x_{i}$'s are objects in
$\mathcal{C}$. The monoidal addition, denoted by $+$, is
concatenation. The morphisms are corresponding families of {\em
  isomorphisms} in $\Sigma\mathcal{C}$ (thus this category differs
from the free monoidal category over $\Sigma\mathcal{C}$ because that
category has as morphisms families of morphisms and not only
isomorphisms).

The category $T^{S}\mathcal{C}$, called the {\em category of (stable)
  triangle (or cone) resolutions over} $\mathcal{C}$ is obtained from
$F^{\ast} \Sigma\mathcal{C}$ by enriching the morphisms with the
elements constructed as follows. Given $x \in
\mathcal{O}b(\mathcal{C})$ and $(y_1, \ldots, y_q) \in
\mathcal{O}b(F^*\Sigma \mathcal{C})$ a morphism $\Psi: x
\longrightarrow (y_1, \ldots, y_q)$ is a triple $(\phi, a, \eta)$,
where $a \in \mathcal{O}b(\mathcal{C})$, $\phi: x \to T^s a$ is an
isomorphism for some index $s$ and $\eta$ is \RE{an equivalence class
  - in an obvious way - of a} cone decomposition of the object $a$
with linearization $(T^{s_1}y_1, T^{s_{2}}y_{2},\ldots,
T^{s_{q-1}}y_{q-1}, y_q)$ for some family of indices $s_1, \ldots,
s_{q-1}$. Below we will also sometimes use a shift index $s_{q}$
attached to the last element $y_{q}$ with the understanding that
$s_{q}=0$.  Thus, not only $a$ admits a cone decomposition of length
$q$ but such \RE{an equivalence class of decompositions} is part of
the data defining the morphism $\Psi$.

We now define the morphisms between two general objects in
$\mathcal{O}b(F^{\ast}\Sigma\mathcal{C})$.  A morphism
$$\Phi\in\mor_{T^{S}\mathcal{C}}((x_{1},\ldots x_{m}),
(y_{1},\ldots, y_{n}))$$ is a sum $\Phi =\Psi_{1}+ \cdots + \Psi_{m}$
where $\Psi_{j}\in \mor_{T^{S}\mathcal{C}}(x_{j},
(y_{\alpha(j)},\ldots, y_{\alpha(j)+\nu(j)}))$, and $\alpha(1)=1$,
$\alpha(j+1) = \alpha(j) + \nu(j) + 1$, $\alpha(m) + \nu(m) = n$.

The composition of the morphisms in $T^{S}\mathcal{C}$ is not quite
obvious (it uses the axioms of a triangulated category; it is
described explicitly in \cite{Bi-Co:cob2}).

There is a projection
functor
\begin{equation} \label{eq:proj} \mathcal{P}:T^{S}\mathcal{C}
   \longrightarrow \Sigma\mathcal{C}
\end{equation}
that is defined by $\mathcal{P}(x_{1},\ldots x_{k})=x_{k}$ and whose
value on morphisms is induced by associating to
$\Phi\in\mor_{T^{S}\mathcal{C}}(x, (x_{1},\ldots, x_{k}))$,
$\Phi=(\phi,a,\eta)$, the composition:
$$\mathcal{P}(\Phi): x\stackrel{\phi}{\longrightarrow} T^{s}a 
\stackrel{p}{\longrightarrow} T^{s}x_{k}$$ with $p:a \to x_{k}$
defined by the last exact triangle in the cone decomposition $\eta$ of
$a$,
$$ T^{-1}x_{k}\longrightarrow a_{k}
\longrightarrow a\stackrel{p}{\longrightarrow} x_{k}~.~$$ 

\subsection{Putting things together}\label{subsec:together}
With the definitions above we can now describe the functor
$\widetilde{\mathcal{F}}$.

The construction of $\widetilde{\mathcal{F}}$ is very simple at the
level of objects:
$$\mathcal{O}b (\mathcal{C}ob^{d}_{0}(M)) \ni (L_{1},\ldots, L_{k}) 
\stackrel{\widetilde{\mathcal{F}}}{\longmapsto} (L_{1},\ldots,
L_{k})\in \mathcal{O}b (T^{S}D\fuk^{d}(M))~.~$$

To describe the functor $\widetilde{\mathcal{F}}$ on morphisms we
first mention that this will be by definition a monoidal functor so
that it is enough to describe $\widetilde{\mathcal{F}}(\Phi)$ where
$$\Phi\in \mor_{\mathcal{C}ob^{d}_{0}(M)}(L, (L_{1},\ldots, L_{k}))~.~$$ Let
$$V:L \cobto (L_{1},\ldots, L_{k}) \ \textrm{with}\ \
[V]=\Phi~.~$$

The triangulated structure of $D\fuk^{d}(M)$ is induced from an
$A_{\infty}$- triangulated completion $\mathcal{F}uk^{d}(M)^{\wedge}$
of $\fuk^{d}(M)$. As explained in~\cite{Se:book-fukaya-categ} there
are multiple such completions but all are equivalent for our purposes.
The precise version that we use here (see Remark 3.21
in~\cite{Se:book-fukaya-categ}) is obtained by first using the Yoneda
embedding to view $\mathcal{F}uk^{d}(M)$ as a functor category over
itself with values into chain complexes and then making use of the
usual cone construction at the level of chain complexes to build a
triangulated closure of the image of the embedding.  The category
$D\fuk^{d}(M)$, has the same objects as
$\mathcal{F}uk^{d}(M)^{\wedge}$ but its morphisms are obtained by
applying the homology functor to the morphisms in
$\mathcal{F}uk^{d}(M)^{\wedge}$.

By rendering explicit the definitions of the various categories
involved we see that to construct $\widetilde{\mathcal{F}}(\Phi)$ we
need to associate to each $N\in \mathcal{L}_{d}^{\ast}(M)$ a sequence
of chain complexes $Z^{N}_{i}$, $1 \leq i \leq k+1$, with $Z_1^N = 0$,
and chain morphisms $u_i: CF(N,L_i) \longrightarrow Z_i^N$ so that
\begin{equation} \label{eq:FL-exact-tr} Z_{i+1}^N = \textnormal{cone}
   (CF(N, L_i) \stackrel{u_i}{\longrightarrow} Z^N_{i}), \;\;
   \forall\, 1 \leq i \leq k,
\end{equation}
as well as a chain homotopy equivalence $\phi_{V}^{N}:CF(N,L)
\longrightarrow Z^{N}_{k+1}$ (we again neglect the grading here).
Moreover, this association is supposed to be functorial in $N$, there
should be a compatibility with all the higher structures of an
$A_{\infty}$-category as well as with the composition of cobordisms.

While these functoriality verifications are postponed to a later
publication we remark that the existence of the exact sequences
in~\eqref{eq:FL-exact-tr} is precisely the statement of Theorem
\ref{cor:exact-tri-explicit} !
\begin{rem}\label{rem:z2-necessity} 
   Working over $\Z_{2}$ instead of $\mathcal{A}$ (and thus the
   requirement~\eqref{eq:Hlgy-vanishes}) is crucial here because the
   maps $u_{i}$, $\phi^{N}_{V}$ above should not depend on any
   additional choices. This is true over $\Z_{2}$ but only true up to
   multiplication with some $T^{a}\in \mathcal{A}$ if working over
   $\mathcal{A}$.
\end{rem}
Finally, from the point of view described here, the
Theorems~\ref{thm:quantum_h} and~\ref{t:2-end-split} can be viewed as
exhibiting algebraic obstructions to the existence of morphisms in
$\mathcal{C}ob^{d}(M)$.


\bibliography{bibliography}

%
%
%

\end{document}